\documentclass[11pt,reqno,a4paper]{amsart}

\usepackage{amsaddr}

\usepackage[utf8]{inputenc}
\usepackage[T1]{fontenc}
\usepackage{lmodern}
\usepackage[ngerman,english]{babel}
\usepackage{microtype}
\usepackage{color}

\usepackage[pdftex]{graphicx}
\usepackage{latexsym}
\usepackage{amsmath,amssymb,amsthm}
\usepackage{bbm}
\usepackage{mathtools}
\usepackage {enumitem}
\usepackage{tikz,pgfplots,subcaption}
\usepackage{caption}
\pgfplotsset{compat=1.16}
\usepackage[backend=biber,maxbibnames=99]{biblatex}  
\usepackage{csquotes}

\renewbibmacro*{publisher+location+date}{%
	\printlist{publisher}%
	\iflistundef{location}
	{\setunit*{\addcomma\space}}
	{\setunit*{\addcomma\space}}%
	\printlist{location}%
	\setunit*{\addcomma\space}%
	\usebibmacro{date}%
	\newunit}

\setlist[enumerate,1]{noitemsep}
\newtheorem{Satz}{Satz}[section]
\newtheorem{Proposition}[Satz]{Proposition} 
\newtheorem{Definition}[Satz]{Definition}     
\newtheorem{Lemma}[Satz]{Lemma}	
\newtheorem{Theorem}[Satz]{Theorem}	
\newtheorem{Corollary}[Satz]{Corollary}	
\theoremstyle{definition}

\newtheorem{Remark}[Satz]{Remark}
\newtheorem{Assumption}[Satz]{Assumption}

\numberwithin{equation}{section} 

\newcommand{\T}{\mathbb{T}} 
\newcommand{\R}{\mathbb{R}} 
\newcommand{\Z}{\mathbb{Z}} 
\newcommand{\N}{\mathbb{N}} 
\newcommand{\eps}{\varepsilon}

\newcommand{\dd}{\, \mathrm{d}}
\newcommand{\iu}{\mathrm{i}}

\newcommand{\supp}{\operatorname{supp}}
\newcommand{\sinc}{\operatorname{sinc}}

\makeatletter
\newcommand{\opnorm}{\@ifstar\@opnorms\@opnorm}
\newcommand{\@opnorms}[1]{%
	\left|\mkern-1.5mu\left|\mkern-1.5mu\left|
	#1
	\right|\mkern-1.5mu\right|\mkern-1.5mu\right|
}
\newcommand{\@opnorm}[2][]{%
	\mathopen{#1|\mkern-1.5mu#1|\mkern-1.5mu#1|}
	#2
	\mathclose{#1|\mkern-1.5mu#1|\mkern-1.5mu#1|}
}
\makeatother

\usepackage[open]{bookmark}

\usepackage{hyperref}

\allowdisplaybreaks

\title[Strang splitting for the 3D semilinear wave equation]{Error analysis of the Strang splitting for the 3D semilinear wave equation with finite-energy data}
\author{Maximilian Ruff}
\address{Karlsruhe Institute of Technology, Department of Mathematics, Englerstraße 2, 76131 Karlsruhe, Germany}
\email{maximilian.ruff@kit.edu}
\thanks{Funded by the Deutsche Forschungsgemeinschaft (DFG, German Research Foundation) – Project-ID 258734477 – SFB 1173.}
\thanks{I thank Benjamin Dörich and Roland Schnaubelt for fruitful discussions. I also thank the referees for very helpful comments, which led to improvements of the presentation, and for pointing out the alternative approach mentioned in Remark \ref{RemIBPSpace}.} 

\begin{document}
	\selectlanguage{english}
	\pagestyle{headings}
	\begin{abstract}
		We study a variant of the Strang splitting for the time integration of the semilinear wave equation under the finite-energy condition on the torus $\T^3$. In the case of a cubic nonlinearity, we show almost second-order convergence in $L^2$ and almost first-order convergence in $H^1$. 
		If the nonlinearity has a quartic form instead, we show {analogous convergence results, where the order is reduced by 1/2 in both cases}. To our knowledge these are the best convergence results available for the 3D cubic and quartic wave equations under the finite-energy condition. 
		Our approach relies on continuous- and discrete-time Strichartz estimates. We also make use of the integration and summation by parts formulas to exploit cancellations in the error terms. Moreover, error bounds for a full discretization using the Fourier pseudo-spectral method in space are given. Finally, we discuss a numerical example indicating the sharpness of our theoretical results.
	\end{abstract}

	\maketitle
	
	\section{Introduction} 

We study the time integration of the semilinear wave equation with power nonlinearity
\begin{equation}\begin{aligned} \label{NLW2} \partial_t^2 u -\Delta u +\mu u^{\alpha} &=0  , \quad (t,x) \in [0,T] \times \T^3, \\ 
		u(0)&=u^0, \qquad\partial_tu(0)=v^0,
\end{aligned}\end{equation}
where $\alpha \in \{2,3,4,5\}$ and we allow for both signs $\mu \in \{-1,1\}$. The initial data $(u^0,v^0)$ is assumed to belong to the physically natural energy space $H^1(\T^3) \times L^2(\T^3)$. To keep the presentation simpler, we mostly restrict ourselves to the model cases of powers $\alpha\in\{3,4\}$. 

It is well known that local wellposedness of \eqref{NLW2} (and its variants with nonlinearity $\mu|u|^{\alpha-1}u$ for $\alpha \in (1,5]$) can be shown by a fixed-point argument. If $\alpha \le 3$, the nonlinear terms can be controlled only using classical tools such as Sobolev embedding. In the case of higher powers $\alpha>3$, one has to exploit the dispersive character of the wave equation. A particular useful tool are the Strichartz estimates, which control mixed space-time $L^pL^q$ norms of solutions to the linear wave equation in terms of the data. Thanks to the $L^p$ norm in time, one can choose the space integrability exponent $q$ larger than predicted by a fixed-time Sobolev embedding. This makes it possible to show local wellposedness of \eqref{NLW2} for powers up to the critical value $\alpha=5$, see, e.g., the monograph \cite{Tao2006}.

In this work we are interested in approximating the temporal evolution of \eqref{NLW2}. A natural choice for the time integration of such equations is the class of second-order trigonometric (or exponential) integrators, cf. chapter XIII.2.2 of \cite{HaiLubWan} for an overview. As explained in \cite{BuchholzGauckler}, these methods in one-step form can be interpreted as variants of the Strang splitting with additional filter functions in the nonlinear part. In the context of an ordinary differential equation with a globally Lipschitz continuous nonlinearity, error estimates for such schemes were derived in, e.g., \cite{Garcia,Grimm,HaiLubWan,BuchholzGauckler}. For the PDE \eqref{NLW2} with pure power nonlinearity, an error analysis was for the first time given in \cite{Gauckler}, but only in the one-dimensional case. The proof uses a similar strategy as the earlier work \cite{Lubich} for the nonlinear Schrödinger equation. It was shown that under the finite-energy assumption $(u,\partial_tu) \in H^1 \times L^2$, trigonometric integrators converge with order two in $L^2 \times H^{-1}$ and with order one in the energy space $H^1 \times L^2$ itself, even if no filter functions are used. In \cite{Averaged}, the same error bounds were shown in a more general setting which in particular allows for rough $L^\infty$ coefficients in the nonlinear part. This made it necessary to equip the schemes with suitable filter functions to avoid numerical resonances for certain step-sizes. The higher dimensional cases $d \in \{2,3\}$ were also considered in \cite{Averaged}, but only under the stronger regularity assumption $(u,\partial_tu) \in H^2 \times H^1$.

In the proofs of the one-dimensional results in \cite{Gauckler,Averaged}, it was crucially exploited that the Sobolev space $H^1$ forms an algebra. This is however not the case in higher dimensions, where the estimates for the nonlinear terms become more delicate. The first attempts to exploit Strichartz estimates in numerical analysis were made in the case of nonlinear Schrödinger equations, starting from \cite{NumDisp}. Subsequent works used discrete-time Strichartz estimates to show error bounds under low regularity assumptions, such as \cite{IgnatSplitting,ChoiKoh,ORS} for the nonlinear Schrödinger equation on $\R^d$ and \cite{ORSBourgain,Ji} in the case of nonlinear Schrödinger equation on the torus $\T^d$. In the latter case, the authors further made use of discrete Bourgain spaces. For the nonlinear wave equation \eqref{NLW2}, less literature is available in this context. Based on discrete-time Strichartz estimates, in \cite{RuffSchnaubelt} an error analysis of the Lie splitting for \eqref{NLW2} (with nonlinearity $\mu|u|^{\alpha-1}u$ for $\alpha \in [3,5]$) on the full space $\R^3$ was given, notably including the scaling-critical power $\alpha=5$. It was shown that under the assumption $(u,\partial_tu) \in H^1 \times L^2$, the scheme converges with optimal first order in $L^2 \times H^{-1}$.  

Recently, another class of methods to approximate the temporal evolution of nonlinear dispersive problems especially in low regularity gained a lot of attention, namely, the low-regularity integrators. See \cite{RS} for an overview. Due to an improved local error structure, such schemes can allow for higher convergence rates at low regularity than classical methods such as the Strang splitting. The authors in \cite{KleinGordon} proposed the corrected Lie splitting, which is a low-regularity integrator that can be applied to the nonlinear wave equation \eqref{NLW2}. It was shown that the corrected Lie splitting is second-order convergent in $H^1 \times L^2$ under the regularity condition $(u,\partial_tu) \in H^{1+d/4} \times H^{d/4}$ for dimensions $d \in \{1,2,3\}$. See also \cite{DiscSol} for an error analysis of the corrected Lie splitting in lower regularity {in the cases $d \in \{1,2\}$ with reduced convergence rates under a CFL-type condition}. However, the analyses from \cite{KleinGordon,DiscSol} do not apply {in our situation} since they {either} require a global Lipschitz condition on the nonlinearity (which is not satisfied by the polynomial $u^\alpha$), {or higher-regularity solutions satisfying $u(t)\in H^s$ with $s>d/2$.}

\subsection{{Our contributions}}

The purpose of this paper is, at first, to extend the discrete-time Strichartz estimates for the wave equation from \cite{RuffSchnaubelt} to the bounded domain $\T^3$. By exploiting the finite propagation speed of the wave equation, we obtain locally in time the same Strichartz estimates as on the full space $\R^3$. {This is in sharp contrast to the Schrödinger case (with infinite speed of propagation) where the Strichartz estimates on the torus are restricted compared to those on the full space, cf.\ \cite{Bourgain93}.} Moreover, we aim to prove optimal error bounds for a second-order scheme applied to \eqref{NLW2} under the finite-energy condition. In the important cubic case $\alpha=3$, we almost recover the optimal temporal second-order convergence {in $L^2 \times H^{-1}$. 
Here it is important that we make full use of all available regularity. In particular, besides the Strichartz estimates we at most apply sharp endpoint Sobolev embeddings such as $H^1(\T^3) \hookrightarrow L^6(\T^3)$, and we avoid the use of Hölder's inequality in time.
In the case of a more rapidly growing nonlinearity with $\alpha=4$, the second order can no longer be achieved. This type of order reduction is an interesting effect that apparently does not occur for nonlinear Schrödinger equations. The difference lies in the fact that the nonlinearity in the first-order formulation of the wave equation formally gains a derivative, which, however, can only be optimally exploited for ``smaller'' $\alpha$ using the Strichartz estimates.
Nevertheless, our convergence rates for $\alpha=4$ are still} higher than those obtained in \cite{RuffSchnaubelt}. 
Finally, we extend our results to the fully discrete setting (using the Fourier pseudo-spectral method) with optimal spatial convergence.

Even though our nonlinearity is of pure power-type, we make use of a filter function inside the nonlinearity when estimating the terms resulting from the local error (compared to the one-dimensional case \cite{Gauckler}). This is essentially because in 3D the multiplication by an $H^1 \cap L^\infty$ function is not a bounded operator on $H^{-1}$. As a filter, we use the operator $\Pi_{\tau^{-1}} {= \operatorname{diag}(\pi_{\tau^{-1}},\pi_{\tau^{-1}})}$ which in both components is the Fourier multiplier for the characteristic function of the square $[-\tau^{-1},\tau^{-1}]^3$, where $\tau>0$ denotes the time step size. This particular choice is made for several reasons. First, it enables us to use the summation by parts formula to exploit cancellations in the terms stemming from the local error. Second, a filter of this type is needed to obtain discrete-time Strichartz estimates (compare, e.g., \cite{IgnatSplitting,ORS,ORSBourgain,Ji,RuffSchnaubelt}), which are necessary for $\alpha=4$. Third, it fits well to the spatial discretization with the Fourier pseudo-spectral method {and increases the computational efficiency of the fully discrete scheme}.
As a main conceptual novelty, the proof of our error estimates combines the summation/integration by parts technique (as already used in, e.g., \cite{BuchholzGauckler,Averaged}) with the use of Strichartz estimates. 

In \cite{RuffSchnaubelt}, the terms stemming from the local error were estimated using discrete-time Strichartz estimates. {In combination with a non-optimal frequency filtering}, this led to a loss of convergence order in the error analysis of the formally second-order corrected Lie splitting. {Similar issues already arose in the Schrödinger case, cf.\ \cite{ORS,ORSBourgain}}. In the present paper, we {use a more suitable filtering that was similarly proposed as ``method ($\tilde B$)'' for the one-dimensional semilinear wave equation in \cite{Gauckler}. Moreover, we} show that at least in the case of the Strang splitting {applied to the semilinear wave equation \eqref{NLW2}}, one can avoid {the issues coming from the application of discrete-time Strichartz estimates to the local error terms} by using the continuous-time Strichartz estimates instead. 
In the cubic case $\alpha=3$, it even turns out that we do not need any discrete Strichartz estimates to prove our error result (the continuous ones are still used). This is related to the fact that, as mentioned above, the wellposedness of \eqref{NLW2} can in this case be shown without using Strichartz estimates. In the case $\alpha=4$ however, we need the discrete-time Strichartz estimates to ensure the stability of the numerical scheme. {Here it is crucial that the discrete Strichartz estimates are only needed for the filtered approximation $(\pi_{\tau^{-1}}u_n)_n$, and not for $(u_n)_n$, as postulated in \cite{IgnatSplitting,ORS,RuffSchnaubelt}, where $u_n \approx u(n\tau)$ are the numerical approximations. This makes no difference for the Schrödinger equation, but it does for the wave equation, again due to the formal gain of a derivative in the nonlinear part of the first-order formulation of the wave equation. A uniform filtering of the entire equation (including initial data) as in \cite{IgnatSplitting,ORS,RuffSchnaubelt} would only lead to sub-optimal results in the case of a second-order integrator applied to the semilinear wave equation.}

While extending our results to the fully discrete setting, we face the difficulty that one cannot take advantage of negative-order Sobolev spaces when estimating the trigonometric interpolation error. We solve this issue by using an $L^q$ estimate for the trigonometric interpolation error from \cite{Hristov,AlexandrovPopov,PopovHristov}. 

\subsection{Results in the semi-discrete setting}

In this paper we analyze a variant of the Strang splitting scheme that computes approximations $U_n \approx (u(n\tau),\partial_tu(n\tau))$ for a step size $\tau>0$ and $n \in \N_0$. With the notation $A(u,v) \coloneqq (v, \Delta u)$ for the wave operator and $G(u,v) \coloneqq (0,-\mu u^\alpha)$ for the nonlinearity, the semi-discrete form of the scheme reads
\begin{equation}\label{Strang}\begin{aligned}
		U_{n+1/2} &= e^{\tau A}[U_n + \tfrac\tau2 G(\Pi_{\tau^{-1}}U_n)],  \\
		U_{n+1} &= U_{n+1/2} + \tfrac\tau2 G(\Pi_{\tau^{-1}}U_{n+1/2}), \\
		U_0 &= (u^0,v^0).
	\end{aligned}
\end{equation}
This scheme fits into the class of trigonometric integrators in one-step formulation as described in Section XIII.2.2 of \cite{HaiLubWan}, with ``inner filter'' $\Pi_{\tau^{-1}}$. It corresponds to a variant of ``method ($\tilde B$)'' that was proposed and analyzed in the one-dimensional case in \cite{Gauckler}. 
{If we ignore the filter, the scheme \eqref{Strang} is also a Strang splitting in the classical sense applied to the first-order reformulation $\partial_tU = AU +G(U)$ of the semilinear wave equation \eqref{NLW2}. Indeed, one step of \eqref{Strang} consists of a half step of the exact flow of the filtered nonlinear ``sub-problem'' $\partial_tU = G(\Pi_{\tau^{-1}}U)$, followed by one full step of the exact flow of the linear ``sub-problem'' $\partial_tU = AU$, and finally another ``nonlinear'' half-step. It is therefore a natural (formally) second-order improvement of the filtered Lie splitting that was analyzed in \cite{RuffSchnaubelt}.} 
See \eqref{DefPiPer} for the precise definition of the filter.

\begin{Theorem}\label{ThmStrang}
	Let $U = (u,\partial_tu) \in C([0,T],H^1(\T^3) \times L^2(\T^3))$ solve the semilinear wave equation \eqref{NLW2}. Then there are a constant $C>0$ and a maximum step size $\tau_0>0$, such that the approximations $U_n$ obtained by the Strang splitting scheme \eqref{Strang} satisfy the following error bounds. 
	If $\alpha = 3$,
	\begin{align*}
	\|U(n\tau)-U_n\|_{H^1 \times L^2} &\le C\tau|\log\tau|, \\
	\|U(n\tau)-U_n\|_{L^2 \times H^{-1}} &\le C\tau^2|\log\tau|, 
	\end{align*}
	and for $\alpha=4$,
	\begin{align*}
		\|U(n\tau)-U_n\|_{H^1 \times L^2} &\le C\tau^\frac12, \\
		\|U(n\tau)-U_n\|_{L^2 \times H^{-1}} &\le C\tau^\frac32.
	\end{align*}
	Theses bounds are uniform in $\tau \in (0,\tau_0]$ and $n \in \N_0$ with $n\tau \le T$. The numbers $C$ and $\tau_0$ only depend on $T$ and $\|U\|_{L^\infty([0,T],H^1\times L^2)}$.
\end{Theorem}

\begin{Remark}
	In the case of a quadratic nonlinearity $\alpha=2$, an inspection of the proof of Theorem \ref{ThmStrang} shows that one obtains the same error bounds as in the cubic case $\alpha=3$. For the case of the critical power $\alpha=5$, it is possible to show the convergence results
	\begin{equation}\label{eq:critbound}
	\begin{aligned}
	\|U(n\tau)-U_n\|_{L^2 \times H^{-1}} &\le C\tau, \\
	\|U(n\tau)-U_n\|_{H^1 \times L^2} &\to 0, 
	\end{aligned}
	\end{equation}
	as $\tau \to 0$, uniformly in $n \in \N_0$ with $n\tau \le T$. We do not {elaborate on this here} since \eqref{eq:critbound} was already shown for the simpler Lie splitting on the full space $\R^3$ in \cite{RuffSchnaubelt}. One can show the corresponding result for the torus $\T^3$ by combining the arguments from there with the Strichartz estimates for $\T^3$ developed in the present paper. {See Theorem 4.4 and Section 7.6 of \cite{DissRuff} for details.}
\end{Remark}

\begin{Remark}
	We compare our 3D results to the known results in 1D. If $\alpha=3$, our convergence rates are almost the same as those obtained in the one-dimensional cases in \cite{Gauckler,Averaged}, {which in turn are generally believed to be optimal in terms of the regularity requirements}. For $\alpha=4$, Theorem \ref{ThmStrang} exhibits an order reduction. This reduction and the convergence behavior in \eqref{eq:critbound} can be observed in our numerical experiment in Subsection \ref{SecNum}. {We therefore believe that our results are optimal.}
\end{Remark}

\begin{Remark}
	The logarithm in the result of Theorem \ref{ThmStrang} for $\alpha=3$ comes from the use of the endpoint Strichartz estimate for the $L^2L^\infty$ norm that only holds with a logarithmic correction, cf. Corollary \ref{CorStrichWave}.
\end{Remark}

\subsection{The fully discrete scheme}

Denoting by $K \ge 1$ the spatial discretization parameter for the Fourier pseudo-spectral method, we consider the fully discrete scheme
\begin{equation}\label{StrangFully}\begin{aligned}
		U^K_{n+1/2} &= e^{\tau A}[U^K_n + \tfrac\tau2 \mathcal I_KG(\Pi_{\tau^{-1} }U^K_n)]  \\
		U^K_{n+1} &= U^K_{n+1/2} + \tfrac\tau2 \mathcal I_KG(\Pi_{\tau^{-1} }U^K_{n+1/2}) \\
		U_0^K &= \Pi_K(u^0,v^0).
	\end{aligned}
\end{equation}
Here, we use the notation $\mathcal I_K = \operatorname{diag}(I_K,I_K)$ for the trigonometric interpolation operator $I_K$, cf.\ Definition \ref{DefTrigInt}.
For this scheme, we obtain fully discrete error bounds of spatial order $K^{-1}$. This is optimal in view of the projection error \[\|(I-\Pi_K)U(n\tau)\|_{L^2\times H^{-1}} \lesssim K^{-1}\|U(n\tau)\|_{H^1\times L^2},\]
cf.\ Lemma \ref{LemPiNew}.
\begin{Theorem}\label{ThmStrangFully}
	Let $U = (u,\partial_tu) \in C([0,T],H^1(\T^3) \times L^2(\T^3))$ solve the semilinear wave equation \eqref{NLW2}. Then there are positive constants $C$, $\tau_0$, and $K_0$, such that the approximations $U_n^K$ obtained by the fully discrete Strang algorithm \eqref{StrangFully} satisfy the error bounds
	\begin{alignat*}{2}
		\|U(n\tau)-U^K_n\|_{L^2 \times H^{-1}} &\le C(\tau^2|\log\tau| +K^{-1}), &&\quad\text{if}\ \alpha=3, \\
		\|U(n\tau)-U^K_n\|_{L^2 \times H^{-1}} &\le C(\tau^\frac32 + K^{-1}), &&\quad\text{if}\ \alpha=4,
	\end{alignat*}
	uniformly in $\tau \in (0,\tau_0]$, $K \ge K_0$ and $n \in \N_0$ with $n\tau \le T$. The numbers $C$, $\tau_0$, and $K_0$ only depend on $T$ and $\|U\|_{L^\infty([0,T],H^1\times L^2)}$.
\end{Theorem}

\subsection{{Discussion and outlook}}
{We comment on possible generalizations of our results.}

\begin{Remark}\label{RemGeneralAlpha}
	We will sometimes exploit that our nonlinearity is a polynomial in $u$, but strictly speaking, this is not necessary. With somewhat greater technical effort we could also treat the equation
	\begin{equation}\label{NLWGeneral} 
		\partial_t^2 u -\Delta u  =  g(u)
	\end{equation}
	with a general nonlinearity $g \in C^2(\R,\R)$ satisfying the bounds
	\begin{align*}
		|g(z)| &\lesssim 1 + |z|^\alpha, \\
		|g'(z)| &\lesssim 1 + |z|^{\alpha-1}, \\
		|g''(z)| &\lesssim 1 + |z|^{\alpha-2},
	\end{align*}
	for $z \in \R$. This covers in particular the semilinear Klein--Gordon equation since the lower-order mass term can be moved into the nonlinearity. One could also allow fractional $\alpha$. For $\alpha \in (3,5)$, the convergence of the {semi-discrete} Strang splitting \eqref{Strang} {in Theorem \ref{ThmStrang}} would then be of order $(7-\alpha)/2$ in the $L^2 \times H^{-1}$ norm and of order $(5-\alpha)/2$ in the $H^1 \times L^2$ norm, respectively. {Similar considerations apply to the error bounds for the fully discrete scheme \eqref{StrangFully} from Theorem \ref{ThmStrangFully}.} In view of the error bound for the trigonometric interpolation from Lemma \ref{LemTrigIntError}, we then would additionally need the third derivative of $g$ with bound
	\[|g'''(z)| \lesssim 1 + |z|^{\alpha-3}. \]
\end{Remark}

\begin{Remark}
	In view of the error bounds of Theorem \ref{ThmStrangFully}, it might be advantageous to choose the spatial resolution finer than the temporal one. In the case $K > \alpha/\tau$, it turns out that the highest frequencies $(\Pi_K-\Pi_{\alpha/\tau})U_n^K$ in \eqref{StrangFully} are only influenced by the linear part $e^{tA}$ and not by the nonlinear function $G$. Therefore, in that case, if one is only interested in the numerical approximation $U_N^K$ for some $N \gg 1$, the high-frequency part can be computed directly from the initial data via $(\Pi_K-\Pi_{\alpha/\tau})U_N^K = e^{N\tau A} (\Pi_K-\Pi_{\alpha/\tau})(u^0,v^0)$, without time-stepping. This idea was also exploited in the recent paper \cite{DiscSol}.
\end{Remark}

\begin{Remark}
	One might wonder if under our assumption $(u,\partial_tu) \in H^1 \times L^2$, a low-regularity integrator such as the corrected Lie splitting proposed in \cite{KleinGordon} can give higher convergence rates than classical schemes such as the Strang splitting. The authors in \cite{KleinGordon} show that this is possible in the one-dimensional case. However, in our 3D case we did not succeed to find such a result so far. 
\end{Remark}

\begin{Remark}
	The differential equation \eqref{NLW2} is posed on the torus $\T^3$ which corresponds to a cube with periodic boundary conditions. We could also treat the problem on a cube with Dirichlet or Neumann boundary conditions by restricting the full Fourier basis to a sine or cosine basis. We omit the details for the sake of brevity.
\end{Remark}

\subsection{Organization of the paper}

In Section \ref{SecStrich}, we give a review on Strichartz estimates for the linear wave equation. We give a proof for the discrete-time inequalities on the torus and state all the Strichartz estimates needed in this paper. Moreover, we collect some important properties of the filter operator. In Section \ref{SecNonlinear} we give a brief review on the local wellposedness theory of \eqref{NLW2} and give important estimates for its solution. In Section \ref{SecErrorAnalysis}, the error analysis of the semi-discrete Strang splitting \eqref{Strang} is carried out, in particular, the proof of Theorem \ref{ThmStrang}. The last Section \ref{SecFullyDiscr} contains the proof of the fully discrete error bounds from Theorem \ref{ThmStrangFully}. Finally, we discuss a numerical experiment to illustrate our temporal error bounds.

\subsection{Notations}
	
	We write $A \lesssim B$ (or $A \lesssim_\beta B$) if $A \le cB$ for a generic constant $c > 0$ (depending on quantities $\beta$). The torus $\T^3$ is understood as the cube $[-\pi,\pi]^3$ where one identifies opposite sides. We denote the space of distributions on the torus by $\mathcal D'(\T^3)$. The $k$-th Fourier coefficient of a distribution $v \in \mathcal D'(\T^3)$ is defined by 
	\[\hat v_k = (2\pi)^{-\frac32} \langle v, e^{-\iu k\cdot x} \rangle_{\mathcal D'(\T^3) \times \mathcal D(\T^3)}, \quad k \in \Z^3. \]
	For a real number $s \in \R$, the Sobolev spaces on $\T^3$ are given by
	\[ H^s(\T^3) = \{ v \in \mathcal D'(\T^3) \colon \|v\|_{H^s(\T^3)} < \infty \}  \]
	with norm
	\begin{equation}\label{DefSobNormTorus}
	\|v\|^2_{H^s(\T^3)} = \sum_{k \in \Z^3}(1+|k|^2)^s|\hat v_k|^2.
	\end{equation}
	
	In Section \ref{SecStrich}, we will also make use of the Fourier transform and Sobolev spaces on $\R^3$. We write $\mathcal F$ for the Fourier transform on $\R^3$, using the convention with the prefactor $(2\pi)^{-3/2}$. We also use the notation $\hat u \coloneqq \mathcal F u$. In the context of Fourier multipliers, we often simply write $\xi$ instead of the map $\xi \mapsto \xi$. For $s \in \R$, we use the inhomogeneous and homogeneous Sobolev norms 
	\[
	\|w\|_{H^{s}(\R^3)} = \|(1+|\xi|^2)^{\frac{s}{2}} \hat w\|_{L^2(\R^3)}, \quad
	\|w\|_{\dot H^{s}(\R^3)} = \||\xi|^{s} \hat w\|_{L^2(\R^3)},
	\]
	if $\hat w \in L^1_{\mathrm{loc}}(\R^3)$. The homogeneous Sobolev space $\dot H^s(\R^3)$ is defined as
	\[\dot H^s(\R^3) = \{w \in \mathcal S'(\R^3) : \hat w \in L^1_{\mathrm{loc}}(\R^3)\ \text{and}\ \|w\|_{\dot H^s(\R^3)} < \infty\},\]
	where $\mathcal S(\R^3)$ denotes the space of tempered distributions on $\R^3$. By Proposition 1.34 of \cite{Bahouri}, the homogeneous space $\dot H^s(\R^3)$ is complete if and only if $s<3/2$.
	
	Let $h \colon \R \to \R$ be a bounded function. To denote the Fourier multiplication operator for the function $\xi \mapsto h(|\xi|)$ (on $\R^3$) and $k \mapsto h(|k|)$ (on $\T^3$), we will use the notation $h(|\nabla|)$ in both cases. It is clear from the definition of the Sobolev norms that the operator $h(|\nabla|)$ is uniformly bounded on all spaces $H^s(\T^3)$, $H^s(\R^3)$ and $\dot H^s(\R^3)$.
	
	Let $p \in [1,\infty]$, $J$ be a time interval, and $X$ be a Banach space. We use the Bochner space $L^p(J,X)$ with norm
	\[\|F\|_{L^p(J,X)} = \Big(\int_J\|F(t)\|_X^p \dd t\Big)^\frac{1}{p},\]
	and the usual modification for $p=\infty$. If a ``free'' variable $t$ occurs in such a Bochner norm, the time integration is carried out with respect to $t$. For a step size $\tau>0$, we also introduce the discrete-time norm 
	\[\|F\|_{\ell^p_{\tau}(J,X)} \coloneqq \Big(\tau \sum_{\substack{n \in \Z\\ n\tau \in J}} \|F_n\|_{X}^p \Big)^{\frac{1}{p}}.  \]
	To simplify notation we often write $\|F_n\|_{\ell^p_{\tau}(J,X)}$ instead of $\|(F_n)_{n \in \Z}\|_{\ell^p_{\tau}(J,X)}$, where again a ``free'' variable $n$ is assumed as the summation variable.	
	In the case $J = [0,T]$, we abbreviate $L^p_TX = L^p([0,T],X) $ and $\ell^p_{\tau,T}X = \ell^p_\tau([0,T],X)$.
	
\section{Linear estimates}\label{SecStrich}

\subsection{Continuous and discrete Strichartz estimates on the full space}

A triple $(p,q,\gamma)$ is called \emph{admissible} (in dimension three), if $p \in (2,\infty]$, $q \in [2,\infty)$, and
\begin{equation}\label{DefAdmissible}
	\frac{1}{p}+\frac{1}{q} \le \frac{1}{2}, \qquad \frac{1}{p}+\frac{3}{q} = \frac{3}{2}-\gamma. 
\end{equation}
One then has $\gamma \in [0,\frac32)$, and the equality in \eqref{DefAdmissible} is called \emph{scaling condition}. The following theorem is well known, cf.\ Chapter IV.1 of \cite{Sogge}.

\begin{Proposition}\label{ThmStrich}
	Let $(p,q,\gamma)$ be admissible. Then we have the estimate
	\begin{align*}
		\|e^{\iu t |\nabla|}f\|_{L^p(\R, L^q(\R^3))} &\lesssim_{p,q} \|f\|_{\dot H^\gamma(\R^3)}, 
	\end{align*}
	for all $f \in \dot H^\gamma(\R^3)$.
\end{Proposition}

Observe that the triple $(p,q,\gamma) = (\infty,2,0)$ corresponds to the usual energy estimate. If we take $p<\infty$, then the scaling condition implies that we save $1/p$ derivatives compared with a fixed-time Sobolev embedding.

\begin{Remark}
	The transformation $f \mapsto \bar f$ shows that in the estimates of this subsection, one can always replace $e^{\iu t |\nabla|}$ with $e^{-\iu t |\nabla|}$. 
\end{Remark}

The estimate from {Proposition} \ref{ThmStrich} can be applied to the wave equation in the following way. Let $\gamma \in [0,\frac32)$. For given initial data $f \in \dot H^\gamma(\R^3)$, $g \in \dot H^{\gamma-1}(\R^3)$, we define the function
\[w(t) \coloneqq \cos(t|\nabla|) f + |\nabla|^{-1}\sin(t|\nabla|)g, \quad t \in \R.\]
Using the Fourier transform, one checks that $w \in C(\R,\dot H^\gamma(\R^3))$ is the distributional solution to the linear wave equation 
\[\partial_{tt}w-\Delta w = 0, \quad w(0) = f, \quad \partial_tw(0) = g, \]
and {Proposition} \ref{ThmStrich} gives the estimate
\[\|w\|_{L^p(\R,L^q(\R^3))} \lesssim_{p,q} \|f\|_{\dot H^\gamma(\R^3)} + \|g\|_{\dot H^{\gamma-1}(\R^3)}, \]
whenever $(p,q,\gamma)$ is admissible.

To obtain discrete-time Strichartz estimates, it is necessary to include a suitable filter operator. This was first observed in case of the Schrödinger equation, cf.\ \cite{IgnatSplitting}. On the full space, we will use the circular Fourier cut-off
\begin{equation}\label{DefPi}
	\tilde \pi_K \coloneqq \mathcal F^{-1} \mathbbm{1}_{B(0,K)}\mathcal F 
\end{equation}
at level $K \ge 1$.
In \cite{RuffSchnaubelt}, the following discrete Strichartz estimate for the half-wave group was proven.

\begin{Proposition}\label{ThmDisStrich}
	Let $(p,q,\gamma)$ be admissible. Then we have the estimate
	\[
		\|\tilde \pi_Ke^{\iu n\tau|\nabla|} f \|_{\ell^p_\tau (\R,L^q(\R^3))} \lesssim_{p,q} (1+K\tau)^\frac{1}{p} \|f\|_{\dot H^\gamma(\R^3)},
	\]
	for all $\tau \in (0,1]$, $K \ge 1$, and $f \in \dot H^\gamma(\R^3)$.
\end{Proposition}

This estimate is optimal in the following sense. If we assume that $K\tau\ge 1$ and only consider the term with $n=0$ in the left-hand side of {Proposition} \ref{ThmDisStrich}, we obtain the frequency-localized Sobolev embedding
\[ \|\tilde \pi_K f \|_{L^q(\R^3)} \lesssim_{p,q} K^\frac{1}{p} \|f\|_{\dot H^\gamma(\R^3)}, \]
which in general is sharp by scaling.

The estimates of {Propositions} \ref{ThmStrich} and \ref{ThmDisStrich} fail at the so-called ``double endpoint'' $(p,q,\gamma) = (2,\infty,1)$. However, estimates with logarithmic corrections in time and frequency are available. In \cite{RuffSchnaubelt}, the following discrete-time bound was shown, inspired by a corresponding inequality for the continuous case in \cite{Joly}. 

\begin{Proposition}\label{ThmEndp}
	The estimate
	\[
		\|\tilde \pi_Ke^{\iu n\tau|\nabla|} f \|_{\ell^2_{\tau,T} L^\infty(\R^3)} \lesssim (1+K\tau + \log(1+KT))^\frac12 \|f\|_{\dot H^1(\R^3)},
	\]
	holds for all $\tau \in (0,1]$, $K \ge 1$, $T \ge 0$, and $f \in \dot H^1(\R^3)$.
\end{Proposition}

\subsection{Strichartz estimates for the half-wave group on the torus}
Thanks to the finite speed of propagation for the wave equation, one expects that locally in time, one has the same Strichartz estimates on the torus $\T^3$ as on the full space $\R^3$. For continuous time, this has been carried out in, e.g., \cite{KapCP1} by using suitable extension and cut-off operators. We will follow the same strategy to prove corresponding versions of the discrete-time {Propositions} \ref{ThmDisStrich} and \ref{ThmEndp} for the torus.

Let $E \colon \mathcal D'(\T^3) \to \mathcal S'(\R^3)$ denote the periodic extension operator (where we interpret $\T^3 = [-\pi,\pi]^3$ as above). Note that for $f \in C^\infty(\T^3)$ we have $Ef \in C^\infty(\R^3)$ with periodic partial derivatives. 
The next lemma shows that an extended Sobolev function multiplied with a smooth cut-off function belongs to the corresponding Sobolev space on $\R^3$. 

\begin{Lemma}\label{LemTorusAbschn}
	Let $\eta \in C_c^\infty(\R^3)$ and $s \in \R$. Then the estimate
	\[\|\eta Ef\|_{H^s(\R^3)} \lesssim_{\eta,s} \|f\|_{H^s(\T^3)} \]
	is true for any $f \in H^s(\T^3)$.
\end{Lemma}
\begin{proof}
	By approximation, it suffices to consider smooth $f$. The statement is clear if $s=0$, and inductively extends to all $s \in \N$. By interpolation, we then infer the assertion for all $s\ge0$. The case $s<0$ is handled via duality. Let $(\phi_j)_{j\in\N}$ be a smooth partition of unity such that $\sum_{j\in \N}\phi_j = \mathbbm 1$ and $\phi_j \in C_c^{\infty}(\R^3)$ with $\supp \phi_j \subseteq \{y_j\} + (-\pi,\pi)^3$ for all $j \in \N$ and some $y_j \in \R^3$. We compute
	\begin{align*}
		\|\eta Ef\|_{H^s(\R^3)} &= \sup_{\|g\|_{H^{-s}(\R^3)}=1} \Big| \int_{\R^3} \eta Ef \cdot g \dd x \Big| \\
		&=  \sup_{\|g\|_{H^{-s}(\R^3)}=1} \Big| \sum_{j \in \N} \int_{\{y_j\} + (-\pi,\pi)^3} Ef \cdot\eta  g\phi_j \dd x \Big| \\
		&\le  \sup_{\|g\|_{H^{-s}(\R^3)}=1} \|f\|_{H^s(\T^3)}\sum_{j \in \N}   \|\eta  g\phi_j\|_{H^{-s}(\R^3)} \lesssim_{\eta,s} \|f\|_{H^s(\T^3)},
	\end{align*} 
	where the supremum is taken over smooth $g$.
	Here we used that we can consider $(\eta g\phi_j)(y_j+\cdot)$ as a test function on $\T^3$, and that the sum is actually finite thanks to the compact support of $\eta$.
\end{proof}

For the discrete-time Strichartz estimates, we need to introduce a Fourier cut-off on the torus (similar as \eqref{DefPi} on $\R^3$). For $f \in \mathcal D'(\T^3)$ and $K \ge1$ we define the square frequency cut-off operator $\pi_K$ via the truncated Fourier series
\begin{equation}\label{DefPiPer}
	(\pi_K f)(x) \coloneqq (2\pi)^{-\frac32} \sum_{|k|_\infty \le K} \hat f_k e^{\iu k \cdot x}, \quad x \in \T^3. 
\end{equation}
Here, the sum is taken over all $k \in \Z^3$ with $|k|_\infty = \max_{j=1,2,3}|k_j| \le K$, and $\hat f_k$
denotes the $k$-th Fourier coefficient of $f$. From the definition of the Sobolev norm \eqref{DefSobNormTorus}, it follows that $\pi_K$ is bounded on all spaces $H^s(\T^3)$, uniformly in $s \in \R$ and $K \ge 1$.

\begin{Theorem}\label{ThmDisStrichTorus}
	Let $(p,q,\gamma)$ be admissible and $T\ge 0$. We then have the estimates
	\begin{align*}
		\|e^{\iu n \tau|\nabla|}\pi_K f \|_{\ell^p_{\tau,T}L^q(\T^3)} &\lesssim_{p,q,T} (1+K\tau)^{\frac1p} \|f\|_{H^\gamma(\T^3)}, \\
		\|e^{\iu n \tau|\nabla|}\pi_K h \|_{\ell^2_{\tau,T}L^\infty(\T^3)} &\lesssim_{T} (1+K\tau+\log K)^{\frac12} \|h\|_{H^1(\T^3)},  
	\end{align*}
	for all $f \in H^\gamma(\T^3)$, $h \in H^1(\T^3)$, $\tau \in (0,1]$, and $K \ge 1$.
\end{Theorem}
\begin{proof}
	We only give the proof for the first estimate, since the second one can be shown in the same way, using {Proposition} \ref{ThmEndp} instead of {Proposition} \ref{ThmDisStrich}. We define the function $v(t) \coloneqq e^{\iu t|\nabla|}\pi_K f$ for $t \in \R$. Since
	\begin{equation*}
		v(t) = e^{\iu t|\nabla|}\pi_K f = \cos(t|\nabla|)\pi_Kf + \iu |\nabla|^{-1}\sin(t|\nabla|)|\nabla|\pi_K f
	\end{equation*}
	is the smooth solution to the linear homogeneous wave equation on $\R \times \T^3$ with initial data $(\pi_Kf,\iu  |\nabla|\pi_Kf)$,
	the extended function $Ev$ solves the corresponding problem on $\R \times \R^3$ with extended initial data $(E\pi_Kf,\iu E |\nabla|\pi_Kf)$, i.e.,
	\begin{equation*}
		(\partial_{tt}-\Delta)Ev = 0, \quad Ev(0) = E\pi_Kf, \quad \partial_tEv(0) = \iu E |\nabla|\pi_Kf. 
	\end{equation*}
	Let $\eta \in C_c^\infty(\R^3)$ be a cut-off function such that $\eta = 1$ on $B(0,\sqrt 3 \pi + T)$. The function
	\begin{equation*}
		w(t) \coloneqq \cos(t|\nabla|)(\eta E\pi_Kf) + \iu|\nabla|^{-1}\sin(t|\nabla|)(\eta E |\nabla|\pi_Kf), \quad t \in \R,
	\end{equation*}
	solves the same full space wave equation with truncated initial data. Finite speed of propagation yields $Ev(t,x) = w(t,x)$ for all $(t,x) \in \R^{1+3}$ with $|t|+|x| \le \sqrt 3 \pi + T$. This condition is in particular satisfied if $(t,x) \in [0,T] \times (-\pi,\pi)^3$. Consequently,
	\begin{align}\label{DisStrichPerZerl}
		&\|v(n\tau)\|_{\ell^p_{\tau,T}L^q(\T^3)} \\
		&= \|Ev(n\tau)\|_{\ell^p_{\tau,T}L^q((-\pi,\pi)^3)} = \|w(n\tau)\|_{\ell^p_{\tau,T}L^q((-\pi,\pi)^3)} \le \|w(n\tau)\|_{\ell^p_{\tau,T}L^q(\R^3)} \nonumber \\
		&\le \|\cos(n\tau|\nabla|)\eta E\pi_Kf\|_{\ell^p_{\tau,T}L^q(\R^3)} + \||\nabla|^{-1} \sin(n\tau|\nabla|) \eta E |\nabla|\pi_Kf\|_{\ell^p_{\tau,T}L^q(\R^3)}.\nonumber
	\end{align}
	We decompose the cosine-term in \eqref{DisStrichPerZerl} as 
	\begin{align}\label{DisStrichPerCos}
		&\|\cos(n\tau|\nabla|)\eta E\pi_Kf\|_{\ell^p_{\tau,T}L^q(\R^3)} \\ 
		&\le \|\tilde \pi_{2K}e^{\pm \iu n\tau|\nabla|}\eta E\pi_Kf\|_{\ell^p_{\tau,T}L^q(\R^3)} + \|(I-\tilde\pi_{2K})e^{\pm \iu n\tau|\nabla|}\eta E\pi_Kf\|_{\ell^p_{\tau,T}L^q(\R^3)}.\nonumber
	\end{align}
	The first term of \eqref{DisStrichPerCos} is estimated using {Proposition} \ref{ThmDisStrich} and Lemma \ref{LemTorusAbschn}, which gives
	\begin{align*}
		\|\tilde \pi_{2K}e^{\pm \iu n\tau|\nabla|}\eta E\pi_Kf\|_{\ell^p_{\tau,T}L^q(\R^3)} &\lesssim_{p,q} (1+K\tau)^{\frac1p} \|\eta E\pi_Kf\|_{\dot H^\gamma(\R^3)} \\
		&\lesssim_T (1+K\tau)^{\frac1p} \|f\|_{H^\gamma(\T^3)}.
	\end{align*}
	For the second term of \eqref{DisStrichPerCos}, we first compute the Fourier transform
	\begin{equation*}
		(2\pi)^{\frac32}\mathcal F(\eta E\pi_Kf) = \hat \eta * \mathcal F(E\pi_Kf) = \hat \eta * \sum_{|k|_\infty \le K} \hat f_k \delta_k  = \sum_{|k|_\infty \le K} \hat f_k \hat \eta(\cdot-k),
	\end{equation*}
	where $\delta_k$ denotes the Dirac delta at $k \in \Z^3$.
	We now use the Hausdorff--Young inequality to obtain
	\begin{align*}
		&\|(I-\tilde \pi_{2K})e^{\pm \iu n\tau|\nabla|}\eta E\pi_Kf\|_{\ell^p_{\tau,T}L^q(\R^3)} \\
		&\le \|\mathbbm1_{\{|\xi|\ge 2K\}} e^{\pm\iu n \tau|\xi|} \mathcal F(\eta E\pi_Kf)\|_{\ell^p_{\tau,T}L^{q'}(\R^3)}\\
		&\lesssim_T \|\mathbbm1_{\{|\xi|\ge 2K\}} \sum_{|k|_\infty \le K} |\hat f_k| |\hat \eta(\cdot-k)|\|_{L^{q'}(\R^3)}. 
	\end{align*}
	For $|\xi|\ge 2K\ge \sqrt{\tfrac43}|k|$, we can estimate
	\[ \sum_{|k|_\infty \le K} |\hat f_k| \lesssim K^{\frac32} \Big(\sum_{|k|_\infty \le K} |\hat f_k|^2\Big)^{\frac12} \lesssim |\xi|^{\frac32} \|f\|_{L^2(\T^3)} \]
	thanks to {the Cauchy--Schwarz inequality and the Parseval equality}. Moreover,
	\[|\hat \eta(\xi-k)| \lesssim_\eta |\xi-k|^{-5} \le (|\xi|-|k|)^{-5} \lesssim |\xi|^{-5} \]
	since $\hat \eta$ is a Schwartz function. Altogether, this gives
	\begin{align*}
		\|(I-\tilde \pi_{2K})e^{\pm \iu n\tau|\nabla|}\eta E\pi_Kf\|_{\ell^p_{\tau,T}L^q(\R^3)} &\lesssim_T  \|\mathbbm1_{\{|\xi|\ge 2K\}}|\xi|^{-\frac72}\|_{L^{q'}(\R^3)} \|f\|_{L^2(\T^3)} \\
		&\lesssim \|f\|_{L^2(\T^3)}.
	\end{align*}
	For the sine term in \eqref{DisStrichPerZerl}, we treat the low frequencies separately to avoid problems coming from the negative homogeneous derivative $|\nabla|^{-1}$.
	We decompose
	\begin{align*}
		&\||\nabla|^{-1} \sin(n\tau|\nabla|) \eta E |\nabla|\pi_Kf\|_{\ell^p_{\tau,T}L^q(\R^3)} \\
		&\le  \|\tilde \pi_1|\nabla|^{-1} \sin(n\tau|\nabla|) \eta E |\nabla|\pi_Kf\|_{\ell^p_{\tau,T}L^q(\R^3)} \\
		&\quad + \|(I-\tilde \pi_1)\tilde \pi_{2K}|\nabla|^{-1} e^{\pm\iu n\tau|\nabla|} \eta E |\nabla|\pi_Kf\|_{\ell^p_{\tau,T}L^q(\R^3)} \\
		&\quad + \|(I-\tilde \pi_{2K})|\nabla|^{-1} e^{\pm\iu n\tau|\nabla|} \eta E |\nabla|\pi_Kf\|_{\ell^p_{\tau,T}L^q(\R^3)}.
	\end{align*}
	For the low frequencies, we use Bernstein's inequality in space, Hölder's inequality in time, the boundedness of $x \mapsto \frac1x\sin x$ and finally Lemma \ref{LemTorusAbschn} to obtain
	\begin{align*}
		&\|\tilde \pi_1|\nabla|^{-1} \sin(n\tau|\nabla|) \eta E |\nabla|\pi_Kf\|_{\ell^p_{\tau,T}L^q(\R^3)} \\
		&\lesssim_T \|\tilde \pi_1|\nabla|^{-1} \sin(n\tau|\nabla|) \eta E |\nabla|\pi_Kf\|_{\ell^\infty_{\tau,T}L^2(\R^3)} \\
		&\lesssim_T \|\tilde \pi_1\eta E |\nabla|\pi_Kf\|_{L^2(\R^3)} \lesssim \|\eta E |\nabla|\pi_Kf\|_{H^{\gamma-1}(\R^3)} \lesssim_T \| |\nabla|\pi_Kf\|_{H^{\gamma-1}(\T^3)} \\ 
		&\le \|f\|_{H^{\gamma}(\T^3)}.
	\end{align*}
	The medium and high frequency terms are treated with the same technique used for the cosine-term. 
	{Proposition} \ref{ThmDisStrich} and Lemma \ref{LemTorusAbschn} give 
	\begin{align*}
		&\|(I-\tilde \pi_1)\tilde \pi_{2K}|\nabla|^{-1} e^{\pm\iu n\tau|\nabla|} \eta E |\nabla|\pi_Kf\|_{\ell^p_{\tau,T}L^q(\R^3)} \\
		&\lesssim_{p,q} (1+K\tau)^{\frac1p} \|(I-\tilde \pi_1)|\nabla|^{-1} \eta E |\nabla|\pi_Kf\|_{\dot H^\gamma(\R^3)} \\ &\lesssim (1+K\tau)^{\frac1p} \|\eta E |\nabla|\pi_Kf\|_{H^{\gamma-1}(\R^3)} \lesssim_T (1+K\tau)^{\frac1p}\|f\|_{H^\gamma(\T^3)},
	\end{align*}
	where the operator $I-\tilde \pi_1$ was used to replace the homogeneous by the inhomogeneous Sobolev norm.
	Finally, we get similar as above
	\begin{align*}
		&\|(I-\tilde \pi_{2K})|\nabla|^{-1} e^{\pm\iu n\tau|\nabla|} \eta E |\nabla|\pi_Kf\|_{\ell^p_{\tau,T}L^q(\R^3)} \\
		&\le \|\mathbbm1_{\{|\xi|\ge 2K\}}|\xi|^{-1}  e^{\pm\iu n \tau|\xi|} \mathcal F(\eta E|\nabla|\pi_Kf)\|_{\ell^p_{\tau,T}L^{q'}(\R^3)}\\
		&\lesssim_T \|\mathbbm1_{\{|\xi|\ge 2K\}} |\xi|^{-1}  \sum_{|k|_\infty \le K} |k\hat f_k \hat \eta(\cdot-k)|\|_{L^{q'}(\R^3)} \\
		&\lesssim \|\mathbbm1_{\{|\xi|\ge 2K\}}  \sum_{|k|_\infty \le K} |\hat f_k \hat \eta(\cdot-k)|\|_{L^{q'}(\R^3)} \\
		&\lesssim_T \|f\|_{L^2(\T^3)}.
	\end{align*}
	The assertion now follows from \eqref{DisStrichPerZerl}.
\end{proof}

We now show that the discrete-time Strichartz estimates imply the ones in continuous time, using an argument from Theorem 1.3 of \cite{TaoBilin}. The estimates could also be deduced from the full space inequalities reasoning as in Theorem \ref{ThmDisStrichTorus}, cf.\ \cite{KapCP1}. 

\begin{Corollary}\label{CorKontStrich}
	Let $(p,q,\gamma)$ be admissible and $T>0$. We then have the estimates
	\begin{align*}
		\|e^{\iu t|\nabla|} f \|_{L^p_TL^q(\T^3)} &\lesssim_{p,q,T}\|f\|_{H^\gamma(\T^3)}, \\ 
		\|e^{\iu t|\nabla|}\pi_K h \|_{L^2_{T}L^\infty(\T^3)} &\lesssim_{T} (1+ \log K)^{\frac12} \|h\|_{H^1(\T^3)} ,
	\end{align*}
for all $f \in H^\gamma(\T^3)$, $h \in H^1(\T^3)$, and $K \ge 1$.
\end{Corollary}
\begin{proof}
	We only give the proof for the second estimate, since it is somewhat non-standard. The first one can be proven in the same way, additionally using the density of functions having compact Fourier support in $H^\gamma(\T^3)$. From Theorem \ref{ThmDisStrichTorus} we get 
	\[\tau \sum_{n=0}^N \|\pi_K e^{\iu n\tau|\nabla|} h \|_{L^\infty(\T^3)}^2 \lesssim_{T} (K\tau+\log K) \|h\|_{H^1(\T^3)}^2 \]
	for all $\tau \in [1/K,1]$ and $N \in \N_0$ with $N\tau \le T$. We now replace $h$ with $e^{\iu \theta|\nabla|}h$ and integrate from $0$ to $\tau$ to obtain
	\[\int_0^\tau  \sum_{n=0}^N \|\pi_K e^{\iu (n\tau+\theta)|\nabla|} h \|_{L^\infty(\T^3)}^2 \dd \theta \lesssim_{T} (K\tau+\log K) \|h\|_{H^1(\T^3)}^2, \]
	which implies the assertion if we set $\tau=1/K$.
\end{proof}

\subsection{Application to the wave equation on $\T^3$}
From now on we will only work on $\T^3$, hence we will abbreviate $L^q = L^q(\T^3)$ etc. Moreover, we will only use admissible triples $(p,q,\gamma)$ with derivative loss $\gamma=1$. We call a pair $(p,q)$ \emph{$H^1$-admissible} if $(p,q,1)$ is admissible in the sense of \eqref{DefAdmissible}.

\begin{Corollary}\label{CorStrichWave}
	Let $T>0$, $f \in H^1$, $g \in L^2$, $F \in L^1_TL^2$, and $w \in C([0,T],H^1)$ be the solution to the inhomogeneous wave equation
	\[\partial_{tt}w-\Delta w = F, \quad w(0) = f, \quad \partial_tw(0) = g. \]
	Let moreover $(p,q)$ be $H^1$-admissible. Then $w$ satisfies the estimates
	\begin{align*}
	\|w\|_{L^p_TL^q} + \|\pi_{\tau^{-1}}w(n\tau)\|_{\ell^p_{\tau,T}L^q} &\lesssim_{p,q,T} \|f\|_{H^1} + \|g\|_{L^2} + \|F\|_{L^1_TL^2}, \\
	\|\pi_Kw\|_{L^2_TL^\infty} &\lesssim_T (1+\log K)^\frac12\Big(\|f\|_{H^1} + \|g\|_{L^2} + \|F\|_{L^1_TL^2}\Big)
	\end{align*}
	for all $\tau \in (0,1]$ and $K \ge 1$.
\end{Corollary}
\begin{proof}
	By Duhamel's formula, $w$ is given by
	\[w(t) = \cos(t|\nabla|)f + t \sinc(t|\nabla|)g + \int_0^t(t-s)\sinc((t-s)|\nabla|)F(s) \dd s,  \]
	for $t \in [0,T]$, where $\sinc(z) = \sin(z)/z$. We only give the details for the discrete-time estimate, since the others are obtained similarly, using Corollary \ref{CorKontStrich} in place of Theorem \ref{ThmDisStrichTorus}. First, by Theorem \ref{ThmDisStrichTorus},
	\[ \|\pi_{\tau^{-1}}\cos(n\tau|\nabla|)f\|_{\ell^p_{\tau,T}L^q} \le \|\pi_{\tau^{-1}}e^{\pm \iu n\tau |\nabla|}f\|_{\ell^p_{\tau,T}L^q} \lesssim_{p,q,T} \|f\|_{H^1}. \]
	For the other terms we need to treat the zero-th Fourier coefficient separately. We get
	\begin{align*}
		\|\pi_{\tau^{-1}}n\tau\sinc(n\tau|\nabla|)g\|_{\ell^p_{\tau,T}L^q} &\lesssim_T \|\pi_{\tau^{-1}}e^{\pm \iu n \tau|\nabla|}|\nabla|^{-1}(g-\hat g_0)\|_{\ell^p_{\tau,T}L^q} + |\hat g_0| \\
		&\lesssim_{p,q,T} \||\nabla|^{-1}(g-\hat g_0)\|_{H^1} + \|g\|_{L^2} \lesssim \|g\|_{L^2},
	\end{align*}
	and similarly
	\begin{align*}
		&\Big\|\pi_K\int_0^{n\tau}(n\tau-s)\sinc((n\tau-s)|\nabla|)F(s) \dd s\Big\|_{\ell^p_{\tau,T}L^q} \\
		&\le \int_0^{T}\|\pi_K(n\tau-s)\sinc((n\tau-s)|\nabla|)F(s)\|_{\ell^p_{\tau,T}L^q} \dd s \\
		&\lesssim_T \int_0^{T}\|\pi_Ke^{\iu(n\tau-s)|\nabla|}|\nabla|^{-1}(F(s)-\hat F_0(s))\|_{\ell^p_{\tau,T}L^q} \dd s + \int_0^T|\hat F_0(s)| \dd s \\
		&\lesssim_{p,q,T} \int_0^T \|e^{-\iu s|\nabla|}|\nabla|^{-1}(F(s)-\hat F_0(s))\|_{H^1} \dd s + \int_0^T\|F(s)\|_{L^2} \dd s \\
		&\lesssim \|F\|_{L^1_TL^2}. \qedhere
	\end{align*}
\end{proof}

It is often convenient to work with the wave equation in first-order formulation. We therefore define the operator
\begin{equation}\label{DefA}
 A \coloneqq \begin{pmatrix}0&I\\\Delta & 0\end{pmatrix},
\end{equation}
which maps continuously $H^{r+1} \times H^{r} \to H^r \times H^{r-1}$ and generates the strongly continuous group of operators
\[e^{tA} =  \begin{pmatrix} \cos(t|\nabla|) & t \sinc(t|\nabla|) \\ -|\nabla|\sin(t|\nabla|) & \cos(t|\nabla|) \end{pmatrix} \]
on $H^r \times H^{r-1}$, for all $r \in \R$.

\begin{Corollary}\label{CorDisStrichWave}
	Let $f \in H^1$, $g \in L^2$, and $F \in \ell^1L^2$. For $\tau \in (0,1]$ and $n \in \N$, we define
	\[W_n \coloneqq e^{n\tau A}(f,g) + \tau\sum_{k=0}^ne^{(n-k)\tau A}\begin{pmatrix}0\\F_k\end{pmatrix}.  \]
	Let $w_n$ be the first component of $W_n$. For $T \ge 0$ and $H^1$-admissible $(p,q)$ we then get the estimate
	\[ \|\pi_{\tau^{-1}}w_n \|_{\ell^p_{\tau,T}L^q} \lesssim_{p,q,T} \|f\|_{H^1} + \|g\|_{L^2} + \|F\|_{\ell^1_{\tau,T}L^2}.\]
\end{Corollary}
\begin{proof}
	The estimate for the term containing $(f,g)$ is already contained in Corollary \ref{CorStrichWave}. The term containing $F$ is treated in the same manner as the corresponding term in Corollary \ref{CorStrichWave}.
\end{proof}

\begin{Corollary}\label{CorStrichDualA}
	Let $(p,q)$ be $H^1$-admissible and $T>0$. Then we have the estimates
	\begin{align*}
		\Big\| \int_0^T  e^{-sA}\begin{pmatrix}0\\F(s)\end{pmatrix} \dd s \Big\|_{L^2\times H^{-1}} &\lesssim_{p,q,T}  \|F\|_{L^{p'}_TL^{q'}}, \\
		\Big\| \int_0^T  e^{-sA}\begin{pmatrix}0\\\pi_KG(s)\end{pmatrix} \dd s \Big\|_{L^2\times H^{-1}} &\lesssim_{T} (1+\log K)^{\frac12} \|G\|_{L^2_TL^1}, 
	\end{align*}
	for all $F \in L^{p'}_TL^{q'}$, $G \in L^2_TL^1$ and $K \ge 1$.
\end{Corollary}
\begin{proof}
	These estimates follow from the dual versions from Corollary \ref{CorKontStrich} for $\gamma=1$, which are given by
	\begin{align}
		\Big\|\int_0^T e^{-\iu s|\nabla|}F(s) \dd s\Big\|_{H^{-1}} &\lesssim_{p,q,T} \|F\|_{L^{p'}_TL^{q'}}, \nonumber \\
		\label{eq:DualStrich}\Big\|\int_0^T \pi_Ke^{-\iu s|\nabla|}G(s) \dd s\Big\|_{H^{-1}} &\lesssim_{T} (1+\log K)^\frac12\|G\|_{L^2_TL^1}. 
	\end{align}
	We give the details for the term containing $G$. We split
	\begin{align*}
		&\Big\| \int_0^T  e^{-sA}\begin{pmatrix}0\\\pi_KG(s)\end{pmatrix} \dd s \Big\|_{L^2\times H^{-1}} \\
		 &\lesssim \Big\| \int_0^T \pi_K s\sinc(-s|\nabla|)G(s) \dd s \Big\|_{L^2} + \Big\| \int_0^T \pi_K \cos(-s|\nabla|)G(s) \dd s \Big\|_{H^{-1}}.
	\end{align*}
	The cosine term is estimated directly using \eqref{eq:DualStrich}, and for the sine term we compute
	\begin{align*}
		&\Big\| \int_0^T \pi_K s\sinc(-s|\nabla|)G(s) \dd s \Big\|_{L^2} \\
		&\lesssim \Big\| \int_0^T \pi_K e^{\pm \iu s |\nabla|}(G(s)-\hat G_0(s)) \dd s \Big\|_{H^{-1}} + \Big| \int_0^T s\hat G_0(s) \dd s \Big| \\
		&\lesssim_{T} (1+\log K)^\frac12 \|G-\hat G_0\|_{L^2_TL^1} + \|G\|_{L^1_TL^1} \lesssim_T (1+\log K)^\frac12 \|G\|_{L^2_TL^1}, 
	\end{align*}
	using \eqref{eq:DualStrich} and $|\hat G_0(s)| \le \|G(s)\|_{L^1}$.
\end{proof}

\subsection{Some properties of the filter operator $\pi_K$}

We will need the following Bernstein-type estimates.

\begin{Lemma}\label{LemBernstein}
	Let $r \in \R$, $s \ge 0$ and $q \in [2,\infty)$. We then have the estimates
	\begin{align*}
		\|\pi_K h\|_{H^{s+r}} &\lesssim K^s \|h\|_{H^r}, \\
		\|\pi_K f\|_{L^q} &\lesssim_q K^{3(\frac12-\frac1q)} \|f\|_{L^2},
	\end{align*}
	for all $K \ge 1$, $f \in L^2$, and $h \in H^r$.
\end{Lemma}
\begin{proof}
	The first estimate follows directly from the representation of the Sobolev norm \eqref{DefSobNormTorus}, and the second one uses the first one combined with Sobolev embedding.
\end{proof}

The following lemma quantifies the convergence $\pi_K \to I$ as $K \to \infty$, and will be used to control the error terms that arise from the insertion of the filter into the numerical scheme \eqref{Strang}.

\begin{Lemma}\label{LemPiNew}
	For all $K \ge 1$ and $s>0$, we can write
	\begin{align*}
	I-\pi_{K} = (K^{-1}|\nabla|\phi_K)^s,
	\end{align*}
	for an operator $\phi_K$ that is bounded uniformly in $K$ on all Sobolev spaces $H^r$, $r \in \R$. We moreover get the estimate
	\begin{equation*}
		\|(I-\pi_K)f\|_{ H^{\gamma}} \le K^{\gamma-r}\|f\|_{ H^{r}} 
	\end{equation*}
	for all $K \ge 1$, $r \in \R$, $\gamma \le r$, and $f \in  H^r$. 
\end{Lemma}
\begin{proof}
	We work in Fourier space. Let $\phi_K$ be the Fourier multiplier for the function $\mathbbm1_{\{|k|_\infty>K\}}K/|k|$, which is bounded by $1$. The estimate then follows from
	\[ \|(I-\pi_K)f\|_{ H^{\gamma}} = K^{\gamma-r}\|(|\nabla|\phi_K)^{r-\gamma}f\|_{ H^{\gamma}} \le K^{\gamma-r}\|f\|_{ H^{r}}. \qedhere \]
\end{proof}

The next lemma will be crucially exploited in the error analysis of \eqref{Strang} when using the summation by parts formula {to exploit cancellations in the error terms}. This strategy is inspired by \cite{Averaged}, cf.\ Property (OF4) in Theorem 3.14 there. Roughly speaking, the idea is the following. 
From basic semigroup theory, for $v \in H^r \times H^{r-1}$, the integral \[\int_0^T e^{tA}v \dd t \in H^{r+1} \times H^r\] is an element of the domain of $A$; and \[A \int_0^T e^{tA}v \dd t = (e^{TA}-I)v.\]
We would like to exploit something similar in the discrete setting, namely, that \[\tau A \sum_{n=0}^{N-1} e^{n\tau A}\] is a bounded operator on $H^r \times H^{r-1}$, uniformly in $\tau \in (0,1]$ and $N \in \N$ with $N\tau \le T$.
If we formally insert the geometric sum formula, we obtain
\[\tau A \sum_{n=0}^{N-1} e^{n\tau A} = \tau A \frac{e^{N\tau A}-I}{e^{\tau A}-I}.\]
But this does not lead anywhere since the operator $e^{\tau A}-I$ might not be invertible for certain ``resonant'' step-sizes $\tau$. However, if we introduce the filter operator
\[\Pi_{\tau^{-1}} \coloneqq \operatorname{diag}(\pi_{\tau^{-1}},\pi_{\tau^{-1}})\]
and apply the assertion of the following Lemma \ref{LemCancel}, we get
\begin{equation}\label{eq:cancellation}
\Pi_{\tau^{-1}}\tau A \sum_{n=0}^{N-1} e^{n\tau A} = \Psi_\tau (e^{N\tau A}-I),
\end{equation}
which indeed is a bounded operator on $H^r \times H^{r-1}$ as desired.

\begin{Lemma}\label{LemCancel}
	For all $\tau \in (0,1]$, we can write
	\[
		\tau A\Pi_{\tau^{-1}} = (e^{\tau A}-I)\Psi_\tau,
	\]
	where the operator $\Psi_\tau$ is bounded uniformly in $\tau$ on all Sobolev spaces $H^r \times H^{r-1}$, $r \in \R$.
\end{Lemma}
\begin{proof}
	One checks that the equality holds for 
	\[\Psi_\tau \coloneqq -\frac\tau2\Pi_{\tau^{-1}}\begin{pmatrix}
		\dfrac{|\nabla|\sin(\tau|\nabla|)}{\cos(\tau|\nabla|)-I} & I \\[2ex] \Delta & \dfrac{|\nabla|\sin(\tau|\nabla|)}{\cos(\tau|\nabla|)-I}
	\end{pmatrix}. \]
	This operator is uniformly bounded in $\tau$ thanks to the presence of $\Pi_{\tau^{-1}}$, which ensures that we only need to consider the Fourier modes with $\tau|k|\le \sqrt3\tau|k|_\infty \le \sqrt3$. Therefore, we can exploit that the function \[x \mapsto \frac{x\sin x} {\cos x -1}\] is bounded on $[0,\sqrt3]$.
\end{proof}

\section{Nonlinear wave equation}\label{SecNonlinear}

\subsection{Review of local wellposedness theory}
The wellposedness theory for \eqref{NLW2} is well known, therefore we only address the most important points. See, e.g., the monographs \cite{Tao2006,Sogge} for more details.\footnote{We also refer to \cite{PlanchonUniqueness} for the question of ``unconditional uniqueness''.} Note that, thanks to finite propagation speed, the local theory is essentially identical to that of the corresponding problem on the full space $\R^3$. We first reformulate the equation \eqref{NLW2} as a first-order system in time. Using the wave operator $A$ from \eqref{DefA} and the notation
\[g(u) \coloneqq -\mu u^\alpha, \qquad G(u,v) \coloneqq (0,g(u)) \]
for the nonlinearity, one obtains the equivalent system
\begin{equation}
	\begin{aligned} \label{NLW} \partial_{t} U(t) &= AU(t) +G(U(t)), \quad t \in [0,T], \\ U(0)&=(u^0,v^0)
\end{aligned}\end{equation}
for the new variable $U \mathrel{\widehat{=}} (u,\partial_tu)$. The local wellposedness is shown by a classical fixed point argument based on the Duhamel formula
\begin{equation}\label{Duh}
	U(t) = e^{t A }(u^0,v^0)+\int_0^t e^{(t-s)A}G(U(s)) \dd s
\end{equation}
for \eqref{NLW}. In the case $\alpha = 3$, the Sobolev embedding $H^1 \hookrightarrow L^6$ implies that the nonlinearity $G$ leaves the space $H^1 \times L^2$ invariant. Therefore, the fixed point space for $U$ can be chosen as a closed ball in $C([0,b],H^1\times L^2)$ for some $b>0$ small enough. If $\alpha=4$, one needs to involve a Strichartz space for $u$ in the fixed point space. A particular choice that fits well to the the estimate from Corollary \ref{CorDisStrichWave} is the space $L^6([0,b],L^9)$. 
This leads to the following local wellposedness theorem for the nonlinear wave equation \eqref{NLW2}.

\begin{Proposition}\label{ThmLokWoh}
	Let $\alpha<5$ and $R>0$. Then there exists a time $b = b(R)>0$ such that for all $(u^0,v^0) \in H^1 \times L^2$ with $\|(u^0,v^0)\|_{H^1 \times L^2} \le R$, there is a unique function {$U=(u,\partial_tu) \in C([0,b],H^1 \times L^2)$ satisfying \eqref{Duh}.}
	For $\alpha=4$, we moreover get $u \in L^6([0,b],L^9)$ with estimate
	\begin{equation}\label{eq:apriori}
		\|u\|_{L^6_bL^9} \lesssim R.
	\end{equation}
\end{Proposition}

\begin{Remark}\label{RemWellp}
{a) The function $U = (u,\partial_tu)$ from Proposition \ref{ThmLokWoh} solves the semilinear wave equation \eqref{NLW2} and the equivalent formulation \eqref{NLW} in the following sense. Since $g(u)$, $\Delta u \in C([0,b], H^{-1})$, the Duhamel formula \eqref{Duh} implies that $\partial_t^2 u$ belongs to $C([0,b], H^{-1})$ and that the differential equation \eqref{NLW2} holds in this space. Thus, the equation \eqref{NLW} then also holds in $C([0,b],L^2 \times H^{-1})$.} \smallskip
	
b) The local wellposedness theory is slightly different in the case of the critical power $\alpha=5$. More precisely, the a-priori existence time $b>0$ then not only depends on the $H^1 \times L^2$ norm of the initial data $(u^0,v^0)$, but also on the initial data itself. \smallskip

c) If $\alpha \in \{3,5\}$ and $\mu=1$, one can exploit the energy conservation law to show that solutions to \eqref{NLW2} are in fact global in time. This does not work if $\mu =-1$ or $\alpha$ is even, since in those cases the energy might become negative and does not control the $H^1 \times L^2$ norm.  
However, in this work we will not address questions of long-time behavior.
\end{Remark}

From now on we will always assume the existence of a solution on a fixed interval $[0,T]$.

\begin{Assumption}\label{AssNew}
	There exists a time $T>0$ and a solution $U = (u,\partial_tu)$ of the nonlinear equation \eqref{NLW2} with $\alpha \in \{3,4\}$ such that
	{$U \in C([0,T], H^1 \times L^2)$}
	with bound
	\begin{equation}\label{DefMNew}
		M \coloneqq \|U\|_{L^\infty_T( H^1 \times L^2)}.
	\end{equation}
\end{Assumption}

\subsection{Nonlinear estimates}
We derive some important estimates for $u$ from Assumption \ref{AssNew} that will be used later. First, we extend the a-priori estimate \eqref{eq:apriori} to other $H^1$-admissible Strichartz pairs $(p,q)$, to the possibly larger time interval $[0,T]$, and also to a discrete-time estimate.

\begin{Proposition}\label{PropStrichUNew}
	Let $u$, $T$, and $M$ be given by Assumption \ref{AssNew} and let $(p,q)$ be $H^1$-admissible. Then we have the estimates
	\begin{align*}
		\|u\|_{L^p_TL^q} + \|\pi_K u\|_{L^p_TL^q} +\|\pi_{\tau^{-1}} u(n\tau)\|_{\ell^p_{\tau}([0,T],L^q)} &\lesssim_{p,q,M,T} 1,  \\
		\|\pi_K u\|_{L^2_TL^\infty} &\lesssim_{M,T} (1 +\log K)^{\frac12}
	\end{align*}
	for all $\tau \in (0,1]$ and $K \ge 1$. 
\end{Proposition}
\begin{proof}
	We only need to show that
	\begin{equation}\label{eq:g(u)}
	 \|g(u)\|_{L^1_TL^2} \lesssim_{M,T} 1, 
	\end{equation}
	then the result follows from Corollary \ref{CorStrichWave}. In the case $\alpha=3$, the bound \eqref{eq:g(u)} already follows from the Sobolev embedding $H^1 \hookrightarrow L^6$. If $\alpha=4$, we take $b = b(M)$ from {Proposition} \ref{ThmLokWoh}. Let $L \in \N$ be the minimal number such that $jb\ge T$. The bound \eqref{eq:apriori} then implies that
	\[ \|u\|_{L^6_TL^9} \le \Big(\sum_{j=1}^L \int_{(j-1)b}^{jb} \|u(t)\|_{L^9}^6\dd t\Big)^\frac16 \lesssim L^\frac16M \lesssim_{M,T} 1. \] 
	Sobolev and Hölder inequalities now yield that
	\[\|g(u)\|_{L^1_TL^2} \le \|u^3\|_{L^1_TL^3}\|u\|_{L^\infty_TL^6} \lesssim_T \|u\|^3_{L^6_TL^9} \|u\|_{L^\infty_TH^1} \lesssim_{M,T}1, \]
	thus \eqref{eq:g(u)} is also true for $\alpha=4$.
\end{proof}

In the next lemma we give convergence rates for the difference between $g(u)$ and $g(\pi_Ku)$. We will often use the elementary Lipschitz bound
\begin{equation}\label{NichtlinLip}
	|g(v)-g(w)| \lesssim (|v|^{\alpha-1}+|w|^{\alpha-1})|v-w|
\end{equation}
for the nonlinearity $g$.

\begin{Lemma}\label{LemStrichg(u)}
	Let $u$, $T$, and $M$ be given by Assumption \ref{AssNew}. Then we have the estimate
	\begin{align*}
		\|g(u)-g(\pi_Ku)\|_{L^1_TH^{-1}} &\lesssim_{M,T} K^{-1}.
	\end{align*}
	Moreover, if $\alpha=3$,
	\[\|g(u)-g(\pi_Ku)\|_{L^1_TL^2} \lesssim_{M,T} K^{-1}(1+\log K), \]
	and for $\alpha =4$,
	\[\|g(u)-g(\pi_Ku)\|_{L^1_TL^2} \lesssim_{M,T} K^{-\frac12}.\]
	These estimates are uniform in $K \ge 1$.
\end{Lemma}
\begin{proof}
	In the general case, we estimate
	\begin{align*}
		&\|g(u)-g(\pi_Ku)\|_{L^1_TH^{-1}} \\
		&\lesssim \|(|u|^{\alpha-1}+|\pi_Ku|^{\alpha-1})(I-\pi_K)u\|_{L^1_TL^\frac{6}{5}} \\
		&\lesssim  (\|u\|^{\alpha-1}_{L^{\alpha-1}_TL^{3(\alpha-1)}}+\|\pi_Ku\|^{\alpha-1}_{L^{\alpha-1}_TL^{3(\alpha-1)}}) \|(I-\pi_K)u\|_{L^\infty_T L^2} \\
		&\lesssim_{M,T} K^{-1}\|u\|_{L^\infty_T H^1} \lesssim_M K^{-1},
	\end{align*}
	using the Sobolev embedding $L^\frac65 \hookrightarrow H^{-1}$, estimate \eqref{NichtlinLip}, Hölder's inequality, Proposition \ref{PropStrichUNew}, and Lemma \ref{LemPiNew}. 
	Similarly, for $\alpha=4$,
	\begin{align*}
		\|g(u)-g(\pi_Ku)\|_{L^1_TL^2} &\lesssim \|(|u|^3+|\pi_Ku|^3)(I-\pi_K)u\|_{L^1_TL^2} \\
		&\lesssim  \|(|u|^3+|\pi_Ku|^3)\|_{L^1_TL^6}\|(I-\pi_K)u\|_{L^\infty_T L^3} \\
		&\lesssim  (\|u\|^3_{L^3_TL^{18}}+\|\pi_Ku\|^3_{L^3_TL^{18}})\|(I-\pi_K)u\|_{L^\infty_T H^{\frac12}} \\
		&\lesssim_{M,T} K^{-\frac12}\|u\|_{L^\infty_T H^1} \lesssim_M K^{-\frac12},
	\end{align*}
	where the Sobolev embedding $H^\frac12 \hookrightarrow L^3$ and Proposition \ref{PropStrichUNew} with $(p,q) = (3,18)$ were used.
	Let now $\alpha=3$. Then we decompose
	\[ \|g(u)-g(\pi_Ku)\|_{L^1_TL^2} \le \|g(u)-g(\pi_{K^2}u)\|_{L^1_TL^2} + \|g(\pi_{K^2}u)-g(\pi_Ku)\|_{L^1_TL^2}. \]
	We obtain similar as before
	\begin{align*}
		\|g(u)-g(\pi_{K^2}u)\|_{L^1_TL^2} &\lesssim \|(|u|^2+|\pi_{K^2}u|^2)(I-\pi_{K^2})u\|_{L^1_TL^2} \\
		&\lesssim  \|(|u|^2+|\pi_{K^2}u|^2)\|_{L^1_TL^6}\|(I-\pi_{K^2})u\|_{L^\infty_T L^3} \\
		&\lesssim_T  (\|u\|^2_{L^4_TL^{12}}+\|\pi_{K^2}u\|^2_{L^4_TL^{12}})\|(I-\pi_{K^2})u\|_{L^\infty_T H^\frac12} \\
		&\lesssim_{M,T} K^{-1} \|u\|_{L^\infty_T H^1} \lesssim_M K^{-1},
	\end{align*}
	using $(p,q)=(4,12)$, and
	\begin{align*}
		&\|g(\pi_{K^2}u)-g(\pi_Ku)\|_{L^1_TL^2} \\
		&\lesssim \|(|\pi_{K^2}u|^2+|\pi_Ku|^2)\pi_{K^2}(I-\pi_K)u\|_{L^1_TL^2} \\
		&\lesssim  \|(|\pi_{K^2}u|^2+|\pi_Ku|^2)\|_{L^1_TL^\infty}\|\pi_{K^2}(I-\pi_K)u\|_{L^\infty_T L^2} \\
		&\lesssim  (\|\pi_{K^2}u\|^2_{L^2_TL^\infty}+\|\pi_Ku\|^2_{L^2_TL^\infty})\|(I-\pi_K)u\|_{L^\infty_T L^2} \\
		&\lesssim_{M,T} K^{-1}(1+\log K) \|u\|_{L^\infty_T H^1} \lesssim_M K^{-1}(1+\log K).
	\end{align*}
	Here, the logarithmic estimate for the $L^2_TL^\infty$ norm from Proposition \ref{PropStrichUNew} was applied.
\end{proof}

\section{Error analysis of the semi-discretization in time}\label{SecErrorAnalysis}

We now start with the error analysis of the splitting scheme \eqref{Strang}. Note that since $G(u,v) = (0,g(u))$ we have $G(\Pi_{\tau^{-1}}U_{n+1/2}) = G(\Pi_{\tau^{-1}}U_{n+1})$, thus we can also state the scheme in the more compact form
\begin{equation}\label{StrangVar}\begin{aligned}
		U_{n+1} &= e^{\tau A}U_n + \frac{\tau}{2}\Big(e^{\tau A} G(\Pi_{\tau^{-1}}U_n) + G(\Pi_{\tau^{-1}}U_{n+1})\Big).
	\end{aligned}
\end{equation} 
In view of a later iteration argument, we allow here for general initial values $U_0 \neq U(0)$. We often denote the discrete times by $t_n \coloneqq n\tau$. 

\subsection{Error recursion}
We first establish a useful decomposition of the error. 

\begin{Proposition}\label{PropFehlerRekStrangNew}
Let $U=(u,\partial_tu)$ be given from Assumption \ref{AssNew} and $U_n$ be given by \eqref{StrangVar}. Define the error $E_n$ by \begin{equation}\label{DefErrorStrang}
	E_n \coloneqq U(t_n) - U_n.
\end{equation}
We then have
\begin{align}\label{FehlerRekStrangNew}
	E_{n} &= e^{n\tau A}E_0 +  B(n\tau) +  D_n + Q_n
\end{align}
for all $\tau \in (0,1]$, and $n \in \N_0$ with $t_{n} \le T$. The appearing terms are given by 
\begin{align}\label{DefTermeStrangNew}
	B(t) &\coloneqq \int_0^{t} e^{(t-s)A}[G(U(s))-G(\Pi_{\tau^{-1}} U(s))] \dd s,\nonumber\\
	D_n &\coloneqq \tfrac{\tau^2}{2}\int_0^{t_n} e^{(n\tau-s)A}(\lfloor\tfrac s\tau\rfloor-\tfrac s\tau)(\lceil\tfrac s\tau\rceil-\tfrac s\tau)\begin{pmatrix}d_1(s) \\ d_2(s)+d_3(s)+d_4(s)\end{pmatrix} \dd s,\\
	Q_n &\coloneqq \tau\sum_{k=0}^{n} e^{(n-k)\tau A} {c_{k,n}}[G(\Pi_{\tau^{-1}} U(t_k))-G(\Pi_{\tau^{-1}}U_k)] , \nonumber
\end{align}
and
\begin{equation}\label{DefdStrangNew}
\begin{aligned}
	d_1(t) &\coloneqq -2g'(\pi_{\tau^{-1}}u(t))\pi_{\tau^{-1}}\partial_tu(t), \\ 
	d_2(t) &\coloneqq g''(\pi_{\tau^{-1}}u(t))\big[|\nabla \pi_{\tau^{-1}}u(t)|^2 + (\pi_{\tau^{-1}}\partial_tu(t))^2\big], \\
	d_3(t) &\coloneqq g'(\pi_{\tau^{-1}}u(t))\pi_{\tau^{-1}}g(u(t)), \\
	d_4(t) &\coloneqq 2g'(\pi_{\tau^{-1}}u(t))\pi_{\tau^{-1}}\Delta u(t).
\end{aligned}
\end{equation}
{Here we define 
	\begin{equation*}\label{Defcnk}
		c_{0,0} \coloneqq 0, \quad c_{0,n}= c_{n,n} \coloneqq \frac12,\quad c_{k,n} \coloneqq 1
	\end{equation*}
	for $k \in \{1,\dots,n-1\}$ and $n \ge 1$.}
We can alternatively write
\begin{equation}\label{eq:StrangDAlt}
	D_n = \tfrac{\tau^2}{2} \int_0^\tau \tfrac s \tau (\tfrac s \tau -1) \sum_{k=0}^{n-1} e^{((n-k)\tau-s) A} \begin{pmatrix}d_1 \\ d_2+d_3+d_4\end{pmatrix}(t_k+s) \dd s.
\end{equation}
\end{Proposition}
\begin{proof}
A straightforward induction based on \eqref{StrangVar} shows that the numerical solution satisfies the discrete Duhamel formula
\begin{equation}\label{DiskrDuhamelStrang}
	U_{n} = e^{n\tau A}U_0 + \tau \sum_{k=0}^{n} {c_{k,n}} e^{(n-k)\tau A}G(\Pi_{\tau^{-1}}U_{k}).
\end{equation}
We subtract it from its continuous analogue (see \eqref{Duh})
\begin{equation*}
	U(n\tau) = e^{n\tau A}U(0)+\int_0^{n\tau}e^{(n\tau-s)A}G(U(s)) \dd s
\end{equation*} 
to obtain
\begin{align*}
	E_{n} &= e^{n\tau A}E_0 +  B(n\tau) + \int_0^{t_n}e^{(n\tau-s)A} G(\Pi_{\tau^{-1}}U(s)) \dd s \\ 
	&\quad - \tau \sum_{k=0}^{n} {c_{k,n}} e^{(n-k)\tau A}G(\Pi_{\tau^{-1}}U(t_k)) + Q_n. 
\end{align*}
To get the desired formula for $D_n$, we use the error representation of the trapezoidal sum
\[\int_0^{t_n} F(s) \dd s - \tau\sum_{k=0}^{n}{c_{k,n}}F(t_k) = \frac12\sum_{k=0}^{n-1}\int_{t_k}^{t_{k+1}}(s-t_k)(s-t_{k+1})F''(s) \dd s, \]
where we set $F(s) \coloneqq e^{(n\tau-s)A}G(\Pi_{\tau^{-1}}U(s))$. We compute 
\begin{align*}
	F'(s) &=  e^{(n\tau-s) A}\Bigg[-\begin{pmatrix}0 & I \\ \Delta & 0 \end{pmatrix}\begin{pmatrix}0 \\ g(\pi_{\tau^{-1}} u(s))\end{pmatrix}+\frac{\mathrm{d}}{\mathrm{d}s}\begin{pmatrix}0 \\ g(\pi_{\tau^{-1}} u(s))\end{pmatrix}\Bigg] \\
	&=  e^{(n\tau-s) A}\begin{pmatrix}-g(\pi_{\tau^{-1}} u(s)) \\ g'(\pi_{\tau^{-1}} u(s))\pi_{\tau^{-1}} \partial_t u(s)\end{pmatrix}
\end{align*}
and
\begin{align*}
F''(s) &=  e^{(n\tau-s) A}\Bigg[-\begin{pmatrix}0 & I \\ \Delta & 0 \end{pmatrix}\begin{pmatrix}-g(\pi_{\tau^{-1}} u(s)) \\ g'(\pi_{\tau^{-1}} u(s))\pi_{\tau^{-1}} \partial_t u(s)\end{pmatrix} \\
&\qquad \qquad \quad \  +\frac{\mathrm{d}}{\mathrm{d}s}\begin{pmatrix}-g(\pi_{\tau^{-1}} u(s)) \\ g'(\pi_{\tau^{-1}} u(s))\pi_{\tau^{-1}} \partial_t u(s)\end{pmatrix}\Bigg] \\
&= e^{(n\tau-s) A}\begin{pmatrix}d_1(s)\\d_2(s)+d_3(s)+d_4(s)\end{pmatrix}, 
\end{align*}
using that $\Delta[g(w)] = g''(w)|\nabla w|^2 + g'(w) \Delta w$ and the differential equation \eqref{NLW2}. Since 
\begin{align*}
&\frac12\sum_{k=0}^{n-1}\int_{t_k}^{t_{k+1}}(s-t_k)(s-t_{k+1})F''(s) \dd s \\
&= \tfrac{\tau^2}{2}\int_{0}^{t_n}(\lfloor\tfrac s\tau\rfloor-\tfrac s\tau)(\lceil\tfrac s\tau\rceil-\tfrac s\tau)F''(s) \dd s = D_n, 
\end{align*}
we can conclude that formula \eqref{FehlerRekStrangNew} is true. The substitution $\tilde s = s - t_k$ yields the alternative representation \eqref{eq:StrangDAlt}.
\end{proof}

\begin{Remark}\label{Rem:en}
	{The first component of $ Q_n$ defined by \eqref{DefTermeStrangNew} satisfies
	\[ [ Q_n]_1 = \tau\sum_{k=0}^{n-1} \Big[e^{(n-k)\tau A} c_{k,n}[G(\Pi_{\tau^{-1}} U(t_k))-G(\Pi_{\tau^{-1}}U_k)]\Big]_1. \]
	Here, the notation $[\cdot]_1$ means that we take the first component of the vector.
	The $n$-th term in the sum vanishes since the first component of the nonlinearity $G$ is zero.}
\end{Remark}

\subsection{Estimates for error terms resulting from the filter}

We now deal with the term $B$ in \eqref{FehlerRekStrangNew} that results from the introduction of the filter function $\Pi_{\tau^{-1}}$. Here we face the following difficulty. If we move the $L^2 \times H^{-1}$ norm inside the integral and apply Lemma \ref{LemPiNew}, we end up with a term roughly of the form
\[ \tau^2 \|g'(u)\Delta u\|_{L^1_TH^{-1}}. \]
Now we would like to use a nonlinear product estimate, but we do not have enough regularity available to obtain an optimal error bound. For example, consider $\alpha=3$ so that $g'(u) \approx u^2$. From Assumption \ref{AssNew} we get $u \in L^\infty_TH^1$ and $\Delta u \in L^\infty_T H^{-1}$. Moreover, thanks to Proposition \ref{PropStrichUNew} we can exploit that almost $u \in L^2_TL^\infty$. But a product estimate of the form $\|vw\|_{H^{-1}} \lesssim \|v\|_{H^1 \cap L^\infty}\|w\|_{H^{-1}}$ is wrong because in 3D, in general one only has \[\|vw\|_{H^{-1}} \lesssim \|v\|_{W^{1,3} \cap L^\infty}\|w\|_{H^{-1}},\]
which would require additional integrability. 

To solve this problem, we follow a different strategy. We do not move the $L^2 \times H^{-1}$ norm into the integral at first. Instead, we involve integration by parts in time, which helps to ``move regularity to the right position''. 
This technique was used previously in, e.g., \cite{Averaged} in a context without Strichartz estimates.

\begin{Remark}\label{RemIBPSpace}
	An alternative approach to the above problem is based on the use of integration by parts in space instead of time. {We do not pursue this approach here since it would be less flexible in view of possible generalizations to non-periodic boundary conditions or $L^\infty$-coefficients in the nonlinear part as in \cite{Averaged}. }
\end{Remark}

\begin{Lemma}\label{LemBPI}
Let $U = (u,\partial_tu)$, $T$, and $M$ be given by Assumption \ref{AssNew}. Let $B$ be given by \eqref{DefTermeStrangNew}. We then have the following estimates. If $\alpha=3$,
\begin{align*} 
\|B(t)\|_{H^1 \times L^2} &\lesssim_{M,T} \tau(1+|\log\tau|), \\
\|B(t)\|_{L^2 \times H^{-1}} &\lesssim_{M,T} \tau^2(1+|\log\tau|),  
\end{align*}
and for $\alpha=4$,
\begin{align*} 
	\|B(t)\|_{H^1 \times L^2} &\lesssim_{M,T} \tau^\frac12, \\
	\|B(t)\|_{L^2 \times H^{-1}} &\lesssim_{M,T} \tau^\frac32,
\end{align*}
uniformly in $\tau \in (0,1]$ and $t \in [0,T]$.
\end{Lemma}
\begin{proof}
Since
\[
	\|B(t)\|_{H^1 \times L^2} \lesssim_T \|g(u)-g(\pi_{\tau^{-1}}u)\|_{L^1_TL^2},
\]
the bounds for the energy norm follow directly from Lemma \ref{LemStrichg(u)} with $K = \tau^{-1}$.
For the estimates in the weaker $L^2 \times H^{-1}$-norm, we use a decomposition. We first split\footnote{Similar as in the proof of Lemma \ref{LemStrichg(u)}, this first decomposition is in principle only necessary for $\alpha=3$.}
\begin{align*}
	B(t) &= \int_0^{t} e^{(t-s)A}[G(U(s))-G(\Pi_{\tau^{-2}} U(s)) ]\dd s \\
	&\quad + \int_0^{t} e^{(t-s)A}[G(\Pi_{\tau^{-2}}U(s))-G(\Pi_{\tau^{-1}} U(s))] \dd s \\
	& \eqqcolon I_1+I_2. 
\end{align*}
Using again Lemma \ref{LemStrichg(u)},
\begin{equation*}
	\|I_1\|_{L^2\times H^{-1}} \lesssim_T \|g(u)-g(\pi_{\tau^{-2}}u)\|_{L^1_TH^{-1}} \lesssim_{M,T} \tau^2.
\end{equation*}
For the second term, we write
\begin{align*}
G(\Pi_{\tau^{-2}}U(s))-G(\Pi_{\tau^{-1}} U(s)) = \int_0^1  G'(U_{\tau,\theta}(s))\Pi_{\tau^{-2}}(I-\Pi_{\tau^{-1}})U(s) \dd \theta.
\end{align*}
Here we use the notation
\[
 U_{\tau,\theta}(s) \coloneqq \theta\pi_{\tau^{-2}}U(s) + (1-\theta)\pi_{\tau^{-1}}U(s)
\]
and write $u_{\tau,\theta}$ for the first component of $U_{\tau,\theta}$ so that
\[
G'(U_{\tau,\theta}(s)) =  \begin{pmatrix}
	0&0 \\ g'(u_{\tau,\theta}(s))&0
\end{pmatrix}.
\]
Moreover, in order to gain a power of $\tau$, we insert the equality
\begin{align*}
	(I-\Pi_{\tau^{-1}}) = (\tau\mathcal D\Phi_\tau)^{\frac{7-\alpha}{2}},
\end{align*}
where $\mathcal D \coloneqq \operatorname{diag}(|\nabla|,|\nabla|)$, and $\Phi_\tau \coloneqq\operatorname{diag}(\phi_{\tau^{-1}},\phi_{\tau^{-1}})$ is given by Lemma \ref{LemPiNew}. This leads to the representation
\begin{align*}
	I_2 = \tau^{\frac{7-\alpha}{2}}\int_0^1 \int_0^{t}e^{(t-s)A} G'(U_{\tau,\theta}(s)) \Pi_{\tau^{-2}}\Phi_\tau\mathcal D^{\frac{7-\alpha}{2}}U(s) \dd s \dd \theta.
\end{align*}
Next we observe that $\mathcal D = JA$ for the operator
\[J \coloneqq \begin{pmatrix}
	0&-|\nabla|^{-1} \\ |\nabla|&0
\end{pmatrix}. \]
Here we define the zero-th Fourier coefficient of $|\nabla|^{-1}f$ to be zero, for arbitrary functions $f \in H^r$.
To simplify notation, we set $\widetilde \Phi_\tau \coloneqq \Pi_{\tau^{-2}}\Phi_\tau J$, which is a bounded operator on $H^r \times H^{r-1}$ for all $r \in \R$, uniformly in $\tau \in [0,1)$. Altogether we derive
\begin{align*}
	I_2 = \tau^{\frac{7-\alpha}{2}}\int_0^1 \int_0^{t}e^{(t-s)A} G'(U_{\tau,\theta}(s)) \widetilde \Phi_\tau\mathcal D^{\frac{5-\alpha}{2}}AU(s) \dd s \dd \theta.
\end{align*}
Using the differential equation $AU = \partial_{t} U - G(U)$ from \eqref{NLW} that holds in $C([0,T],L^2 \times H^{-1})$, we split this term again into
\begin{align*}
	I_2 &= \tau^{\frac{7-\alpha}{2}} \int_0^1 \int_0^{t}e^{(t-s)A} G'(U_{\tau,\theta}(s))\widetilde \Phi_\tau\mathcal D^{\frac{5-\alpha}{2}}\partial_{t} U(s)  \dd s \dd \theta \\
	&\quad - \tau^{\frac{7-\alpha}{2}} \int_0^1 \int_0^{t}e^{(t-s)A}G'(U_{\tau,\theta}(s)) \widetilde \Phi_\tau\mathcal D^{\frac{5-\alpha}{2}}G(U(s)) \dd s \dd \theta \\
	&\eqqcolon \tau^{\frac{7-\alpha}{2}}(I_{2,1}-I_{2,2}).
\end{align*}
The term involving $G(U)$ is estimated using Hölder's inequality with $\frac56 = \frac13 + \frac12$ and the Sobolev embedding $L^{\frac6\alpha} \hookrightarrow H^{\frac{3-\alpha}{2}}$ by
\begin{align*}
	&\|I_{2,2}\|_{L^2 \times H^{-1}} \lesssim_T \sup_{\theta \in [0,1]} \|g'(u_{\tau,\theta}) [\widetilde \Phi_\tau\mathcal D^{\frac{5-\alpha}{2}}G(U)]_1\|_{L^1_TL^{\frac65}} \\
	&\le \sup_{\theta \in [0,1]} \|g'(u_{\tau,\theta})\|_{L^1_TL^3}\| |\nabla|^{\frac{5-\alpha}{2}}g(u)\|_{L^\infty_TH^{-1}} \\
	&\lesssim \||\pi_{\tau^{-2}}u|^{\alpha-1}+|\pi_{\tau^{-1}}u|^{\alpha-1}\|_{L^1_T L^3} \|u^\alpha\|_{L^\infty_TL^{\frac6\alpha}} \\
	&\lesssim \Big(\|\pi_{\tau^{-2}}u\|^{\alpha-1}_{L^{\alpha-1}L^{3(\alpha-1)}}+\|\pi_{\tau^{-1}}u\|^{\alpha-1}_{L^{\alpha-1}L^{3(\alpha-1)}}\Big)\|u\|_{L^\infty_TL^6}^\alpha \\
	&\lesssim_{M,T} 1,
\end{align*} 
where the estimate in the last line follows from Lemma \ref{LemStrichg(u)}. Here, the notation $[\cdot]_1$ means that we take the first component of the vector.

The term involving $\partial_{t}U$ is integrated by parts in time, which gives
\begin{align*}
	I_{2,1} &= \int_0^1 \Big[e^{(t-s)A} G'(U_{\tau,\theta}(s)) \widetilde \Phi_\tau\mathcal D^{\frac{5-\alpha}{2}}U(s)\Big]_{s=0}^{t} \dd \theta \\
	&\quad + \int_0^1\int_0^{t}Ae^{(t-s)A} G'(U_{\tau,\theta}(s)) \widetilde \Phi_\tau\mathcal D^{\frac{5-\alpha}{2}}U(s) \dd s \dd \theta\\
	&\quad - \int_0^1 \int_0^{t}e^{(t-s)A} \frac{\mathrm{d}}{\mathrm{d}s} G'(U_{\tau,\theta}(s)) \widetilde \Phi_\tau\mathcal D^{\frac{5-\alpha}{2}}U(s) \dd s \dd \theta \\
	&\eqqcolon I_{2,1,1}+I_{2,1,2}-I_{2,1,3}.
\end{align*}
The boundary terms $I_{2,1,1}$ can be estimated only using Sobolev and Hölder inequalities. For $\theta \in [0,1]$ and $s \in \{0,t\}$, we get
\begin{align*}
	\|I_{2,1,1}\|_{L^2\times H^{-1}} &\lesssim_T \|g'( u_{\tau,\theta}(s))[\mathcal D^{\frac{5-\alpha}{2}}\widetilde \Phi_\tau U(s)]_1\|_{L^{\frac{6}{5}}} \\
	&\le \|g'( u_{\tau,\theta}(s))\|_{L^{\frac{6}{\alpha-1}}} \|[\mathcal D^{\frac{5-\alpha}{2}}\widetilde \Phi_\tau U(s)]_1\|_{L^{\frac{6}{6-\alpha}}} \\
	&\lesssim \Big(\|\pi_{\tau^{-2}}u(s)\|_{L^6}^{\alpha-1}+\|\pi_{\tau^{-1}}u(s)\|_{L^6}^{\alpha-1}\Big)\|U(s)\|_{H^1\times L^2} \lesssim_M 1,
\end{align*} 
using $H^{\frac{\alpha-3}{2}} \hookrightarrow L^{\frac{6}{6-\alpha}}$.
The term involving $A$ is estimated by
\begin{align}\label{I232}
	\|I_{2,1,2}\|_{L^2 \times H^{-1}}	&\lesssim_T \|g'(u_{\tau,\theta})[\mathcal D^{\frac{5-\alpha}{2}}\widetilde \Phi_\tau U]_1 \|_{L^1_TL^2} \nonumber\\
	&\le \|g'(u_{\tau,\theta})\|_{L^1_TL^{\frac{6}{\alpha-3}}}\|[\mathcal D^{\frac{5-\alpha}{2}}\widetilde \Phi_\tau U]_1 \|_{L^\infty_TL^{\frac{6}{6-\alpha}}} \nonumber\\
	&\lesssim_M \|\pi_{\tau^{-2}}u\|_{L^{\alpha-1}_TL^{\frac{6(\alpha-1)}{\alpha-3}}}^{\alpha-1}+\|\pi_{\tau^{-1}}u\|_{L^{\alpha-1}_TL^{\frac{6(\alpha-1)}{\alpha-3}}}^{\alpha-1}.
\end{align}
If $\alpha=4$, {the quantity \eqref{I232} contains the} $L^3_TL^{18}$ norm which by Proposition \ref{PropStrichUNew} is uniformly bounded by a constant depending on $M$ and $T$. Thus, in this case, $\|I_{2,1,2}\|_{L^2 \times H^{-1}} \lesssim_{M,T} 1$. If $\alpha=3$ we instead use the logarithmic endpoint estimate from Proposition \ref{PropStrichUNew} for the $L^2_TL^\infty$ norm {to bound the term \eqref{I232}}, which gives $\|I_{2,1,2}\|_{L^2 \times H^{-1}} \lesssim_{M,T} 1+|\log\tau|$. Finally, to get the estimate for $I_{2,1,3}$, we observe that
\begin{align*}
	\frac{\mathrm{d}}{\mathrm{d}s} g'(u_{\tau,\theta}(s)) &= g''(u_{\tau,\theta}(s))\partial_su_{\tau,\theta}(s) \\
	&= g''(u_{\tau,\theta}(s))(\theta\pi_{\tau^{-2}}\partial_tu(s) + (1-\theta)\pi_{\tau^{-1}}\partial_tu(s)).
\end{align*}
If $\alpha=4$, we use the dual Strichartz estimate from Corollary \ref{CorStrichDualA} with $(p,q)=(3,18)$ to obtain
\begin{align*}
	\|I_{2,1,3}\|_{L^2 \times H^{-1}} &\lesssim_T \| g''(u_{\tau,\theta})\partial_su_{\tau,\theta}[\mathcal D^{\frac{1}{2}}\widetilde \Phi_\tau U]_1\|_{L^\frac32_TL^{\frac{18}{17}}} \\
	&\le \| g''(u_{\tau,\theta})\|_{L^\frac32_TL^9} \|\partial_su_{\tau,\theta}\|_{L^\infty_TL^2} \|[\mathcal D^{\frac{1}{2}}\widetilde \Phi_\tau U]_1\|_{L^\infty_TL^3} \\	
	&\lesssim \Big(\|\pi_{\tau^{-2}}u\|^2_{L^3_TL^{18}}+\|\pi_{\tau^{-1}}u\|^2_{L^3_TL^{18}} \Big)\|\partial_tu\|_{L^\infty_T L^2} \|U\|_{L^\infty_T(H^1\times L^2)} \\
	&\lesssim_{M,T} 1.
\end{align*}
In the case $\alpha=3$, we exploit that the polynomial $g(u) = -\mu u^3$ keeps the frequency localization $\pi_{\tau^{-2}}$ up to a factor $3$. This means that $I_{2,1,3} = \Pi_{3\tau^{-2}}I_{2,1,3}$.\footnote{In view of Remark \ref{RemGeneralAlpha}, this argument could be avoided by involving another triangle inequality with $\Pi_{\tau^{-2}}I_{2,1,3}$, for instance.} Hence, we can apply the dual endpoint logarithmic Strichartz estimate from Corollary \ref{CorStrichDualA} to conclude
\begin{align*}
	&\|\Pi_{3\tau^{-2}}I_{2,1,3}\|_{L^2 \times H^{-1}} \\
	&\lesssim_T (1+|\log\tau|)^{\frac12} \| g''(u_{\tau,\theta})\partial_su_{\tau,\theta}[\mathcal D\widetilde \Phi_\tau U]_1\|_{L^2_TL^1} \\
	&\le (1+|\log\tau|)^{\frac12} \| g''(u_{\tau,\theta})\|_{L^2_TL^\infty}\|\partial_su_{\tau,\theta}\|_{L^\infty_TL^2}\|[\mathcal D\widetilde \Phi_\tau U]_1\|_{L^\infty_TL^2} \\
	&\lesssim (1+|\log\tau|)^{\frac12} \Big(\|\pi_{\tau^{-2}}u\|_{L^2_TL^\infty}+\|\pi_{\tau^{-1}}u\|_{L^2_TL^\infty} \Big)\|\partial_tu\|_{L^\infty_T L^2} \|U\|_{L^\infty_T(H^1\times L^2)} \\
	&\lesssim_{M,T} (1+|\log\tau|),
\end{align*}
again in the end using the logarithmic endpoint estimate for $u$ from Proposition \ref{PropStrichUNew}.
\end{proof}

\subsection{Estimates for local error terms}

Next, we treat the term $D_n$ from \eqref{DefTermeStrangNew} that includes the local error terms. 

\begin{Lemma}\label{Lemd}
Let $U = (u,\partial_tu)$, $T$, and $M$ be given by Assumption \ref{AssNew}. For the terms defined in \eqref{DefdStrangNew}, we have the estimate
\[ \|d_3\|_{L^1_TH^{-1}} \lesssim_{M,T} 1. \]
Moreover, if $\alpha=3$,
\begin{align*}
	\|d_1\|_{L^1_TL^2} &\lesssim_{M,T} (1+|\log\tau|), \\
	\|d_2\|_{L^2_TL^1} &\lesssim_{M,T} (1+|\log\tau|)^{\frac12},
\end{align*}
and for $\alpha=4$,
\begin{align*}
	\|d_1\|_{L^1_TL^2} + \|d_2\|_{L^\frac32_TL^{\frac{18}{17}}} \lesssim_{M,T} \tau^{-\frac12}.
\end{align*}
All estimates are uniform in $\tau \in (0,1]$. 
\end{Lemma}
\begin{proof}
First consider
\begin{align*}
	\|d_3\|_{L^1_TH^{-1}} &\lesssim \|g'(\pi_{\tau^{-1}}u)\pi_{\tau^{-1}}g(u)\|_{L^1_TL^\frac65} \lesssim \|g'(\pi_{\tau^{-1}}u)\|_{L^2_TL^3}\|g(u)\|_{L^2_TL^2} \\
	&\lesssim \|\pi_{\tau^{-1}}u\|^{\alpha-1}_{L^{2(\alpha-1)}_TL^{3(\alpha-1)}}\|u\|_{L^{2\alpha}_TL^{2\alpha}}^\alpha \lesssim_{M,T} 1,
\end{align*}
where for $\alpha=3$ it is enough to use Sobolev embedding, and for $\alpha=4$ we use Proposition \ref{PropStrichUNew} and the $H^1$ admissibility of $(p,q)=(6,9)$ and $(8,8)$.
Let now $\alpha=3$, we then derive
\begin{align*}
	\|d_1\|_{L^1_TL^2} &\lesssim \|g'(\pi_{\tau^{-1}}u)\pi_{\tau^{-1}}\partial_tu\|_{L^1_TL^2} \le \|g'(\pi_{\tau^{-1}}u)\|_{L^1_TL^\infty}\|\pi_{\tau^{-1}}\partial_tu\|_{L^\infty_TL^2}\\ 
	&\lesssim_M \|\pi_{\tau^{-1}}u\|^2_{L^2_TL^\infty} \lesssim_{M,T} 1+|\log\tau|
\end{align*}
and
\begin{align*}
	\|d_2\|_{L^2_TL^1} &\lesssim \|g''(\pi_{\tau^{-1}}u)\big[|\nabla \pi_{\tau^{-1}}u|^2 + (\pi_{\tau^{-1}}\partial_tu)^2\big]\|_{L^2_TL^1} \\
	&\le \|g''(\pi_{\tau^{-1}}u)\|_{L^2_TL^\infty} \||\nabla \pi_{\tau^{-1}}u|^2 + (\pi_{\tau^{-1}}\partial_tu)^2\|_{L^\infty_TL^1}\\ 
	&\lesssim \|\pi_{\tau^{-1}}u\|_{L^2_TL^\infty}\Big(\||\nabla u|\|^2_{L^\infty_TL^2}+\|\partial_tu\|^2_{L^\infty_TL^2}\Big)
	\lesssim_{M,T} (1+|\log\tau|)^\frac12,
\end{align*}
using the logarithmic endpoint estimate from Proposition \ref{PropStrichUNew}. Similarly, for $\alpha=4$,
\begin{align*}
	\|d_1\|_{L^1_TL^2} &\lesssim  \|g'(\pi_{\tau^{-1}}u)\|_{L^1_TL^6}\|\pi_{\tau^{-1}}\partial_tu\|_{L^\infty_TL^3}\\ 
	&\lesssim \|\pi_{\tau^{-1}}u\|^3_{L^3_TL^{18}}\tau^{-\frac12}\|\partial_tu\|_{L^\infty_TL^2} \lesssim_{M,T} \tau^{-\frac12}
\end{align*}
and
\begin{align*}
	\|d_2\|_{L^\frac32_TL^{\frac{18}{17}}} &\lesssim  \|g''(\pi_{\tau^{-1}}u)\|_{L^{\frac32}_TL^9} \||\nabla \pi_{\tau^{-1}}u|^2 + (\pi_{\tau^{-1}}\partial_tu)^2\|_{L^\infty_TL^{\frac65}}\\ 
	&\lesssim \|\pi_{\tau^{-1}}u\|^2_{L^3_TL^{18}}\Big(\||\nabla \pi_{\tau^{-1}}u|\|^2_{L^\infty_TL^{\frac{12}{5}}}+\|\pi_{\tau^{-1}}\partial_tu\|^2_{L^\infty_TL^{\frac{12}{5}}}\Big) \\
	&\lesssim_{M,T} \tau^{-\frac12}\Big(\||\nabla u|\|^2_{L^\infty_TL^{2}}+\|\partial_tu\|^2_{L^\infty_TL^2}\Big) \lesssim_M \tau^{-\frac12},
\end{align*}
employing Proposition \ref{PropStrichUNew} with $(p,q)=(3,18)$. The loss of $\tau^{-1/2}$ comes from the Bernstein inequality from Lemma \ref{LemBernstein}.
\end{proof}

The term $d_4$ from \eqref{DefdStrangNew} is the most difficult because it involves second partial derivatives of $u$. Therefore, we follow the same strategy as in Lemma \ref{LemBPI}. However, since the situation is now more ``discrete'' in time, we apply summation by parts instead of integration by parts. Roughly speaking, this transforms the term containing $d_4$ into terms that can be estimated in the same way as $d_1$ and $d_2$ in Lemma \ref{Lemd}. To use summation by parts, we need the filter $\Pi_{\tau^{-1}}$, cf.\ the discussion before Lemma \ref{LemCancel}. Such a strategy was again already used in \cite{Averaged} in a situation without Strichartz estimates.

\begin{Lemma}\label{LemLokFStrangNew}
Let $U = (u,\partial_tu)$, $T$, and $M$ be given by Assumption \ref{AssNew}. Let $D_n$ be given by \eqref{DefTermeStrangNew}. We then have the following estimates. If $\alpha=3$,
\begin{align*} 
	\|D_n\|_{H^1 \times L^2} \lesssim_{M,T} \tau(1+|\log\tau|), \\
	\|D_n\|_{L^2 \times H^{-1}} \lesssim_{M,T} \tau^2(1+|\log\tau|),  
\end{align*}
and for $\alpha=4$,
\begin{align*} 
	\|D_n\|_{H^1 \times L^2} \lesssim_{M,T} \tau^{\frac12}, \\
	\|D_n\|_{L^2 \times H^{-1}} \lesssim_{M,T} \tau^{\frac32},  
\end{align*}
uniformly in $\tau \in (0,1]$ and $n \in \N_0$ with $n\tau \le T$.
\end{Lemma}
\begin{proof}
We first note that it suffices to show the bounds in the $L^2\times H^{-1}$ norm, since by frequency localization and the Bernstein inequality from Lemma \ref{LemBernstein} we have\footnote{One could also prove the bounds for the energy norm directly by employing a first-order representation of the quadrature error in the proof of Proposition \ref{PropFehlerRekStrangNew}.}
\[\|D_n\|_{H^1\times L^2} = \|\Pi_{\alpha\tau^{-1}}D_n\|_{H^1\times L^2} \lesssim \tau^{-1} \|D_n\|_{L^2\times H^{-1}}. \]
Using Corollary \ref{CorStrichDualA} for the term involving $d_2$, we start with the estimate
\begin{align*}
	\|D_n\|_{L^2\times H^{-1}} &\lesssim_T \tau^2\Big( \|d_1\|_{L^1_TL^2} + \tilde d_2 + \|d_3\|_{L^1_TH^{-1}} +\|D_{n,4}\|_{L^2 \times H^{-1}}),
\end{align*}
where $\tilde d_2 \coloneqq (1+|\log\tau|)^{1/2}\|d_2\|_{L^2_TL^1}$ for $\alpha=3$ and $\tilde d_2 \coloneqq \|d_2\|_{L^{3/2}_TL^{{18}/{17}}}$ for $\alpha=4$, as well as
\[D_{n,4} \coloneqq \int_0^{t_n} e^{(n\tau-s)A}(\lfloor\tfrac s\tau\rfloor-\tfrac s\tau)(\lceil\tfrac s\tau\rceil-\tfrac s\tau) \begin{pmatrix}0 \\ d_4(s)\end{pmatrix}\dd s  \]
independent of $\alpha$.
The terms containing $d_1$, $d_2$, and $d_3$ are estimated by Lemma \ref{Lemd}.
We still need to deal with the term $D_{n,4}$. As in \eqref{eq:StrangDAlt}, we write
\begin{align*}
	D_{n,4} = \int_0^\tau \tfrac s \tau (\tfrac s \tau -1) \sum_{k=0}^{n-1} e^{((n-k)\tau-s) A} \begin{pmatrix}0 \\ d_4(t_k+s)\end{pmatrix} \dd s. 
\end{align*}
Moreover,
\begin{align*}
	\begin{pmatrix}0 \\ d_4(t_k+s)\end{pmatrix}  = 2 G'(\Pi_{\tau^{-1}}U(t_k+s))A^2\Pi_{\tau^{-1}}U(t_k+s),
\end{align*}
where
\[G'(\Pi_{\tau^{-1}}U(t_k+s)) = \begin{pmatrix}
	0&0 \\ g'(\pi_{\tau^{-1}}u(t_k+s))&0
\end{pmatrix}.\]
Thus,
\[D_{n,4} = 2\int_0^\tau \tfrac s \tau (\tfrac s \tau -1) \sum_{k=0}^{n-1} e^{((n-k)\tau-s) A} G'(\Pi_{\tau^{-1}}U(t_k+s))A^2\Pi_{\tau^{-1}}U(t_k+s) \dd s. \]
We now apply the summation by parts formula 
\[ \sum_{k=0}^{n-1}a_kb_k = a_{n-1}b_{n-1} + a_{n-1}\sum_{k=0}^{n-2}b_k + \sum_{k=0}^{n-2} (a_k-a_{k+1})\sum_{j=0}^k b_j\]
with $a_k=e^{(n-k)\tau A}G'(\Pi_{\tau^{-1}}U(t_k+s))$ and $b_k=A^2\Pi_{\tau^{-1}}U(t_k+s)$. 
This yields 
\[D_{n,4} = 2(I_1+I_2+I_3) \]
with
\begin{align*}
	I_1 &\coloneqq \int_0^\tau \tfrac s \tau (\tfrac s \tau -1) e^{(\tau-s) A}G'(\Pi_{\tau^{-1}}U(t_{n-1}+s))A^2\Pi_{\tau^{-1}}U(t_{n-1}+s) \dd s, \\
	I_2 &\coloneqq \int_0^\tau \tfrac s \tau (\tfrac s \tau -1) e^{(\tau-s) A}G'(\Pi_{\tau^{-1}}U(t_{n-1}+s))\sum_{k=0}^{n-2}A^2\Pi_{\tau^{-1}}U(t_{k}+s) \dd s, \\
	I_3 &\coloneqq \int_0^\tau \tfrac s \tau (\tfrac s \tau -1)\sum_{k=0}^{n-2}e^{((n-k)\tau-s)A}\Big(G'(\Pi_{\tau^{-1}}U(t_{k}+s)) \\
	&\qquad \qquad \qquad-e^{-\tau A}G'(\Pi_{\tau^{-1}}U(t_{k+1}+s))\Big)\sum_{j=0}^{k}A^2\Pi_{\tau^{-1}}U(t_{j}+s) \dd s.
\end{align*}
Next, we insert the equality
\[\tau A \Pi_{\tau^{-1}} = (e^{\tau A}-I)\Psi_\tau, \]
from Lemma \ref{LemCancel}, where the operator $\Psi_\tau$
is bounded on $H^r \times H^{r-1}$ for all $r \in \R$, uniformly in $\tau \in (0,1]$. 
The term $I_1$ is bounded by Sobolev, Hölder, and Bernstein inequalities via
\begin{align*}
	\|I_1\|_{L^2 \times H^{-1}} &\lesssim \sup_{s \in[0,\tau]}\|g'(\pi_{\tau^{-1}}u(t_{n-1}+s))[(e^{\tau A}-I)\Psi_\tau AU(t_{n-1}+s)]_1\|_{L^{\frac65}} \\
	&\lesssim  \sup_{s \in[0,\tau]}\|g'(\pi_{\tau^{-1}}u(t_{n-1}+s))\|_{L^3}\|AU(t_{n-1}+s)\|_{L^2\times H^{-1}} \\
	&\lesssim_M \|\pi_{\tau^{-1}}u\|_{L^\infty_T L^{3(\alpha-1)}}^{\alpha-1} \lesssim \tau^{-\frac{\alpha-3}{2}}\|u\|_{L^\infty_T H^1}^{\alpha-1} \lesssim_M \tau^{-\frac{\alpha-3}{2}}.
\end{align*}
{As our next step}, for $j \in \{0,\dots,n-2\}$ and $s \in [0,\tau]$ we define the sum
\[ S(\tau,j,s) \coloneqq \tau\sum_{k=0}^{j}A^2\Pi_{\tau^{-1}}U(t_{k}+s) = \sum_{k=0}^{j}(e^{\tau A}-I)\Psi_{\tau}AU(t_{k}+s). \]
{Note that a rough estimate of $S(\tau,j,s)$ in the $L^2 \times H^{-1}$ norm using the triangle inequality would either produce a loss of $\tau^{-1}$, destroying the optimal convergence rate, or require $H^2$ regularity of $u(t)$, which is not covered by our Assumption \ref{AssNew}. Therefore, we have to treat this term with more care.}
A shifted version of Duhamel's formula \eqref{Duh} yields
\begin{align*}
	S(\tau,j,s) &= \sum_{k=0}^{j}\Psi_\tau A[U(t_{k+1}+s)-U(t_k+s) ] \\
	& \quad- \sum_{k=0}^j \int_0^\tau \Psi_\tau Ae^{(\tau-\sigma)A}G(U(t_k+s+\sigma)) \dd \sigma.
\end{align*}
We can exploit this telescopic sum to conclude that
\begin{align}\label{eq:telsum}
	\|S(\tau,j,s)\|_{L^2 \times H^{-1}} &\lesssim_T \|U(t_{j+1}+s)-U(s)\|_{H^1 \times L^2} \nonumber\\
	&\quad + \sum_{k=0}^j \int_0^\tau\|g(u(t_k+s+\sigma))\|_{L^2} \dd \sigma \nonumber\\ 
	&\lesssim \|U\|_{L^\infty_T(H^1 \times L^2)} + \|g(u)\|_{L^1_TL^2} \lesssim_{M,T} 1,
\end{align}
uniformly in $j \in \{0,\dots,n-2\}$, $\tau \in (0,1]$ and $s \in [0,\tau]$, where also \eqref{eq:g(u)} was used. {This is the cancellation that was already discussed in \eqref{eq:cancellation} in a simplified form.}
Hence, we can bound the term $I_2$ similar as $I_1$ by
\begin{align*}
	\|I_2\|_{L^2 \times H^{-1}} &\lesssim \sup_{s \in[0,\tau]}\|g'(\pi_{\tau^{-1}}u(t_{n-1}+s))[S(\tau,n-2,s)]_1\|_{L^{\frac65}} \\
	&\lesssim  \sup_{s \in[0,\tau]}\|g'(\pi_{\tau^{-1}}u(t_{n-1}+s))\|_{L^3}\|S(\tau,n-2,s)\|_{L^2\times H^{-1}} \\
	&\lesssim_{M,T}\tau^{-\frac{\alpha-3}{2}}.
\end{align*}
For the term $I_3$, we need another decomposition. We split it as
\[I_3 = I_{3,1} + I_{3,2}, \]
where
\begin{align*}
	I_{3,1} &\coloneqq \tfrac1\tau\int_0^\tau \tfrac s \tau (\tfrac s \tau -1)\sum_{k=0}^{n-2}e^{((n-k)\tau-s)A}(I-e^{-\tau A})G'(\Pi_{\tau^{-1}}U(t_{k}+s))\\
	&\quad \cdot S(\tau,k,s) \dd s, \\
	I_{3,2} &\coloneqq \tfrac1\tau\int_0^\tau \tfrac s \tau (\tfrac s \tau -1)\sum_{k=0}^{n-2}e^{((n-k-1)\tau-s)A} \\
	&\quad\ \cdot\Big(G'(\Pi_{\tau^{-1}}U(t_{k}+s))-G'(\Pi_{\tau^{-1}}U(t_{k+1}+s))\Big) S(\tau,k,s)\dd s.
\end{align*}
For $I_{3,1}$, observe that
\[ I-e^{-\tau A} = \tau A \varphi_1(-\tau A),\]
where $\varphi_1(z) \coloneqq (e^z-1)/z$. Since the function $\varphi_1$ is bounded on $\iu\R$, the operator $\varphi_1(-\tau A)$ is bonded on $H^1 \times L^2$, uniformly in $\tau \in (0,1]$. 
Hence, for $\alpha=3$, we derive
\begin{align*}
	\|I_{3,1}\|_{L^2 \times H^{-1}} &\lesssim_T \int_0^\tau \sum_{k=0}^{n-2}\|G'(\Pi_{\tau^{-1}}U(t_{k}+s))S(\tau,k,s)\|_{H^1 \times L^2} \dd s \\
	&\le \int_0^\tau \sum_{k=0}^{n-2}\|g'(\pi_{\tau^{-1}}u(t_{k}+s))\|_{L^\infty}\|S(\tau,k,s)\|_{L^2 \times H^{-1}} \dd s \\
	&\lesssim_{M,T} \|\pi_{\tau^{-1}}u\|_{L^2_TL^\infty}^2 \lesssim_{M,T} 1+|\log\tau|,
\end{align*}
using estimate \eqref{eq:telsum} and the logarithmic endpoint estimate for $u$ from Proposition \ref{PropStrichUNew}. Similarly, for $\alpha=4$, one obtains
\begin{align*}
	\|I_{3,1}\|_{L^2 \times H^{-1}} &\lesssim_T \int_0^\tau \sum_{k=0}^{n-2}\|g'(\pi_{\tau^{-1}}u(t_{k}+s))\|_{L^6}\|[S(\tau,k,s)]_1\|_{L^3} \dd s \\
	&\lesssim \|\pi_{\tau^{-1}}u\|_{L^3_TL^{18}}^3\tau^{-\frac12} \sup_{\substack{k \in \{0,\dots,n-2\}\\ s\in[0,\tau] }}\|S(\tau,k,s)\|_{L^2 \times H^{-1}} \lesssim_{M,T} \tau^{-\frac12}
\end{align*}
using Proposition \ref{PropStrichUNew} and the Bernstein inequality.
For the term $I_{3,2}$, we first substitute $\tilde s = s+t_k$ to get
\begin{align*}
	I_{3,2} &= \tfrac1\tau\int_0^{t_{n-1}} e^{((n-1)\tau-s)A} (\tfrac s\tau-\lfloor\tfrac s\tau\rfloor)(\tfrac s\tau-\lceil\tfrac s\tau\rceil)\\
	&\quad \cdot \Big(G'(\Pi_{\tau^{-1}}U(s))-G'(\Pi_{\tau^{-1}}U(\tau+s))\Big) S(\tau,\lfloor\tfrac s\tau\rfloor,s-\tau\lfloor\tfrac s\tau\rfloor)\dd s .
\end{align*}
We first consider the case $\alpha=3$. By the polynomial structure of $g$, we have the frequency localization $I_{3,2} = \Pi_{3\tau^{-1}}I_{3,2}$. We can thus apply the dual logarithmic endpoint Strichartz estimate from Corollary \ref{CorStrichDualA} to obtain
\begin{align}\label{eq:I_32}
	\|I_{3,2}\|_{L^2 \times H^{-1}} &\lesssim_T \tfrac1\tau \Big((1+|\log\tau|) \int_0^{t_{n-1}}\big\|\big(g'(\pi_{\tau^{-1}}u(s))-g'(\pi_{\tau^{-1}}u(\tau+s))\big) \nonumber \\
	&\hspace{4.5cm} \cdot [S(\tau,\lfloor\tfrac s\tau\rfloor,s-\tau\lfloor\tfrac s\tau\rfloor)]_1\big\|_{L^1}^2 \dd s\Big)^{\frac12} \nonumber \\
	&\lesssim_{M,T} \tfrac1\tau(1+|\log\tau|)^{\frac12} \|g'(\pi_{\tau^{-1}}u)-g'(\pi_{\tau^{-1}}u(\tau+\cdot))\|_{L^2_{t_{n-1}}L^2},
\end{align}
also using the bound \eqref{eq:telsum} for $S$. To conclude,
the equation
\[ u(s) - u(\tau+s) = -\int_0^\tau \partial_t u(s+\sigma) \dd \sigma  \]
implies
\begin{align*}
	&\|g'(\pi_{\tau^{-1}}u)-g'(\pi_{\tau^{-1}}u(\tau+\cdot))\|_{L^2_{t_{n-1}}L^2} \\
	&\lesssim \||g''(\pi_{\tau^{-1}}u)|+ |g''(\pi_{\tau^{-1}}u(\tau+\cdot))|\|_{L^2_{t_{n-1}}L^\infty}\sup_{s\in[0,t_{n-1}]}\int_0^\tau \|\partial_t u(s+\sigma)\|_{L^2} \dd \sigma \\
	&\lesssim_{M,T} \|\pi_{\tau^{-1}}u\|_{L^2_TL^\infty}\tau\|\partial_tu\|_{L^\infty_TL^2} \lesssim_{M,T} (1+|\log\tau|)^{\frac12}\tau. 
\end{align*}
Together with \eqref{eq:I_32}, this implies the desired bound for $I_{3,2}$.

If $\alpha=4$, we follow a similar strategy to obtain
\begin{align*}
	\|I_{3,2}\|_{L^2 \times H^{-1}} &\lesssim_T \tfrac1\tau \Big( \int_0^{t_{n-1}}\big\|\big(g'(\pi_{\tau^{-1}}u(s))-g'(\pi_{\tau^{-1}}u(\tau+s))\big) \nonumber \\
	&\hspace{4cm} \cdot [S(\tau,\lfloor\tfrac s\tau\rfloor,s-\tau\lfloor\tfrac s\tau\rfloor)]_1\big\|_{L^{\frac{18}{17}}}^{\frac32} \dd s\Big)^{\frac23} \nonumber \\
	&\lesssim_{M,T} \tfrac1\tau \|g'(\pi_{\tau^{-1}}u)-g'(\pi_{\tau^{-1}}u(\tau+\cdot))\|_{L^{\frac32}_{t_{n-1}}L^{\frac94}} \\
	&\lesssim \|g''(\pi_{\tau^{-1}}u)\|_{L^{\frac32}_TL^9} \|\pi_{\tau^{-1}}\partial_t u\|_{L^\infty_TL^3}\\
	&\lesssim_{M,T}\tau^{-\frac12} \|\pi_{\tau^{-1}}u\|^2_{L^3_TL^{18}}\|\partial_tu\|_{L^\infty_TL^2} \lesssim_{M,T} \tau^{-\frac12},
\end{align*}
which concludes the proof.
\end{proof}

\subsection{Proof of the global error bounds for $\alpha=3$}
We give different proofs of the global error bounds depending on $\alpha \in \{3,4\}$, since the proof for $\alpha=3$ is somewhat simpler and does not use discrete-time Strichartz estimates. We still need to deal with the term $Q_n$ from \eqref{FehlerRekStrangNew}. For $\alpha=3$ it turns out that it is enough to use Sobolev and Hölder inequalities. 
We write $u_n$ for the first component of $U_n$, as well as $e_n$ for the first component of $E_n$. 

\begin{Lemma}\label{LemStabStrang3}
Let $U = (u,\partial_tu)$, $T$, and $M$ be given by Assumption \ref{AssNew} with $\alpha=3$. Define the error $E_n$ by \eqref{Strang} and \eqref{DefErrorStrang}. We then have the estimates
\begin{align*} 
	\|g(\pi_{\tau^{-1}}u(t_n))-g(\pi_{\tau^{-1}}u_n)\|_{L^2} &\lesssim_{M}\Big(1+{\|e_{n}\|_{H^1}^2}\Big)\|e_n\|_{H^1}, \\
	\|g(\pi_{\tau^{-1}}u(t_n))-g(\pi_{\tau^{-1}}u_n)\|_{H^{-1}} &\lesssim_{M}\Big(1+{\|e_{n}\|_{H^1}^2}\Big)\|e_n\|_{L^2}
\end{align*}
for all $\tau \in (0,1]$ and $n \in \N$ with $n\tau \le T$.
\end{Lemma}
\begin{proof}
From \eqref{NichtlinLip} combined with Sobolev and Hölder inequalities, we deduce
\begin{align*}
\|g(\pi_{\tau^{-1}}u(t_n))-g(\pi_{\tau^{-1}}u_n)\|_{L^2} &\lesssim_M \Big(1+\|u_n\|_{H^1}^2\Big)\|e_n\|_{H^1}, \\ 
\|g(\pi_{\tau^{-1}}u(t_n))-g(\pi_{\tau^{-1}}u_n)\|_{H^{-1}} &\lesssim_M \Big(1+\|u_n\|_{H^1}^2\Big)\|e_n\|_{L^2}.
\end{align*}
{The assertion follows from this by inserting $u_n = u(t_n)-e_n$.}
\end{proof}

We can now give the proof of the global error bound for $\alpha=3$. We use a standard procedure based on the discrete Gronwall inequality. The error bound for the $H^1 \times L^2$ norm is inductively exploited to get a uniform control on the numerical solution $U_n$ in $H^1 \times L^2$, which is also essential to obtain the error bound in the $L^2 \times H^{-1}$ norm. This strategy goes back to \cite{Lubich}.

\begin{proof}[Proof of Theorem \ref{ThmStrang} for $\alpha=3$.]
We apply Lemmas \ref{LemBPI}, \ref{LemLokFStrangNew} and \ref{LemStabStrang3} to the formula \eqref{FehlerRekStrangNew}, which gives
\begin{align}\label{eq:StrangIterationH1en}
	{\|e_n\|_{ H^1}} &\lesssim {\|B(n\tau)\|_{H^1 \times L^2} + \| D_n\|_{ H^1 \times L^2}  + \|[ Q_n]_1\|_{H^1}} \nonumber\\
	&\lesssim_{M,T} {\tau|\log\tau| + \tau\sum_{k=1}^{n-1}\Big(1+\|e_k\|_{H^1}^2\Big)\|e_k\|_{ H^1},}\\
	\label{eq:StrangIterationH1}
	\|E_n\|_{H^1 \times L^2} &\lesssim \|B(n\tau)\|_{ H^1 \times L^2} + \| D_n\|_{H^1 \times L^2} +   \|Q_n\|_{H^1 \times L^2} \nonumber\\
	&\lesssim_{M,T} \tau|\log\tau| +  \tau\sum_{k=1}^{n}\Big(1+{\|e_k\|_{ H^1}^2}\Big)\|e_k\|_{ H^1}
\end{align}
and similarly
\begin{align}\label{eq:StrangIterationL2}
	\|E_n\|_{L^2 \times H^{-1}} &\lesssim \|B(n\tau)\|_{L^2 \times H^{-1}} + \| D_n\|_{L^2 \times H^{-1}}+ \|Q_n\|_{L^2 \times H^{-1}} \nonumber\\
	&\lesssim_{M,T} \tau^2|\log\tau| + \tau\sum_{k=1}^{n}\Big(1+{\|e_k\|_{ H^1}^2}\Big)\|e_k\|_{L^2}
\end{align}
for all $\tau\in(0,e^{-1}]$ and $n \in \N_0$ with $n\tau \le T$. Here {we use Remark \ref{Rem:en} for the estimate on $e_n$, and} we exploit that $E_0 = 0$. Let $c=c(M,T)>0$ be maximum of the implicit constants in {\eqref{eq:StrangIterationH1en},} \eqref{eq:StrangIterationH1}, and \eqref{eq:StrangIterationL2}. We then define {the final error constant}
\[C \coloneqq 2ce^{4cT} \]
and choose the maximum step size $\tau_0 \in (0,e^{-1}] $ satisfying
\begin{equation*}
	4c\tau_0 \le 1 , \qquad \tau_0|\log\tau_0|C \le 1 .
\end{equation*}

Let $\tau \in (0,\tau_0]$. {We first show that the bound
	\begin{equation}\label{eq:enboundalpha3}
		\|e_n\|_{H^1} \le 1 
	\end{equation}
	holds for all $n\in \N_0$ with $n\tau \le T$. We proceed by induction on $n$. For $n=0$, \eqref{eq:enboundalpha3} is clear since $e_0=0$.} Let now $n \in \N$ with $n\tau \le T$. {As induction hypothesis, we assume that we already have
	\begin{equation*}
		\|e_k\|_{H^1} \le 1
	\end{equation*}
	for all $k \in \{0,\dots,n-1\}$. 
	Inequality \eqref{eq:StrangIterationH1en} thus yields
	\[ \|e_n\|_{H^1} \le c\tau|\log\tau| + 2c\tau\sum_{k=1}^{n-1}\|e_k\|_{H^1}. \]
	The discrete Gronwall inequality then implies that
	\[ \|e_n\|_{H^1} \le ce^{2cn\tau}\tau|\log\tau| \le C\tau|\log\tau| \le 1,  \]
	using the restriction $\tau \le \tau_0$. Hence, \eqref{eq:enboundalpha3} is true.} 
{From} \eqref{eq:StrangIterationH1} we {now} infer that
\[ \|E_n\|_{H^1 \times L^2} \le c\tau|\log\tau| + 2c\tau\sum_{k=1}^{n}\|e_k\|_{H^1}. \]
Hence, using the step size restriction $2c\tau \le \frac12$, we can absorb the $n$-th term in the above sum to get
\[ \|E_n\|_{H^1 \times L^2} \le 2c\tau|\log\tau| + 4c\tau\sum_{k=1}^{n-1}\|E_k\|_{H^1 \times L^2}. \]
The discrete Gronwall inequality now implies {the desired bound}
\[ \|E_n\|_{H^1 \times L^2} \le 2ce^{4cn\tau}\tau|\log\tau| \le C\tau|\log \tau|,  \]
which concludes the proof of the bound in the energy norm. Similarly, starting from \eqref{eq:StrangIterationL2} we establish {the inequality}
\[ \|E_n\|_{L^2 \times H^{-1}}  \le C\tau^2|\log \tau|, \]
{again using \eqref{eq:enboundalpha3} and the discrete Gronwall lemma.}
\end{proof}

\subsection{Proof of the global error bounds for $\alpha=4$}

In this case, estimates in discrete Strichartz norms are needed to control the term $Q_n$ in \eqref{FehlerRekStrangNew}. The estimate for the solution $u$ is already contained in Proposition \ref{PropStrichUNew}. However, we will also need a corresponding estimate for the approximation $u_n$ which a priori is not clear. 
To this aim, we first show a ``discrete local wellposedness result'' for the scheme \eqref{Strang}. It should be compared to {Proposition} \ref{ThmLokWoh}.

\begin{Lemma}\label{LemDiskrWellp}
	Let $R>0$. Then there is a time $b_0 = b_0(R)>0$ such that for all $U_0 \in H^1 \times L^2$ with $\|U_0\|_{H^1\times L^2} \le R$, the sequence $(U_n)$ defined by \eqref{StrangVar} with $\alpha=4$ satisfies the estimate
	\[\|\pi_{\tau^{-1}}u_n\|_{\ell^6_\tau([0,b_0],L^9)} \lesssim R \]
	for all $\tau \in (0,b_0]$.
\end{Lemma}
\begin{proof}
	Let $j \in \N_0$ with $t_{j+1} \le 1$. The discrete Duhamel formula \eqref{DiskrDuhamelStrang} and the discrete Strichartz estimate from Corollary \ref{CorDisStrichWave} imply
	\begin{align}\label{eq:StrangdiskrWellp1}
		&\max\{\|U_n\|_{\ell^\infty_\tau([0,t_{j+1}],H^1\times L^2)},\|\pi_{\tau^{-1}}u_n\|_{\ell^6_\tau([0,t_{j+1}],L^9)}\} \nonumber\\
		&\lesssim \|U_0\|_{H^1\times L^2} + \|g(\pi_{\tau^{-1}}u_n)\|_{\ell^1_\tau([0,t_{j}],L^2)} + \tau\|g(\pi_{\tau^{-1}}u_{j+1})\|_{L^2}.
	\end{align}
	Using Hölder, Sobolev, and Bernstein {inequalities}, we estimate 
	\begin{align*}
		\|g(\pi_{\tau^{-1}}u_n)\|_{\ell^1_\tau([0,t_{j}],L^2)} &\le \|\pi_{\tau^{-1}}u_n\|^3_{\ell^3_\tau([0,t_{j}],L^9)}\|\pi_{\tau^{-1}}u_n\|_{\ell^\infty_\tau([0,t_{j}],L^6)} \\
		&\lesssim t_{j+1}^\frac12 \|\pi_{\tau^{-1}}u_n\|^3_{\ell^6_\tau([0,t_{j}],L^9)}\|u_n\|_{\ell^\infty_\tau([0,t_{j}],H^1)}
	\end{align*}
	as well as
	\begin{align*}
		\tau\|g(\pi_{\tau^{-1}}u_{j+1})\|_{L^2} = \tau \|\pi_{\tau^{-1}}u_{j+1}\|^4_{L^8} \lesssim \tau^\frac12\|u_{j+1}\|^4_{H^1}.
	\end{align*} 
	Similarly, the definition of the scheme \eqref{Strang} leads to
	\begin{align*}
		\|u_{j+1}\|_{H^1} \lesssim \|U_j\|_{H^1\times L^2} + \tau \|g(\pi_{\tau^{-1}}u_j)\|_{L^2} \lesssim \|U_j\|_{H^1\times L^2} + \tau^\frac12\|u_{j}\|^4_{H^1}.
	\end{align*}
	Plugging this into \eqref{eq:StrangdiskrWellp1}, we derive
	\begin{align}\label{eq:StrangdiskrWellp2}
		&\max\{\|U_n\|_{\ell^\infty_\tau([0,t_{j+1}],H^1\times L^2)},\|\pi_{\tau^{-1}}u_n\|_{\ell^6_\tau([0,t_{j+1}],L^9)}\} \nonumber\\
		&\lesssim \|U_0\|_{H^1\times L^2} + t_{j+1}^\frac12 \|\pi_{\tau^{-1}}u_n\|^3_{\ell^6_\tau([0,t_{j}],L^9)}\|u_n\|_{\ell^\infty_\tau([0,t_{j}],H^1)} \nonumber \\
		&\quad + \tau^\frac12\|U_j\|^4_{H^1\times L^2} + \tau^\frac12(\tau^\frac12\|u_{j}\|^4_{H^1})^4.
	\end{align}
	The Bernstein inequality from Lemma \ref{LemBernstein} also gives
	\begin{equation}\label{eq:StrangdiskrWellp3}
		\tau^\frac16\|\pi_{\tau^{-1}}u_0\|_{L^9} \lesssim \|u_0\|_{H^1}.
	\end{equation}
	Let $C$ be the maximum of $1$ and the implicit constants in \eqref{eq:StrangdiskrWellp2} and \eqref{eq:StrangdiskrWellp3}. We choose the time $b_0 \in (0,1]$ such that 
	\begin{equation}\label{Defb0}
	b_0^\frac12(2C)^4R^3 \le \frac13.
	\end{equation}
	We next show via induction that
	\begin{equation}\label{eq:StrangDiskrWellpInd}
		\max\{\|U_n\|_{\ell^\infty_\tau([0,t_{j}],H^1\times L^2)},\|\pi_{\tau^{-1}}u_n\|_{\ell^6_\tau([0,t_{j}],L^9)}\} \le 2CR
	\end{equation}
	for all $j \in \N_0$ with $t_j \le b_0$. For $j=0$, the claim follows from \eqref{eq:StrangdiskrWellp3}. Assume now that \eqref{eq:StrangDiskrWellpInd} holds for some $j \in \N_0$ with $t_{j+1} \le b_0$. Estimate \eqref{eq:StrangdiskrWellp2} and \eqref{Defb0} then imply
	\begin{align*}
		&\max\{\|U_n\|_{\ell^\infty_\tau([0,t_{j+1}],H^1\times L^2)},\|\pi_{\tau^{-1}}u_n\|_{\ell^6_\tau([0,t_{j+1}],L^9)}\} \nonumber\\
		&\le C[ R + b_0^\frac12 (2CR)^4+ \tau^\frac12(2CR)^4 + \tau^\frac12(\tau^\frac12(2CR)^4)^4] \le 2CR
	\end{align*}
	for all $\tau \in (0,b_0]$, which ends the proof.
\end{proof}

Using the previous lemma, we can now give an estimate for $Q_n$ on a possibly small time interval of fixed size, under the assumption that we have control on the $H^1 \times L^2$ norm of the starting value $U_0$ of the numerical scheme.

\begin{Lemma}\label{LemStabStrang4}
	Let $U$, $T$, and $M$ be given by Assumption \ref{AssNew} with $\alpha =4$. Let moreover $R>0$ and $U_0 \in H^1 \times L^2$ with $\|U_0\|_{H^1\times L^2} \le R$. Define $U_n$ by \eqref{StrangVar}, $E_n$ and $Q_n$ by Proposition \ref{PropFehlerRekStrangNew} and $b_0(R)$ by Lemma \ref{LemDiskrWellp}. Then for any time $b>0$ with $b \le \min\{b_0,T\}$, we obtain
	\begin{align*}
		\|Q_n\|_{\ell^\infty_\tau([0,b],H^1 \times L^2)} &\lesssim_{M,T,R} b^{\frac12}\|E_n\|_{\ell^\infty_\tau([0,b],H^1 \times L^2)}, \\
		\|Q_n\|_{\ell^\infty_\tau([0,b],L^2 \times H^{-1})} &\lesssim_{M,T,R} b^{\frac12}\|E_n\|_{\ell^\infty_\tau([0,b],L^2 \times H^{-1})},
	\end{align*}
	for all $\tau \in (0,b]$.
\end{Lemma}
\begin{proof}
	We estimate
	\begin{align*}
		&\|Q_n\|_{\ell^\infty_\tau([0,b],H^1 \times L^2)} \\
		&\lesssim_{T} \|g(\pi_{\tau^{-1}}u(t_n))-g(\pi_{\tau^{-1}}u_n)\|_{\ell^1_\tau([0,b ],L^2)} \\
		&\lesssim \||g'(\pi_{\tau^{-1}}u(t_n))|+|g'(\pi_{\tau^{-1}}u_n)|\|_{\ell^1_\tau([0,b ],L^3)}\|u(t_n)-u_n\|_{\ell^\infty_\tau([0,b ],L^6)} \\
		&\lesssim b^\frac12\Big(\|\pi_{\tau^{-1}}u(t_n)\|^3_{\ell^6_\tau([0,b ],L^9)}+\|\pi_{\tau^{-1}}u_n\|^3_{\ell^6_\tau([0,b ],L^9)} \Big)\|E_n\|_{\ell^\infty_\tau([0,b ],H^1\times L^2)} \\
		&\lesssim_{M,T,R} b^\frac12 \|E_n\|_{\ell^\infty_\tau([0,b ],H^1\times L^2)},
	\end{align*}
	using the estimates for the discrete Strichartz norm from Proposition \ref{PropStrichUNew} and Lemma \ref{LemDiskrWellp}. The other claim follows similarly.
\end{proof}

We now show the global error bound for $\alpha=4$. Unlike as for the case $\alpha=3$, it is not enough to use the discrete Gronwall lemma. Instead, we need to apply Lemma \ref{LemStabStrang4} iteratively on the possibly small intervals $[0,T_1]$, $[T_1, 2T_1]$ and so on, where we reach the final time $T$ after finitely many iterations. Similar as in the case $\alpha=3$, the uniform boundedness of the numerical solution in $H^1 \times L^2$ (which is needed to apply Lemma \ref{LemStabStrang4}) follows from the error bound for this norm. This method goes back to \cite{IgnatSplitting,ChoiKoh,ORS} in the context of Schrödinger equations.

\begin{proof}[Proof of Theorem \ref{ThmStrang} for $\alpha=4$.]
	We set $R \coloneqq M+1$ and define $b_0 = b_0(R)$ from Lemma \ref{LemDiskrWellp}. Formula \eqref{FehlerRekStrangNew} and Lemmas \ref{LemBPI}, \ref{LemLokFStrangNew}, and \ref{LemStabStrang4} yield
	\begin{align}\label{eq:Rekalpha4H1}
		\|E_n\|_{\ell^\infty_\tau([t_j,t_j+b],H^1\times L^2)} &\lesssim_{M,T} \|E_j\|_{H^1\times L^2} + \tau^\frac12 + b^\frac12\|E_n\|_{\ell^\infty_\tau([t_j,t_j+b],H^1\times L^2)}, \\
	\label{eq:Rekalpha4L2}
		\|E_n\|_{\ell^\infty_\tau([t_j,t_j+b],L^2 \times H^{-1})} &\lesssim_{M,T} \!  \|E_j\|_{L^2\times H^{-1}} + \tau^\frac32 + b^\frac12\|E_n\|_{\ell^\infty_\tau([t_j,t_j+b],L^2\times H^{-1})}
	\end{align} 
 	for all $j \in \N_0$ and $b \in (0,b_0]$ which satisfy $\|U_j\|_{H^1\times L^2} \le R$ and $j\tau+b \le T$.
	Let $c=c(M,T)$ be the maximum of $1$ and the implicit constants from \eqref{eq:Rekalpha4H1} and \eqref{eq:Rekalpha4L2}. We define the time $T_1 \in (0,T]$ by 
	\begin{equation}\label{DefT1New}
		T_1 \coloneqq \min\{T,b_0,c^{-2}\}.
	\end{equation}
	Moreover, we set $L\coloneqq \lceil \frac{2T}{T_1} \rceil \in \N$ and define the final error constant $C \coloneqq 2(2c)^{L+1}$ and the maximum step size $\tau_0>0$ by 
	\begin{equation}\label{Deftau0New}
		\tau_0 \coloneqq \min\{T_1,C^{-2}\}.
	\end{equation}
	For a step size $\tau \in (0,\tau_0]$, we define the indices $N\coloneqq \lfloor T/\tau \rfloor \in \N$, $N_1\coloneqq \lfloor T_1/\tau \rfloor \in \{1,\dots,N\}$, and $N_m \coloneqq mN_1$ for all $m \in \N_0$. These definitions imply that $N \le N_L$. In addition, we define the quotient $\ell \coloneqq \lfloor N/N_1\rfloor \in \{1,\dots,L\}$. This gives the decomposition
	\[[0,t_N] = \bigcup_{m=0}^{\ell-1}[t_{N_m},t_{N_{m+1}}] \cup [t_{N_\ell},t_N] \eqqcolon \bigcup_{m=0}^{\ell}J_m, \]
	where each interval $J_m$ has length less or equal $T_1$. To measure the error in each of these intervals, we set
	$\mathrm{Err}_m \coloneqq \|E_n\|_{\ell^\infty_\tau(J_m,H^1\times L^2)}$ for $m \in \{-1,\dots,\ell\}$, where $J_{-1} \coloneqq \{0\}$.
	We aim to show the recursion formula
	\begin{equation}\label{ErrRekursionNew}
		\mathrm{Err}_{m} \le 2c(\mathrm{Err}_{m-1} + \tau^\frac12).
	\end{equation}
	Note that once \eqref{ErrRekursionNew} is proved for all indices in $\{0,\dots,m\}$, one can derive the absolute bound
	\begin{equation}\label{ErrAbsNew}
		\mathrm{Err}_m \le 2c\tau^\frac12\sum_{k=0}^m (2c)^k = 2c\tau^\frac12\frac{(2c)^{m+1}-1}{2c-1} \le 2(2c)^{L+1}\tau^\frac12 = C\tau^\frac12\le 1,
	\end{equation}
	using $\mathrm{Err}_{-1} = 0$ and the step size restriction $\tau \le \tau_0$ from \eqref{Deftau0New}. 
	
	Let us now fix an index $m \in \{0,\dots,\ell\}$. If $m>0$ we assume that the inequality \eqref{ErrRekursionNew} holds for all indices in $\{0,\dots,m-1\}$. From \eqref{ErrAbsNew} we get that
	$\mathrm{Err}_{m-1}\le 1$,
	and thus 
	\[ \|U_{N_m}\|_{H^1 \times L^2} \le \|U(t_{N_m})\|_{H^1 \times L^2} + \|E_{N_m}\|_{H^1 \times L^2} \le M+\mathrm{Err}_{m-1} \le M+1 = R. \]
	Estimate \eqref{eq:Rekalpha4H1} and the definition \eqref{DefT1New} of $T_1$ then imply
	\begin{equation*}
		\mathrm{Err}_m \le c\|E_{N_m}\|_{H^1 \times L^2} + c\tau^\frac12 + cT_1^{\frac12}\mathrm{Err}_m \le c\mathrm{Err}_{m-1} + c\tau^\frac12 + \frac12 \mathrm{Err}_m.
	\end{equation*}
Hence, the recursion \eqref{ErrRekursionNew} and the bound \eqref{ErrAbsNew} are true for all $m \in \{0,\dots,\ell\}$. This finishes the proof of the bound in the energy norm. Similarly, starting from \eqref{eq:Rekalpha4L2} we obtain the recursion formula
\[ \|E_n\|_{\ell^\infty_\tau(J_m,L^2 \times H^{-1})} \le 2c(\|E_n\|_{\ell^\infty_\tau(J_{m-1},L^2 \times H^{-1})} + \tau^\frac32) \]
for all $m \in \{0,\dots,\ell\}$, which yields the estimate
\[ \|E_n\|_{\ell^\infty_\tau([0,T],L^2 \times H^{-1})} \le C\tau^\frac32 \]
as in \eqref{ErrAbsNew}.
\end{proof}

\section{Full discretization}\label{SecFullyDiscr}	

\subsection{Proof of the fully discrete error bound}

The proof of Theorem \ref{ThmStrangFully} is very similar to that of Theorem \ref{ThmStrang}, such that we only highlight the differences. These mainly come from the introduction of the trigonometric interpolation operator $\mathcal I_K$ in \eqref{StrangFully}.

	\begin{Definition}\label{DefTrigInt}
		Let $N \ge 1$ and $f \in C(\T^3)$. We define the trigonometric interpolation $I_N f$ as the trigonometric polynomial
		\[(I_Nf)(x) \coloneqq (2\pi)^{-\frac 32} \sum_{|k|_\infty \le N}\tilde f_{k,N} e^{\iu k \cdot x}, \qquad x \in \T^3,\]
		where the coefficients $\tilde f_{k,N}$ are given by the discrete Fourier transform
		\[\tilde f_{k,N} \coloneqq  (2\pi)^\frac32  (2N+1)^{-3}  \sum_{|j|_\infty \le N} f\Big(\frac {2\pi j}{2N+1}  \Big) e^{-\iu \frac {2\pi j} {2N+1}  \cdot k}. \]
		We moreover set $\mathcal I_N \coloneqq \operatorname{diag}(I_N,I_N)$. 
	\end{Definition}

We need the following well-known generalization of Bernstein's inequality to the $L^q$ setting, see, e.g., inequality (5.2) in \cite{Guo}.

\begin{Lemma}\label{LemBernsteinLq}
	The estimate
	\[ \|\pi_K f\|_{W^{1,q}} \lesssim K \|\pi_K f\|_{L^q} \]
	holds for all $q \in [1,\infty]$, $f \in \mathcal D'(\T^3)$, and $K \ge 1$.
\end{Lemma}

We further need an estimate for the trigonometric interpolation error in $L^q$-based Sobolev spaces. For a proof, see Corollary 3 of \cite{Hristov}, Theorem 1 of \cite{AlexandrovPopov}, and Lemma 3 of \cite{PopovHristov}.

\begin{Lemma}\label{LemTrigIntError}
	Let $q \in (1,\infty)$. We then have the inequality
	\[ \|(I-I_K)f\|_{L^q} \lesssim_q \sum_{m=1}^3 K^{-m} \|f\|_{W^{m,q}} \]
	for all $f \in W^{3,q}$.
\end{Lemma}

These two results can be combined to the following estimates, which are used below with $c = \alpha$.\footnote{This is not possible if one has a general nonlinearity as outlined in Remark \ref{RemGeneralAlpha}. In that case, the proof of Theorem \ref{ThmStrangFully} becomes slightly more complex. In particular, it would be necessary to write out the $W^{3,q}$ norm appearing in Lemma \ref{LemTrigIntError} using product and chain rules, then apply nonlinear product estimates, and only afterwards use Lemma \ref{LemBernsteinLq}.}

\begin{Lemma}\label{LemTrigIntStab}
	Let $q \in (1,\infty)$ and $c \ge 1$. Then the estimates
	\begin{align*}
	\|(I-I_K) \pi_{cK}f\|_{L^q} &\lesssim_{q,c} K^{-1}\|\pi_{cK}f\|_{W^{1,q}}, \\	
	\|I_K \pi_{cK}f\|_{L^q} &\lesssim_{q,c} \|\pi_{cK}f\|_{L^q},
	\end{align*}
	hold for all $f \in \mathcal D'(\T^3)$ and $K \ge 1$.
\end{Lemma}

\begin{proof}[Proof of Theorem \ref{ThmStrangFully}]
	We use the decomposition 
	\[ \|U(t_n) - U_n^K\|_{L^2 \times H^{-1}} \le \|(I-\Pi_K)U(t_n)\|_{L^2 \times H^{-1}} + \|\Pi_KU(t_n)-U_n^K\|_{L^2 \times H^{-1}}. \]
	Thanks to Lemma \ref{LemPiNew},
	\[ \|(I-\Pi_K)U(t_n)\|_{L^2 \times H^{-1}} \lesssim_M K^{-1}.  \]
	It thus remains to estimate $E_n^K \coloneqq \Pi_KU(t_n)-U_n^K$. We define $N \coloneqq \min\{K,\tau^{-1}\}$. Analogously to Proposition \ref{PropFehlerRekStrangNew}, we decompose
	\begin{equation}\label{eq:FullyDiscErrRek}
		E_n^K = e^{n\tau A }E_0^K + \Pi_K \tilde B(t_n) + \Pi_K \tilde D_n +  \tilde Q_n + \Pi_K\tilde P_n 
	\end{equation} 
	where the terms $\tilde B$ and $\tilde D_n$ are defined in the same way as $B$ and $D_n$ from \eqref{DefTermeStrangNew}, but with $\Pi_N$ instead of $\Pi_{\tau^{-1}}$.
	The other terms are defined by
	\begin{align*}
	\tilde Q_n &\coloneqq \tau \sum_{k=0}^n {c_{k,n}} e^{(n-k)\tau A} \mathcal I_K [G(\Pi_NU(t_k))-G(\Pi_NU_k^K)], \\
	\tilde P_n &\coloneqq  \tau\sum_{k=0}^n{c_{k,n}}e^{(n-k)\tau A} (I-\mathcal I_K)G(\Pi_NU(t_k)).
	\end{align*}
	The term $\tilde B$ is bounded using Lemma \ref{LemBPI} with $N^{-1}$ instead of $\tau$.
	For the term $\tilde D_n$, we get the same estimates as for $D_n$ from Lemma \ref{LemLokFStrangNew}, since the additional $\pi_K$ inside the nonlinearity does not affect the argument. 
	For $\tilde Q_n$, it is also possible to get the same estimates as for $Q_n$ from Lemmas \ref{LemStabStrang3} and \ref{LemStabStrang4}. Here one uses the frequency localization $g(\pi_Nu) = \pi_{\alpha K}g(\pi_Nu)$ and the second inequality from Lemma \ref{LemTrigIntStab} to get rid of the interpolation operator $I_K$. We still need to deal with the term $\tilde P_n$, which contains the interpolation error. For the sake of brevity, we only give the details for $\alpha=4$, since the easier case $\alpha=3$ is treated similarly. We estimate
	\begin{align*}
		\|\tilde P_n\|_{\ell^\infty_\tau([0,T],H^1 \times L^2)} &\lesssim_T \|(I-I_K)g(\pi_Nu(t_n))\|_{\ell^1_\tau([0,T],L^2)} \\
		&\lesssim K^{-1}\|g(\pi_Nu(t_n))\|_{\ell^1_\tau([0,T],H^1)} \\
		&\lesssim K^{-1} \|g'(\pi_Nu(t_n))\|_{\ell^1_\tau([0,T],L^6)}\||\pi_N\nabla u(t_n)|\|_{\ell^\infty_\tau([0,T],L^3)} \\
		&\lesssim K^{-\frac12}\|\pi_Nu(t_n)\|^3_{\ell^3_\tau([0,T],L^{18})} \||\pi_N\nabla u(t_n)|\|_{\ell^\infty_\tau([0,T],L^2)} \\
		 &\lesssim_{M,T} K^{-\frac12}
	\end{align*}
	using Lemmas \ref{LemTrigIntStab} and \ref{LemBernstein} and Proposition \ref{PropStrichUNew}. Similarly, we obtain
	\begin{align*}
		\|\tilde P_n\|_{\ell^\infty_\tau([0,T],L^2 \times H^{-1})} &\lesssim_T \|(I-I_K)g(\pi_Nu(t_n))\|_{\ell^1_\tau([0,T],L^\frac65)} \\
		&\lesssim K^{-1}\|g(\pi_Nu(t_n))\|_{\ell^1_\tau([0,T],W^{1,\frac65})} \\
		&\lesssim K^{-1} \|g'(\pi_Nu(t_n))\|_{\ell^1_\tau([0,T],L^3)}\||\pi_N\nabla u(t_n)|\|_{\ell^\infty_\tau([0,T],L^2)} \\
		&\lesssim K^{-1}\|\pi_Nu(t_n)\|^3_{\ell^3_\tau([0,T],L^{9})} \lesssim_{M,T} K^{-1}.
	\end{align*}
	Here it is important to use the interpolation error estimate from Lemma \ref{LemTrigIntStab} with $q = 6/5$. If we stuck to $L^2$-based estimates, we could for the $L^2 \times H^{-1}$ norm only reach a sub-optimal estimate of order $K^{-1/2}$ (the same as above for the energy norm), since optimal error bounds for trigonometric interpolation in negative Sobolev spaces are not available. Now that the estimates for all term appearing in \eqref{eq:FullyDiscErrRek} are given, we can finish the proof as in the semi-discrete case above.
\end{proof}

\subsection{Numerical experiment}\label{SecNum}

For the numerical tests we need initial data $(u^0,v^0)$ which lie in $H^1 \times L^2$ but do not have higher regularity. The standard approach to obtain initial in a Sobolev space $H^s(\T^3)$ is to take Fourier coefficients of the form 
\begin{equation}\label{eq:randomcoeff}
(1+|k|^2)^{-\frac12(\frac32+s+\eps)}r_k, \qquad k \in \Z^3,
\end{equation}
for some numbers $r \in \ell^\infty(\Z^3)$ and small $\eps>0$. Most commonly, one uses $r_k$ uniformly distributed in $[-1,1]+\iu[-1,1]$. This approach is well suited to precisely obtain the desired differentiability of order $s$. However, it is known that such random initial data does not only belong to $H^s$, but also to all $L^q$-based Sobolev spaces $H^{s,q}$ for $1 \le q<\infty$ with probability one. This can be exploited to obtain an improved local wellposedness theory for the nonlinear wave equation \eqref{NLW2} with random initial data compared to the deterministic setting, cf.\ \cite{BurqTzvetkov}. Since our error bounds are purely deterministic and heavily use $L^q$-based inequalities such as Sobolev and Strichartz estimates, it is crucial to numerically work with initial data which do not have higher integrability than predicted by Sobolev embedding. The following lemma shows that this can be achieved by simply taking $r = \mathbbm1$ in \eqref{eq:randomcoeff}.

\begin{Lemma}
	Let $s \in \R$. We define a distribution $f \in \mathcal D'(\T^d)$ by its Fourier coefficients
	\begin{equation}\label{eq:detcoeff}
	\hat f_k \coloneqq (1+|k|^2)^{-\frac12(\frac d2+s)}, \qquad k \in \Z^d.
	\end{equation}
	Then, for all $\eps >0$, the following assertions hold.
	\begin{enumerate}
		\item $f \in H^{s-\eps}(\T^d)$, but $f \notin H^s(\T^d)$.
		\item If $-d/2 \le s < d/2$, then $f \notin L^{\frac{2d}{d-2s}+\eps}(\T^d)$.
	\end{enumerate}	
\end{Lemma}
\begin{proof}
	a) We have
	\begin{align*}
		\|f\|^2_{H^{s-\eps}} &= \sum_{k\in\Z^d}(1+|k|^2)^{s-\eps}|\hat f_k|^2 = \sum_{k\in\Z^d}(1+|k|^2)^{-\frac d2 - \eps}\\
		&\lesssim  \Big(1+\int_{\R^d\setminus B(0,1)}|x|^{-d-2\eps} \dd x\Big) < \infty,
	\end{align*}
	but $\|f\|^2_{H^s} \eqsim 1+\int_{\R^d\setminus B(0,1)}|x|^{-d} \dd x = \infty$. \smallskip
	
	b) For $N \in \N$, we consider the truncated Fourier series
	\[\pi_N f(x) = (2\pi)^{-\frac d2} \sum_{|k|_{\infty}\le N} \hat f_k e^{\iu k \cdot x}, \qquad x \in \T^d. \]
	Let $a>0$ such that $\cos(z) \ge 1/2$ for all $|z|\le a$. For all $|x|_1\le a/N$, we can thus infer that
	\begin{align*}
		\pi_N f(x) &\eqsim \sum_{|k|_{\infty}\le N} (1+|k|^2)^{-\frac12(\frac d2+s)} e^{\iu k \cdot x} = \sum_{|k|_{\infty}\le N} (1+|k|^2)^{-\frac12(\frac d2+s)} \cos( k \cdot x) \\
		&\gtrsim \sum_{1\le|k|_{\infty}\le N} |k|^{-\frac d2-s}. 
	\end{align*}
	Hence, for $q \in (1,\infty)$,
	\begin{align*}
		\|\pi_N f\|_{L^q} &\gtrsim N^{-\frac d q}  \sum_{1\le|k|_{\infty}\le N} |k|^{-\frac d2-s} \eqsim N^{-\frac d q} \int_{1\le|x|\le N} |x|^{-\frac d2-s} \dd x \\
		&\eqsim N^{-\frac d q} \int_1^N \rho^{d-1-\frac d2-s} \dd \rho = \frac{2}{d-2s}N^{-\frac d q + \frac d 2 -s},
	\end{align*}
	which is unbounded as $N \to \infty$ if $d/q < d/2 -s$. By Theorem 4.1.8 in \cite{GrafakosClassical}, $f \in L^q(\T^d)$ would imply that $\pi_N f \to f$ in $L^q$. Thus, if $d/q < d/2 -s$, $f$ cannot belong to $L^q(\T^d)$.  
\end{proof} 

We illustrate our error bounds by a numerical experiment for the nonlinear wave equation \eqref{NLW2} with $\mu = 1$ and powers $\alpha \in \{3,4,5\}$. We focus on the error of the time integration, {since we have proven optimal spatial error bounds, as noted before Theorem \ref{ThmStrangFully}}. The initial data $(u^0,v^0)$ are defined using \eqref{eq:detcoeff} with $d=3$ and $s = 1+\eps$ or $s=\eps$, for a very small $\eps>0$. We use a scaling such that $\|u^0\|_{H^1} = \|v^0\|_{L^2} \approx 3$. We apply the scheme \eqref{StrangFully} with spatial discretization parameters $K \in {\{2^5,2^6,2^7\}}$. For the implementation we identify $\T^3 = [0,1]^3$ such that the spatial resolution (distance of the collocation points) is $h = (2K+1)^{-1}$. We compare the errors in the $\ell^\infty_{\tau}([0,1/4],L^2 \times H^{-1})$ and $\ell^\infty_{\tau}([0,1/4],H^1 \times L^2)$ norms for various step sizes $\tau$, where the reference solution is computed using \eqref{StrangFully} with the same $K$ and $\tau_{\mathrm{ref}} = 2^{-12}$. In the plots only the temporal error is visible since the reference solution has the same spatial accuracy.
Our Python code to reproduce the results is available at \url{https://doi.org/10.35097/6c5w0sn5g44m69rx}.

For the cubic equation with $\alpha=3$, in Figure \ref{fig:a3} we numerically observe temporal convergence rates of order $2$ in the $L^2 \times H^{-1}$ norm and order $1$ in the $H^1 \times L^2$ norm, uniformly in the spatial discretization parameter $K$. These observations are in accordance with Theorem \ref{ThmStrang}. If $\tau$ is small compared to the spatial resolution, the error is of second order even in $H^1 \times L^2$, however with deteriorating error constant as $K \to \infty$. This behavior was already observed in the one-dimensional case in \cite{Gauckler}.

\begin{figure}[h]
	\centering
	\begin{subfigure}{0.5\textwidth}
		\begin{tikzpicture}[scale=0.78]
			
			\definecolor{darkgray176}{RGB}{176,176,176}
			\definecolor{darkgreen}{RGB}{0,100,0}
			\definecolor{lightgray204}{RGB}{204,204,204}
			\definecolor{limegreen}{RGB}{50,205,50}
			
			\begin{axis}[
				legend cell align={left},
				legend style={
					fill opacity=0.8,
					draw opacity=1,
					text opacity=1,
					at={(0.03,0.97)},
					anchor=north west,
					draw=lightgray204
				},
				log basis x={10},
				log basis y={10},
				tick align=outside,
				tick pos=left,
				title={$L^2 \times H^{-1}$ error},
				x grid style={darkgray176},
				xlabel={$\tau$},
				xmin=0.00158643046163327, xmax=0.153893051668115,
				xmode=log,
				xtick style={color=black},
				y grid style={darkgray176},
				ymin=1.51657353041009e-06, ymax=0.0236207370497629,
				ymode=log,
				ytick style={color=black}
				]
				\addplot [semithick, darkgreen, mark=square, mark size=3, mark options={solid}]
				table {%
					0.125	0.0100354
					0.100342	0.00704424
					0.0805664	0.00452271
					0.0649414	0.00316879
					0.052002	0.00209714
					0.041748	0.00142125
					0.0336914	0.000951007
					0.0270996	0.000629245
					0.0217285	0.000414261
					0.017334	0.00027559
					0.013916	0.000182219
					0.0112305	0.000120525
					0.0090332	7.90742e-05
					0.00732422	5.22915e-05
					0.00585938	3.32295e-05
					0.00463867	2.00814e-05
					0.00366211	1.17819e-05
					0.00292969	7.44283e-06
					0.00244141	4.90953e-06
					0.00195312	2.90349e-06
				};
				\addlegendentry{$K=2^7$}
				\addplot [semithick, red, mark=o, mark size=3, mark options={solid}]
				table {%
					0.125	0.0100296
					0.100342	0.0070385
					0.0805664	0.00451594
					0.0649414	0.00316217
					0.052002	0.00209077
					0.041748	0.00141433
					0.0336914	0.000943814
					0.0270996	0.0006223
					0.0217285	0.000407138
					0.017334	0.000267998
					0.013916	0.000174401
					0.0112305	0.000112385
					0.0090332	7.03443e-05
					0.00732422	4.43858e-05
					0.00585938	2.81467e-05
					0.00463867	1.63619e-05
					0.00366211	9.46296e-06
					0.00292969	5.93638e-06
					0.00244141	4.08656e-06
					0.00195312	2.58999e-06
				};
				\addlegendentry{$K=2^6$}
				\addplot [semithick, blue, mark=triangle, mark size=3, mark options={solid}]
				table {%
					0.125	0.0100084
					0.100342	0.00701777
					0.0805664	0.00449114
					0.0649414	0.00313814
					0.052002	0.00206792
					0.041748	0.00138932
					0.0336914	0.000917698
					0.0270996	0.000597009
					0.0217285	0.000380738
					0.017334	0.000239556
					0.013916	0.000150611
					0.0112305	9.59207e-05
					0.0090332	5.71307e-05
					0.00732422	3.488e-05
					0.00585938	2.18362e-05
					0.00463867	1.3562e-05
					0.00366211	8.40233e-06
					0.00292969	5.35056e-06
					0.00244141	3.69975e-06
					0.00195312	2.35195e-06
				};
				\addlegendentry{$K=2^5$}
				\addplot [semithick, black, dashed]
				table {%
					0.125 0.0152310009851225
					0.100341796875 0.00981459013903765
					0.08056640625 0.00632727053558291
					0.06494140625 0.00411104090005238
					0.052001953125 0.00263601411321268
					0.041748046875 0.00169895057604205
					0.03369140625 0.00110648797134656
					0.027099609375 0.000715870525885372
					0.021728515625 0.000460223231518386
					0.017333984375 0.000292890457023631
					0.013916015625 0.000188772286226895
					0.01123046875 0.000122943107927396
					0.009033203125 7.95411695428191e-05
					0.00732421875 5.22914920296108e-05
					0.005859375 3.34665548989509e-05
					0.004638671875 2.0974698469655e-05
					0.003662109375 1.30728730074027e-05
					0.0029296875 8.36663872473773e-06
					0.00244140625 5.81016578106787e-06
					0.001953125 3.71850609988343e-06
				};
				\addlegendentry{$\mathcal O(\tau^2)$}
			\end{axis}
		\end{tikzpicture}
		
	\end{subfigure}%
	\begin{subfigure}{0.5\textwidth}

		\begin{tikzpicture}[scale=0.78]
			
			\definecolor{darkgray176}{RGB}{176,176,176}
			\definecolor{darkgreen}{RGB}{0,100,0}
			\definecolor{lightgray204}{RGB}{204,204,204}
			\definecolor{limegreen}{RGB}{50,205,50}
			
			\begin{axis}[
				legend cell align={left},
				legend style={
					fill opacity=0.8,
					draw opacity=1,
					text opacity=1,
					at={(0.03,0.97)},
					anchor=north west,
					draw=lightgray204
				},
				log basis x={10},
				log basis y={10},
				tick align=outside,
				tick pos=left,
				title={$H^1 \times L^2$ error},
				x grid style={darkgray176},
				xlabel={$\tau$},
				xmin=0.00158643046163327, xmax=0.153893051668115,
				xmode=log,
				xtick style={color=black},
				y grid style={darkgray176},
				ymin=0.000150460376849438, ymax=9.43078406102876,
				ymode=log,
				ytick style={color=black}
				]
				\addplot [semithick, darkgreen, mark=x, mark size=3, mark options={solid}]
				table {%
					0.125	0.160354
					0.100342	0.129799
					0.0805664	0.104786
					0.0649414	0.0845373
					0.052002	0.0680435
					0.041748	0.0547015
					0.0336914	0.0442067
					0.0270996	0.0355579
					0.0217285	0.0284727
					0.017334	0.0226292
					0.013916	0.0180787
					0.0112305	0.0145029
					0.0090332	0.0114726
					0.00732422	0.00920954
					0.00585938	0.00710204
					0.00463867	0.0054408
					0.00366211	0.00426997
					0.00292969	0.00301607
					0.00244141	0.00217749
					0.00195312	0.00130712
				};
				\addlegendentry{$K=2^7$, no filter}
				\addplot [semithick, darkgreen, mark=square, mark size=3, mark options={solid}]
				table {%
					0.125	0.156274
					0.100342	0.14983
					0.0805664	0.122049
					0.0649414	0.103713
					0.052002	0.0898679
					0.041748	0.0761462
					0.0336914	0.0641657
					0.0270996	0.0529447
					0.0217285	0.0436979
					0.017334	0.0368049
					0.013916	0.0303049
					0.0112305	0.0244323
					0.0090332	0.0194527
					0.00732422	0.0152437
					0.00585938	0.0111947
					0.00463867	0.0071816
					0.00366211	0.00426997
					0.00292969	0.00301607
					0.00244141	0.00217749
					0.00195312	0.00130712
				};
				\addlegendentry{$K=2^7$}
				\addplot [semithick, red, mark=o, mark size=3, mark options={solid}]
				table {%
					0.125	0.155222
					0.100342	0.148664
					0.0805664	0.120635
					0.0649414	0.102177
					0.052002	0.0880383
					0.041748	0.0740752
					0.0336914	0.0617481
					0.0270996	0.050117
					0.0217285	0.0403479
					0.017334	0.0328722
					0.013916	0.0255685
					0.0112305	0.0187358
					0.0090332	0.0123496
					0.00732422	0.00834351
					0.00585938	0.0059738
					0.00463867	0.00389861
					0.00366211	0.00222511
					0.00292969	0.00121307
					0.00244141	0.00075427
					0.00195312	0.000521195
				};
				\addlegendentry{$K=2^6$}
				\addplot [semithick, blue, mark=triangle, mark size=3, mark options={solid}]
				table {%
					0.125	0.152445
					0.100342	0.145547
					0.0805664	0.116861
					0.0649414	0.0981261
					0.052002	0.0831551
					0.041748	0.0685419
					0.0336914	0.0552011
					0.0270996	0.0423084
					0.0217285	0.0309349
					0.017334	0.0213543
					0.013916	0.0150951
					0.0112305	0.0110023
					0.0090332	0.00729578
					0.00732422	0.00443409
					0.00585938	0.00242806
					0.00463867	0.00138837
					0.00366211	0.000910648
					0.00292969	0.000553013
					0.00244141	0.000393267
					0.00195312	0.000248584
				};
				\addlegendentry{$K=2^5$}
				\addplot [semithick, black, densely dotted]
				table {%
					0.125 0.272212471326603
					0.100341796875 0.218514308037566
					0.08056640625 0.175449444409725
					0.06494140625 0.141422885493899
					0.052001953125 0.113244641391731
					0.041748046875 0.0909147121032209
					0.03369140625 0.0733697676622485
					0.027099609375 0.0590148131196346
					0.021728515625 0.0473181834923197
					0.017333984375 0.0377482137972438
					0.013916015625 0.030304904034407
					0.01123046875 0.0244565892207495
					0.009033203125 0.0196716043732115
					0.00732421875 0.0159499494917931
					0.005859375 0.0127599595934345
					0.004638671875 0.0101016346781357
					0.003662109375 0.00797497474589657
					0.0029296875 0.00637997979671726
					0.00244140625 0.00531664983059772
					0.001953125 0.00425331986447817
				};
				\addplot [semithick, black, dashed]
				table {%
					0.125 5.70816652903605
					0.100341796875 3.67824248600501
					0.08056640625 2.37128957753001
					0.06494140625 1.54070675250425
					0.052001953125 0.987906674407335
					0.041748046875 0.636720647718594
					0.03369140625 0.414681714549875
					0.027099609375 0.268288878647816
					0.021728515625 0.172479198747614
					0.017333984375 0.109767408267482
					0.013916015625 0.0707467386353993
					0.01123046875 0.0460757460610973
					0.009033203125 0.0298098754053129
					0.00732421875 0.0195974345250414
					0.005859375 0.0125423580960265
					0.004638671875 0.00786074873726659
					0.003662109375 0.00489935863126034
					0.0029296875 0.00313558952400662
					0.00244140625 0.0021774927250046
					0.001953125 0.00139359534400294
				};
			\end{axis}
			
			\begin{axis}[
				hide axis,
				legend style={
					fill opacity=0.8,
					draw opacity=1,
					text opacity=1,
					at={(0.97,0.03)},
					anchor=south east,
					draw=lightgray204
				},    
				]
				\addplot[forget plot] coordinates {(0,0)};
				\addlegendimage{semithick, black, densely dotted, line legend}
				\addlegendentry{$\mathcal O(\tau)$}
				\addlegendimage{semithick, black, dashed}
				\addlegendentry{$\mathcal O(\tau^2)$}
			\end{axis}
		\end{tikzpicture}
	\end{subfigure}%
	\caption{\small Errors for $\alpha=3$}
	\label{fig:a3}
\end{figure}
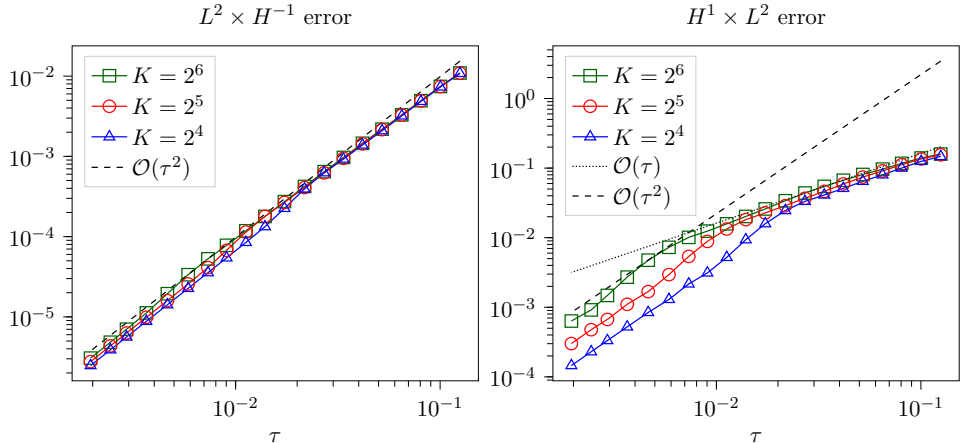

In the case $\alpha=4$, we observe in Figure \ref{fig:a4} that the convergence rates which are uniform in $K$ have reduced to $3/2$ for the $L^2 \times H^{-1}$ norm and order $1/2$ for the energy norm, again in accordance with Theorem \ref{ThmStrang}.

\begin{figure}[h]
	\centering
	\begin{subfigure}{0.5\textwidth}
		\begin{tikzpicture}[scale=0.78]
			
			\definecolor{darkgray176}{RGB}{176,176,176}
			\definecolor{darkgreen}{RGB}{0,100,0}
			\definecolor{lightgray204}{RGB}{204,204,204}
			\definecolor{limegreen}{RGB}{50,205,50}
			
			\begin{axis}[
				legend cell align={left},
				legend style={
					fill opacity=0.8,
					draw opacity=1,
					text opacity=1,
					at={(0.03,0.97)},
					anchor=north west,
					draw=lightgray204
				},
				log basis x={10},
				log basis y={10},
				tick align=outside,
				tick pos=left,
				title={$L^2 \times H^{-1}$ error},
				x grid style={darkgray176},
				xlabel={$\tau$},
				xmin=0.00158643046163327, xmax=0.153893051668115,
				xmode=log,
				xtick style={color=black},
				y grid style={darkgray176},
				ymin=1.15248074259808e-05, ymax=0.378190904578514,
				ymode=log,
				ytick style={color=black}
				]
				\addplot [semithick, darkgreen, mark=square, mark size=3, mark options={solid}]
				table {%
					0.125	0.032496
					0.100342	0.0258705
					0.0805664	0.0183385
					0.0649414	0.0138416
					0.052002	0.0100674
					0.041748	0.00741645
					0.0336914	0.00542878
					0.0270996	0.00394043
					0.0217285	0.00284764
					0.017334	0.00206459
					0.013916	0.00149388
					0.0112305	0.00108242
					0.0090332	0.000778763
					0.00732422	0.000564372
					0.00585938	0.000398326
					0.00463867	0.000272025
					0.00366211	0.00018182
					0.00292969	0.00012877
					0.00244141	8.9928e-05
					0.00195312	5.18176e-05
				};
				\addlegendentry{$K=2^7$}
				\addplot [semithick, red, mark=o, mark size=3, mark options={solid}]
				table {%
					0.125	0.0324456
					0.100342	0.0258148
					0.0805664	0.0182799
					0.0649414	0.0137789
					0.052002	0.0100016
					0.041748	0.00734789
					0.0336914	0.00535848
					0.0270996	0.00386923
					0.0217285	0.00277639
					0.017334	0.00199429
					0.013916	0.00142545
					0.0112305	0.00101707
					0.0090332	0.000710171
					0.00732422	0.000498081
					0.00585938	0.000354776
					0.00463867	0.00022256
					0.00366211	0.00012171
					0.00292969	6.3272e-05
					0.00244141	4.1796e-05
					0.00195312	2.81395e-05
				};
				\addlegendentry{$K=2^6$}
				\addplot [semithick, blue, mark=triangle, mark size=3, mark options={solid}]
				table {%
					0.125	0.0322787
					0.100342	0.0256358
					0.0805664	0.0180966
					0.0649414	0.0135892
					0.052002	0.00980902
					0.041748	0.00715406
					0.0336914	0.0051677
					0.0270996	0.00368349
					0.0217285	0.00259797
					0.017334	0.00180412
					0.013916	0.00125608
					0.0112305	0.000899166
					0.0090332	0.000576812
					0.00732422	0.000338247
					0.00585938	0.000176891
					0.00463867	0.000103807
					0.00366211	6.72405e-05
					0.00292969	4.15072e-05
					0.00244141	2.90547e-05
					0.00195312	1.84889e-05
				};
				\addlegendentry{$K=2^5$}
				\addplot [semithick, black, densely dotted]
				table {%
					0.125 0.040216753027528
					0.100341796875 0.0289244290809709
					0.08056640625 0.0208100427545312
					0.06494140625 0.0150599842477302
					0.052001953125 0.0107912428910262
					0.041748046875 0.00776240375122915
					0.03369140625 0.005627567233091
					0.027099609375 0.00405963072563398
					0.021728515625 0.00291465444148231
					0.017333984375 0.00207677566267488
					0.013916015625 0.00149387529844357
					0.01123046875 0.00108302581896927
					0.009033203125 0.000781276297418417
					0.00732421875 0.000570405815603587
					0.005859375 0.000408149177104266
					0.004638671875 0.000287496435230709
					0.003662109375 0.00020166891012077
					0.0029296875 0.000144302525433068
					0.00244140625 0.000109774650395353
					0.001953125 7.85483457568907e-05
				};
				\addlegendentry{$\mathcal O(\tau^{1.5})$}
				\addplot [semithick, black, dashed]
				table {%
					0.125 0.235740844160861
					0.100341796875 0.151907269044864
					0.08056640625 0.0979315869488441
					0.06494140625 0.0636294523980938
					0.052001953125 0.04079943221563
					0.041748046875 0.026295845123702
					0.03369140625 0.0171258874366739
					0.027099609375 0.0110800283085097
					0.021728515625 0.00712319651259683
					0.017333984375 0.00453327024618111
					0.013916015625 0.00292176056930023
					0.01123046875 0.00190287638185266
					0.009033203125 0.00123111425650108
					0.00732421875 0.00080935195825491
					0.005859375 0.000517985253283143
					0.004638671875 0.000324640063255581
					0.003662109375 0.000202337989563728
					0.0029296875 0.000129496313320786
					0.00244140625 8.99279953616567e-05
					0.001953125 5.75539170314603e-05
				};
				\addlegendentry{$\mathcal O(\tau^2)$}
			\end{axis}
			
		\end{tikzpicture}
		
	\end{subfigure}%
	\begin{subfigure}{0.5\textwidth}
		\begin{tikzpicture}[scale=0.78]
			
			\definecolor{darkgray176}{RGB}{176,176,176}
			\definecolor{darkgreen}{RGB}{0,100,0}
			\definecolor{lightgray204}{RGB}{204,204,204}
			\definecolor{limegreen}{RGB}{50,205,50}
			
			\begin{axis}[
				legend cell align={left},
				legend style={
					fill opacity=0.8,
					draw opacity=1,
					text opacity=1,
					at={(0.03,0.97)},
					anchor=north west,
					draw=lightgray204
				},
				log basis x={10},
				log basis y={10},
				tick align=outside,
				tick pos=left,
				title={$H^1 \times L^2$ error},
				x grid style={darkgray176},
				xlabel={$\tau$},
				xmin=0.00158643046163327, xmax=0.153893051668115,
				xmode=log,
				xtick style={color=black},
				y grid style={darkgray176},
				ymin=0.00157073136040144, ymax=238.503070039141,
				ymode=log,
				ytick style={color=black}
				]
				\addplot [semithick, darkgreen, mark=x, mark size=3, mark options={solid}]
				table {%
					0.125	1.70198
					0.100342	1.34544
					0.0805664	1.06478
					0.0649414	0.848002
					0.052002	0.673352
					0.041748	0.539405
					0.0336914	0.437903
					0.0270996	0.358002
					0.0217285	0.295377
					0.017334	0.245638
					0.013916	0.207928
					0.0112305	0.178314
					0.0090332	0.152935
					0.00732422	0.133436
					0.00585938	0.114389
					0.00463867	0.0970462
					0.00366211	0.084291
					0.00292969	0.0697113
					0.00244141	0.0528976
					0.00195312	0.0326967	
				};
				\addlegendentry{$K=2^7$, no filter}
				\addplot [semithick, darkgreen, mark=square, mark size=3, mark options={solid}]
				table {%
					0.125	0.677935
					0.100342	0.676246
					0.0805664	0.604792
					0.0649414	0.565968
					0.052002	0.522246
					0.041748	0.481323
					0.0336914	0.437385
					0.0270996	0.392347
					0.0217285	0.350116
					0.017334	0.315673
					0.013916	0.279079
					0.0112305	0.242929
					0.0090332	0.208546
					0.00732422	0.17631
					0.00585938	0.143116
					0.00463867	0.107926
					0.00366211	0.084291
					0.00292969	0.0697113
					0.00244141	0.0528976
					0.00195312	0.0326967
				};
				\addlegendentry{$K=2^7$}
				\addplot [semithick, red, mark=o, mark size=3, mark options={solid}]
				table {%
					0.125	0.653578
					0.100342	0.651496
					0.0805664	0.577275
					0.0649414	0.536983
					0.052002	0.490778
					0.041748	0.447545
					0.0336914	0.400596
					0.0270996	0.352007
					0.0217285	0.30562
					0.017334	0.26684
					0.013916	0.224435
					0.0112305	0.181775
					0.0090332	0.138765
					0.00732422	0.11468
					0.00585938	0.0956192
					0.00463867	0.0665353
					0.00366211	0.0390277
					0.00292969	0.0213553
					0.00244141	0.0125306
					0.00195312	0.00822187
				};
				\addlegendentry{$K=2^6$}
				\addplot [semithick, blue, mark=triangle, mark size=3, mark options={solid}]
				table {%
					0.125	0.612777
					0.100342	0.609726
					0.0805664	0.530664
					0.0649414	0.487947
					0.052002	0.437042
					0.041748	0.389519
					0.0336914	0.336533
					0.0270996	0.280268
					0.0217285	0.225248
					0.017334	0.175932
					0.013916	0.149105
					0.0112305	0.123214
					0.0090332	0.0867859
					0.00732422	0.0542258
					0.00585938	0.0299195
					0.00463867	0.0151557
					0.00366211	0.0100798
					0.00292969	0.00627087
					0.00244141	0.00433455
					0.00195312	0.00270159
				};
				\addlegendentry{$K=2^5$}
				\addplot [semithick, black, densely dotted]
				table {%
					0.125 0.847701131770267
					0.100341796875 0.75950160058192
					0.08056640625 0.680557387174171
					0.06494140625 0.611010397732117
					0.052001953125 0.546761069338424
					0.041748046875 0.48989819315787
					0.03369140625 0.440095903111597
					0.027099609375 0.394701949897081
					0.021728515625 0.353429491464331
					0.017333984375 0.315672650564947
					0.013916015625 0.282842853695208
					0.01123046875 0.254089488130732
					0.009033203125 0.22788127702275
					0.00732421875 0.205195772869214
					0.005859375 0.183532678732468
					0.004638671875 0.16329939771929
					0.003662109375 0.145095322466636
					0.0029296875 0.12977720170106
					0.00244140625 0.118469834702614
					0.001953125 0.105962641471283
				};
				\addplot [semithick, black, dashed]
				table {%
					0.125 138.667968403465
					0.100341796875 89.3552089335696
					0.08056640625 57.6055212369435
					0.06494140625 37.4282484907362
					0.052001953125 23.9991266574737
					0.041748046875 15.4677965701512
					0.03369140625 10.0738250361465
					0.027099609375 6.51751723746906
					0.021728515625 4.19002142991579
					0.017333984375 2.66656962860819
					0.013916015625 1.71864406335013
					0.01123046875 1.11931389290516
					0.009033203125 0.724168581941007
					0.00732421875 0.476078687908624
					0.005859375 0.304690360261519
					0.004638671875 0.190960451483348
					0.003662109375 0.119019671977156
					0.0029296875 0.0761725900653798
					0.00244140625 0.0528976319898471
					0.001953125 0.0338544844735021
				};
			\end{axis}
			\begin{axis}[
				hide axis,
				legend style={
					fill opacity=0.8,
					draw opacity=1,
					text opacity=1,
					at={(0.97,0.03)},
					anchor=south east,
					draw=lightgray204
				},    
				]
				\addplot[forget plot] coordinates {(0,0)};
				\addlegendimage{semithick, black, densely dotted, line legend}
				\addlegendentry{$\mathcal O(\tau^{0.5})$}
				\addlegendimage{semithick, black, dashed}
				\addlegendentry{$\mathcal O(\tau^2)$}
			\end{axis}
		\end{tikzpicture}
		
	\end{subfigure}%
	\caption{\small Errors for $\alpha=4$}
	\label{fig:a4}
\end{figure}
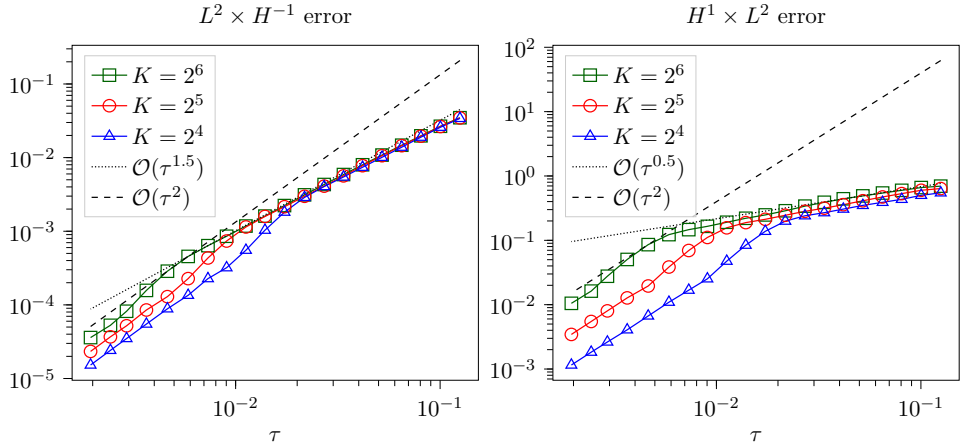

We also did the experiment for the critical power $\alpha=5$, see Figure \ref{fig:a5}. Here it turned out that we get temporal convergence of order $1$ in the $L^2 \times H^{-1}$ norm, uniformly in $K$. Moreover, we cannot observe a clear convergence order for the error in the $H^1 \times L^2$ norm if $\tau$ is not small compared to the spatial resolution. 
This behavior fits to \eqref{eq:critbound}.

\begin{figure}[h]
	\centering
	\begin{subfigure}{0.5\textwidth}
		\begin{tikzpicture}[scale=0.78]
			
			\definecolor{darkgray176}{RGB}{176,176,176}
			\definecolor{darkgreen}{RGB}{0,100,0}
			\definecolor{lightgray204}{RGB}{204,204,204}
			\definecolor{limegreen}{RGB}{50,205,50}
			
			\begin{axis}[
				legend cell align={left},
				legend style={
					fill opacity=0.8,
					draw opacity=1,
					text opacity=1,
					at={(0.03,0.97)},
					anchor=north west,
					draw=lightgray204
				},
				log basis x={10},
				log basis y={10},
				tick align=outside,
				tick pos=left,
				title={$L^2 \times H^{-1}$ error},
				x grid style={darkgray176},
				xlabel={$\tau$},
				xmin=0.00158643046163327, xmax=0.153893051668115,
				xmode=log,
				xtick style={color=black},
				y grid style={darkgray176},
				ymin=6.56718098930426e-05, ymax=5.69295494728749,
				ymode=log,
				ytick style={color=black}
				]
				\addplot [semithick, darkgreen, mark=square, mark size=3, mark options={solid}]
				table {%
					0.125	0.0776124
					0.100342	0.0682285
					0.0805664	0.051953
					0.0649414	0.0429985
					0.052002	0.0341873
					0.041748	0.0276781
					0.0336914	0.0221431
					0.0270996	0.0175943
					0.0217285	0.013964
					0.017334	0.0112216
					0.013916	0.00894569
					0.0112305	0.00711641
					0.0090332	0.00563234
					0.00732422	0.00447123
					0.00585938	0.00346238
					0.00463867	0.00259432
					0.00366211	0.00189362
					0.00292969	0.00152193
					0.00244141	0.00121963
					0.00195312	0.000828944
				};
				\addlegendentry{$K=2^7$}
				\addplot [semithick, red, mark=o, mark size=3, mark options={solid}]
				table {%
					0.125	0.0774925
					0.100342	0.0680768
					0.0805664	0.0517696
					0.0649414	0.0427774
					0.052002	0.0339262
					0.041748	0.0273733
					0.0336914	0.0217957
					0.0270996	0.0172034
					0.0217285	0.0135292
					0.017334	0.0107384
					0.013916	0.00841114
					0.0112305	0.00652524
					0.0090332	0.00494012
					0.00732422	0.00373827
					0.00585938	0.00300667
					0.00463867	0.00223411
					0.00366211	0.00143335
					0.00292969	0.000835542
					0.00244141	0.00050615
					0.00195312	0.000262082
				};
				\addlegendentry{$K=2^6$}
				\addplot [semithick, blue, mark=triangle, mark size=3, mark options={solid}]
				table {%
					0.125	0.0770256
					0.100342	0.0675214
					0.0805664	0.0511392
					0.0649414	0.0420558
					0.052002	0.033115
					0.041748	0.0264717
					0.0336914	0.0208125
					0.0270996	0.0161196
					0.0217285	0.0123308
					0.017334	0.00932082
					0.013916	0.00701915
					0.0112305	0.00566911
					0.0090332	0.00423007
					0.00732422	0.00283269
					0.00585938	0.00165108
					0.00463867	0.000863188
					0.00366211	0.000427126
					0.00292969	0.000231005
					0.00244141	0.000184076
					0.00195312	0.000110111
				};
				\addlegendentry{$K=2^5$}
				\addplot [semithick, black, densely dotted]
				table {%
					0.125 0.084995123046895
					0.100341796875 0.0682285069770974
					0.08056640625 0.0547820129013191
					0.06494140625 0.0441576225204572
					0.052001953125 0.035359299236306
					0.041748046875 0.0283870430488653
					0.03369140625 0.0229088417587334
					0.027099609375 0.0184266770668073
					0.021728515625 0.0147745428733861
					0.017333984375 0.0117864330787687
					0.013916015625 0.00946234768295511
					0.01123046875 0.00763628058624448
					0.009033203125 0.00614222568893578
					0.00732421875 0.00498018299102901
					0.005859375 0.00398414639282321
					0.004638671875 0.00315411589431837
					0.003662109375 0.0024900914955145
					0.0029296875 0.0019920731964116
					0.00244140625 0.00166006099700967
					0.001953125 0.00132804879760774
				};
				\addlegendentry{$\mathcal O(\tau)$}
				\addplot [semithick, black, dashed]
				table {%
					0.125 3.39535412020337
					0.100341796875 2.18790669761228
					0.08056640625 1.41049981571253
					0.06494140625 0.91644926501888
					0.052001953125 0.58763054305842
					0.041748046875 0.378736686053721
					0.03369140625 0.246662612400638
					0.027099609375 0.159584648571112
					0.021728515625 0.102594757027172
					0.017333984375 0.0652922825620467
					0.013916015625 0.042081854005969
					0.01123046875 0.0274069569334042
					0.009033203125 0.0177316276190125
					0.00732421875 0.0116570232703515
					0.005859375 0.00746049489302497
					0.004638671875 0.00467576155621878
					0.003662109375 0.00291425581758788
					0.0029296875 0.00186512372325624
					0.00244140625 0.00129522480781684
					0.001953125 0.000828943877002775
				};
				\addlegendentry{$\mathcal O(\tau^2)$}
			\end{axis}
			
		\end{tikzpicture}
		
	\end{subfigure}%
	\begin{subfigure}{0.5\textwidth}
		\begin{tikzpicture}[scale=0.78]
			
			\definecolor{darkgray176}{RGB}{176,176,176}
			\definecolor{darkgreen}{RGB}{0,100,0}
			\definecolor{lightgray204}{RGB}{204,204,204}
			\definecolor{limegreen}{RGB}{50,205,50}
			
			\begin{axis}[
				legend cell align={left},
				legend style={
					fill opacity=0.8,
					draw opacity=1,
					text opacity=1,
					at={(0.03,0.97)},
					anchor=north west,
					draw=lightgray204
				},
				log basis x={10},
				log basis y={10},
				tick align=outside,
				tick pos=left,
				title={$H^1 \times L^2$ error},
				x grid style={darkgray176},
				xlabel={$\tau$},
				xmin=0.00158643046163327, xmax=0.153893051668115,
				xmode=log,
				xtick style={color=black},
				y grid style={darkgray176},
				ymin=0.0108469044809689, ymax=3792.88443435493,
				ymode=log,
				ytick style={color=black}
				]
				\addplot [semithick, darkgreen, mark=x, mark size=3, mark options={solid}]
				table {%
					0.125	28.9733
					0.100342	22.8869
					0.0805664	18.0183
					0.0649414	14.1858
					0.052002	11.031
					0.041748	8.55201
					0.0336914	6.62883
					0.0270996	5.08496
					0.0217285	3.86386
					0.017334	2.91298
					0.013916	2.21711
					0.0112305	1.71624
					0.0090332	1.36669
					0.00732422	1.14567
					0.00585938	1.00081
					0.00463867	0.902729
					0.00366211	0.838991
					0.00292969	0.792036
					0.00244141	0.701449
					0.00195312	0.518352
				};
				\addlegendentry{$K=2^7$, no filter}
				\addplot [semithick, darkgreen, mark=square, mark size=3, mark options={solid}]
				table {%
					0.125	2.91341
					0.100342	2.9271
					0.0805664	2.78847
					0.0649414	2.7313
					0.052002	2.63286
					0.041748	2.54619
					0.0336914	2.43167
					0.0270996	2.30061
					0.0217285	2.1624
					0.017334	2.04609
					0.013916	1.90119
					0.0112305	1.73817
					0.0090332	1.56375
					0.00732422	1.38069
					0.00585938	1.18959
					0.00463867	0.979484
					0.00366211	0.838991
					0.00292969	0.792036
					0.00244141	0.701449
					0.00195312	0.518352
				};
				\addlegendentry{$K=2^7$}
				\addplot [semithick, red, mark=o, mark size=3, mark options={solid}]
				table {%
					0.125	2.55316
					0.100342	2.56825
					0.0805664	2.41005
					0.0649414	2.34456
					0.052002	2.23026
					0.041748	2.12903
					0.0336914	1.99387
					0.0270996	1.83703
					0.0217285	1.66787
					0.017334	1.52069
					0.013916	1.33274
					0.0112305	1.14173
					0.0090332	0.942724
					0.00732422	0.828284
					0.00585938	0.782708
					0.00463867	0.657579
					0.00366211	0.457697
					0.00292969	0.282602
					0.00244141	0.177876
					0.00195312	0.095603
				};
				\addlegendentry{$K=2^6$}
				\addplot [semithick, blue, mark=triangle, mark size=3, mark options={solid}]
				table {%
					0.125	2.15617
					0.100342	2.1726
					0.0805664	1.98694
					0.0649414	1.90956
					0.052002	1.77134
					0.041748	1.64715
					0.0336914	1.47661
					0.0270996	1.27759
					0.0217285	1.08045
					0.017334	0.908521
					0.013916	0.806053
					0.0112305	0.757098
					0.0090332	0.629914
					0.00732422	0.453536
					0.00585938	0.280595
					0.00463867	0.154215
					0.00366211	0.0792887
					0.00292969	0.0410526
					0.00244141	0.03037
					0.00195312	0.0193772
				};
				\addlegendentry{$K=2^5$}
				\addplot [semithick, black, dashed]
				table {%
					0.125 2123.17060306392
					0.100341796875 1368.13393188537
					0.08056640625 882.008661932603
					0.06494140625 573.070751916467
					0.052001953125 367.455013620021
					0.041748046875 236.830259720581
					0.03369140625 154.242175921437
					0.027099609375 99.7908973707218
					0.021728515625 64.1541837572833
					0.017333984375 40.828334846669
					0.013916015625 26.3144733022868
					0.01123046875 17.1380195468263
					0.009033203125 11.0878774856358
					0.00732421875 7.28932778456697
					0.005859375 4.66516978212286
					0.004638671875 2.92383036692075
					0.003662109375 1.82233194614174
					0.0029296875 1.16629244553072
					0.00244140625 0.80992530939633
					0.001953125 0.518352198013651
				};
				\addlegendentry{$\mathcal O(\tau^2)$}
			\end{axis}

		\end{tikzpicture}
		
	\end{subfigure}%
	\caption{\small Errors for $\alpha=5$}
	\label{fig:a5}
\end{figure}
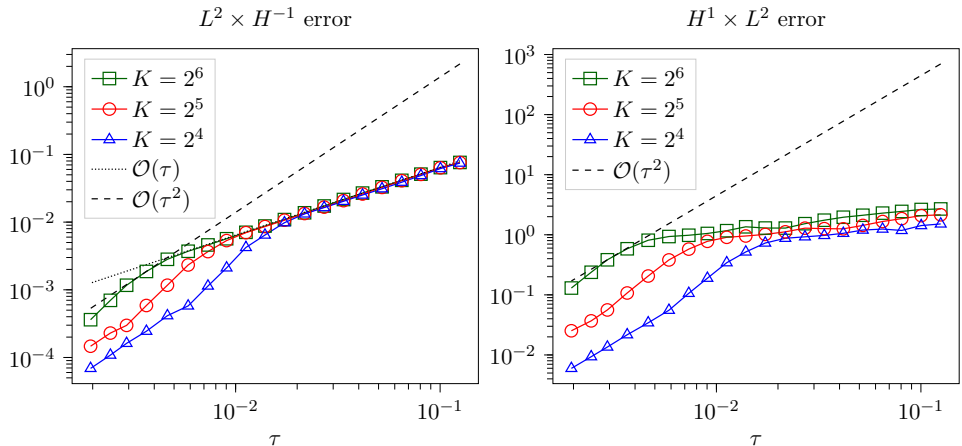

{Finally, we investigated the impact of the filter $\Pi_{\tau^{-1}}$ in \eqref{StrangFully} on the errors. If the errors are measured in $L^2 \times H^{-1}$, the removal of the filter did not result in any significant difference (not shown in the plots). In the right-hand panels of Figures \ref{fig:a3}, \ref{fig:a4} and \ref{fig:a5}, we included for comparison the $H^1 \times L^2$ errors of \eqref{StrangFully} in the case $K=2^7$ without the filter $\Pi_{\tau^{-1}}$. We observe that for $\alpha=3$, the omission of the filter leads to qualitatively similar results with a slightly better error constant. 
On the other hand, in the cases $\alpha \in \{4,5\}$, the introduction of the filter reduces the error at least for the greater values of $\tau$. 
Note however that the filter also increases the computational efficiency of the scheme {(regardless of $\alpha \in \{3,4,5\}$)}, as already pointed out in the lower-dimensional cases in \cite{Gauckler,DiscSol}.

\section*{Statements and Declarations}

\subsection*{Competing interests}

The author has no relevant financial or non-financial interests to disclose.
	
\printbibliography

\end{document}